\documentclass[12pt]{amsart}

\usepackage{amssymb}
\usepackage[all]{xy}

\newtheorem{theorem}{Theorem}[section]
\newtheorem{lemma}[theorem]{Lemma}
\newtheorem{prop}[theorem]{Proposition}
\newtheorem{cor}[theorem]{Corollary}
\theoremstyle{definition}
\newtheorem{definition}[theorem]{Definition}

\theoremstyle{remark}
\newtheorem{remark}[theorem]{Remark}
\numberwithin{equation}{section}

\newcommand\B{\mathbb{B}}
\newcommand\C{\mathbb{C}}

\newcommand\Z{\mathbb{Z}}

\newcommand\R{\mathbb{R}}

\newcommand\T{\mathbb{T}}
\newcommand\cA{\mathcal{A}}
\newcommand\cB{\mathcal{B}}
\newcommand\cC{\mathcal{C}}
\newcommand\cD{\mathcal{D}}
\newcommand\cH{\mathcal{H}}

\newcommand\cO{\mathcal{O}}

\newcommand\cU{\mathcal{U}}
\newcommand\cW{\mathcal{W}}
\newcommand\cZ{\mathcal{Z}}
\newcommand\fA{\mathfrak{A}}
\newcommand\fS{\mathfrak{S}}
\newcommand\fc{\mathfrak{c}}
\newcommand\Aut{\operatorname{Aut}}
\newcommand\Out{\operatorname{Out}}

\newcommand\End{\operatorname{End}}
\newcommand\id{\mathrm{id}}
\newcommand\Ad{\mathrm{Ad}}
\newcommand\Hom{\operatorname{Hom}}

\newcommand{\Ind}{\operatorname{Ind}}

\newcommand\inpr[2]{\langle{#1,#2}\rangle}

\newcommand{\talpha}{\tilde{\alpha}}
\newcommand{\trho}{\tilde{\rho}}
\newcommand{\hrho}{\hat{\rho}}
\newcommand{\tsigma}{\tilde{\sigma}}
\newcommand{\hbeta}{\hat{\beta}}
\newcommand{\brho}{\overline{\rho}}
\newcommand{\biota}{\overline{\iota}}

\title[The classification of $3^n$ subfactors]
{The classification of $3^n$ subfactors and related fusion categories} 

\author{Masaki Izumi}
\address{Department of Mathematics\\ Graduate School of Science\\
Kyoto University\\ Sakyo-ku, Kyoto 606-8502\\ Japan}
\email{izumi@math.kyoto-u.ac.jp}


\thanks{Supported in part by JSPS KAKENHI Grant Number JP15H03623}

\begin{document} 

\dedicatory{In memory of Uffe Haagerup and John Roberts} 

\begin{abstract} We investigate a (potentially infinite) series of subfactors, called $3^n$ subfactors, including 
$A_4$, $A_7$, and the Haagerup subfactor as the first three members corresponding to $n=1,2,3$. 
Generalizing our previous work for odd $n$, we further develop a Cuntz algebra method to construct $3^n$ subfactors 
and show that the classification of the $3^n$ subfactors and related fusion categories is reduced 
to explicit polynomial equations under a mild assumption, which automatically holds for odd $n$. 
In particular, our method with $n=4$ gives a uniform construction of 4 finite depth subfactors, up to dual, 
without intermediate subfactors of index $3+\sqrt{5}$. 
It also provides a key step for a new construction of the Asaeda-Haagerup subfactor due to Grossman, 
Snyder, and the author. 
\end{abstract}

\maketitle

\tableofcontents
\section{Introduction} 
The theory of subfactors, introduced by Vaughan Jones \cite{J83}, is a rich source of a new kind of symmetries, 
sometimes called quantum symmetries (see \cite{EK98}). 
One of the ways to describe such a symmetry encoded in a subfactor $N\subset M$ of finite index is to 
consider the category of bimodules generated by two basic bimodules ${}_NM_M$ and ${}_MM_N$ via tensor product over $M$ and $N$, where 
${}_NM_M$ and ${}_MM_N$ are $M$ regarded as $N-M$ and $M-N$ bimodules respectively. 
The $M-M$ bimodules and $N-N$ bimodules arising in this way form rigid tensor categories, called the even part 
of the subfactor, while the $M-N$ bimodules and $N-M$ bimodules form bimodule categories over them, 
giving categorical Morita equivalence of them. 
All information about the original subfactor can be stated in terms of these categories. 
For example, the subfactor is of finite depth if and only if there are only finitely many 
isomorphism classes of irreducible bimodules as above, that is, the two tensor categories are fusion categories. 
The two principal graphs of the subfactor are nothing but the induction-reduction graphs with respect to the 
basic bimodules between the $M-M$ bimodules and $M-N$ bimodules for the one principal graph, 
and the $M-N$ bimodules and $N-N$ bimodules for the other. 
Moreover, this process of passing from the subfactor to these categories can be reversed; 
namely, the original subfactor is recovered from the category of $N-N$ bimodules with an algebra object 
${}_NM_N\cong {}_N(M\otimes_MM)_N$, called a $Q$-system \cite{L94}. 
This is the approach we adopt to construct and classify a specific class of subfactors in this paper. 

The categories arising from a subfactor always carry a special analytic structure. 
Namely, they are C$^*$-categories, which were introduced by Ghez-Lima-Roberts \cite{GLR} as categorical counterparts of C$^*$-algebras. 
Therefore to classify a specific class of finite depth subfactors in our approach, the first task is to 
classify a specific class of C$^*$-fusion categories. 
For this purpose, it is not necessarily convenient to realize C$^*$-fusion categories as bimodules 
over von Neumann factors, and we take an alternative (but of course, mathematically equivalent) approach 
heavily influenced by algebraic quantum field theory. 
Since his epoch-making joint work \cite{DHR71} with Sergio Doplicher and Rudolf Haag, John E. Roberts devoted himself  
to studying categorical aspects of algebraic quantum field theory and related mathematical structure 
(\cite{R76}, \cite{R77}, \cite{GLR}, \cite{LR97}, just to name a few). 
In their work, C$^*$-tensor categories naturally appear as categories consisting of endomorphisms 
of relevant operator algebras. 
In fact, it is known that every C$^*$-fusion category is uniquely embedded in the category of endomorphisms 
of the hyperfinite type III$_1$ factor (\cite{P95}, \cite{HY00}, see also \cite[Section 2]{I15} 
for a precise statement). 
Our Cuntz algebra method relies on this fact. 

Now we turn our attention to specific subfactors appearing in an ongoing project of the classification of small 
index subfactors. 
The fusion categories arising from the Jones subfactors of index less than 4 can be described 
by the quantum $SU(2)$ at roots of unity. 
Likewise, general subfactors with index less than or equal to 4 are related to either 
the quantum $SU(2)$ or ordinary groups in a little more, but not too much, complicated way.  
Uffe Haagerup \cite{H93} was the first to systematically explore subfactors with index beyond 4, 
and in the early 90s he came up with countably many candidate principal graph pairs of potential subfactors with index 
between 4 and $3+\sqrt{3}$. 
Moreover, he showed that the first candidate indeed arises from a subfactor, now called the Haagerup subfactor 
(see \cite{AH99} for the proof).  
The Haagerup subfactor is the first subfactor that is not directly related to either an ordinary group or 
a quantum group, and whether its Drinfeld center is related to a quantum group (conformal field theory) or not 
is an interesting open problem. 
At the time of writing, the classification of finite depth subfactors is completed up to index  5+1/4 
(see \cite{JMS14}, \cite{IMPPS15}, \cite{AMP15}, and references therein), and it turns out that three subfactors actually exist 
among Haagerup's list, namely the Haagerup subfactor, Asaeda-Haagerup subfactors constructed in \cite{AH99}, 
and the extended Haagerup subfactor constructed in \cite{BPMS12}. 
The original construction of the Haagerup subfactor in \cite{AH99} used computation of connections,  
a special type of 6j-symbols. 
Later Peters \cite{Pe10} and the author \cite{I01} gave different constructions, based on the Jones Planar algebra 
and the Cuntz algebra respectively. 
This work is a natural continuation of \cite{I01}, which used the fact that one of the principal graphs of 
the Haagerup subfactor has a $\Z_3$-symmetry. 
It is natural to generalize $\Z_3$ to arbitrary finite groups. 

Let us recall our construction in \cite{I01} briefly. 
The principal graphs of the Haagerup subfactor are as in Figure \ref{Haagerup}, where $\iota$ means ${}_MM_N$, 
or in the endomorphism language, the inclusion map $\iota :N \hookrightarrow M$. 
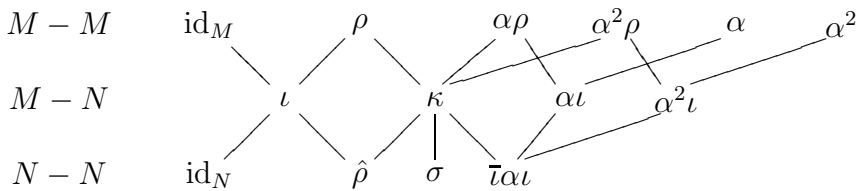
\begin{figure}[h]
${}$
\begin{xy} 
(10,10)*{M-M},(30,10)*{\id_M},(50,10)*{\rho},(70,10)*{\alpha\rho},(84,10)*{\alpha^2\rho},(100,10)*{\alpha},(114,10)*{\alpha^2},
(10,0)*{M-N},(40,0)*{\iota},(60,0)*{\kappa},(78,0)*{\alpha\iota},(92,0)*{\alpha^2\iota},
(10,-10)*{N-N},(30,-10)*{\id_N},(50,-10,)*{\hat{\rho}},(60,-10)*{\sigma},(70,-10)*{\biota\alpha\iota},
{(33,7) \ar @{-} (38,2)},{(42,2) \ar @{-} (48,8)},{(52,8) \ar @{-} (58,2)},{(61,2) \ar @{-} (68,8)},
{(62,2) \ar @{-} (82,8)},{(72,8) \ar @{-} (76,2)},{(86,8) \ar @{-} (90,2)},
{(80,2) \ar @{-} (98,8)},{(94,2) \ar @{-} (112,8)},
{(32,-8) \ar @{-} (38,-2)},{(42,-2) \ar @{-} (48,-8)},{(52,-8) \ar @{-} (58,-2)},{(62,-2) \ar @{-} (68,-8)},
{(60,-2) \ar @{-} (60,-8)},{(71,-8) \ar @{-} (76,-2)},{(72,-8) \ar @{-} (90,-2)},
\end{xy}
\caption{The principal graphs of the Haagerup subfactor}
\label{Haagerup}
\end{figure}
We call the upper principal graph $3^3$ because it has a central vertex $\kappa$ having three legs of length 3 out of it. 
The endomorphisms of $M$ corresponding to the three end points are automorphisms, and they form a group 
$\{\id_M,\alpha,\alpha^2\}$ of order three, and the $M-M$ part has the following fusion rules. 
$$[\alpha^3]=[\id_M],\quad [\alpha][\rho]=[\rho][\alpha^2],$$
$$[\rho^2]=[\id_M]+[\rho]+[\alpha\rho]+[\alpha^2\rho].$$
This fusion category is often referred to as the Haagerup category. 
In \cite{I01}, we explicitly constructed a fusion category with these fusion rules consisting of the endomorphisms 
of the Cuntz algebra $\cO_4$, and with a $Q$-system giving the Haagerup subfactor.  
Moreover, using the explicit formula, we were able to determine the structure of its Drinfeld center. 

In the above construction, we can generalize $\Z_3$ to an arbitrary group. 
We define $3^n$ graph as a graph with a unique central vertex having exactly $n$ legs of length $3$ out of it. 
\begin{figure}[h]
${}$
\begin{xy}
(0,0)*{\bullet},(5,0)*{\bullet},(10,0)*{\bullet},(15,0)*{\bullet},(20,0)*{\bullet},(25,0)*{\bullet},(30,0)*{\bullet},
(15,15)*{\bullet},(15,10)*{\bullet},(15,5)*{\bullet},(15,-5)*{\bullet},(15,-10)*{\bullet},(15,-15)*{\bullet},
{(0,0) \ar @{-} (30,0)},{(15,15) \ar @{-} (15,-15)},
\end{xy}
\caption{$3^4$ graph}
\end{figure}
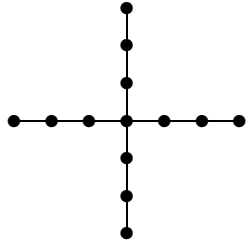
We say that a subfactor $N\subset M$ is $3^n$ if the principal graph between the $M-M$ bimodules and the $M-N$ 
bimodules is the $3^n$ graph. 
As in the case of the Haagerup subfactor, a group $G$ of order $n$ naturally arises from a $3^n$ subfactor 
in the automorphism group of $M$, say $\{\alpha_g\}_{g\in G}$. 
When we would like to specify this group, we can say that the subfactor is $3^G$ instead of $3^n$. 
It is easy to show that the only possible fusion rules, other than the group part, are 
$$[\alpha_g][\rho]=[\rho][\alpha_{g^\tau}]$$
$$[\rho^2]=[\id_M]+\sum_{g\in G}[\alpha_g\rho],$$
where $\tau$ is a group automorphism of $G$ of order two. 
In the case where $G$ is an abelian odd group and $\tau=-1$, in \cite{I01} we obtained polynomial equations 
whose solutions give $3^G$ subfactors via Cuntz algebra endomorphisms, and solved the equations 
for $G=\Z_3$ and $G=\Z_5$. 
Evans-Gannon \cite{EG11} showed that there exist solutions for the polynomial equations 
when $G$ is an odd cyclic group of small order. 
We show that these solutions actually classify $3^n$ subfactor with odd groups $n$ 
($G$ being abelian and $\tau=-1$ automatically hold for odd $G$, see Theorem \ref{odd}). 
One of the main purposes of this paper is to obtain the polynomial equations in the even case too. 
It turns out that we can also classify $3^n$ subfactors with even $n$ by the polynomial equations 
with an extra assumption of $G$ being abelian and $\tau=-1$. 
So far there is no known $3^n$ subfactor not satisfying this condition. 

The classification list \cite{AMP15} of small index subfactors shows that there are relatively few 
finite depth subfactors. 
However, $3+\sqrt{5}$ is an exceptionally rich index value, and there are exactly 4 finite depth subfactors, up to dual, 
without non-trivial intermediate subfactors (see \cite{MP15-1}, \cite{MP15-2}). 
Our method gives uniform construction of them. 
Namely, the four subfactors are the unique $3^{\Z_2\times \Z_2}$ subfactor and its equivariantization by $\Z_3$, and 
the unique $3^{\Z_4}$ subfactor and its de-equivariantization by $\Z_2$. 

Recently, Pinhas Grossman, Noah Snyder, and the author \cite{GIS15} gave a new construction of 
the Asaeda-Haagerup subfactor based on the study \cite{GS12-2} of the Brauer-Picard groupoid 
of the corresponding fusion categories. 
The new construction requires a similar fusion category to the one as above with the group $G=\Z_4$ 
but having non-trivial multiplicity in the fusion rules. 
It turns out that we can construct the desired fusion category from a $3^{\Z_4\times \Z_2}$ subfactor via 
de-equivariantization by $\Z_2$. 
Our new construction solves a lot of open problems about the Asaeda-Haagerup subfactors. 
For example, we can compute the Drinfeld center of fusion categories for the Asaeda-Haagerup subfactor 
(see \cite{GI15}).  

This paper is organized as follows. 
In Section 2, we set up an appropriate class of fusion categories for our classification purpose, 
which we call generalized Haagerup categories. 
Since the definition of the class involves subtlety of cohomological nature, 
we begin with a more general class of fusion categories, and formulate cohomological invariants for them 
mimicking the $E_2$-term of the spectral sequence for the cohomology of semidirect product groups. 
We also prepare the basics of an operator algebraic method to classify C$^*$-fusion categories.  

Using Cuntz algebras, we deduce polynomial equations for generalized Haagerup categories in Section 3, and give a 
reconstruction theorem in Section 4 (supplemented by a free product method in Appendix). 
In Section 5, we obtain a complete classification result for generalized Haagerup categories. 
There is a symmetry group $\Gamma$ acting on the gauge equivalence classes of the solutions of 
the polynomial equations, and each $\Gamma$-orbit corresponds to an equivalence class of 
generalized Haagerup categories, while the outer automorphism group of the category is given by 
the stabilizer subgroup. 

In Section 6, we discuss a necessary and sufficient condition for the existence of a $Q$-system giving rise to 
a $3^G$ subfactor, and shows that it significantly simplifies the polynomial equations. 
This condition Eq.(\ref{Q1}) was not separated from the other conditions in \cite{I01}, where only odd groups were treated.  
We also discuss a strategy to solve the polynomial equations without assuming the existence of the $Q$-system. 
In Section 7, we state our classification results for $3^G$ subfactors putting the results obtained 
in the preceding sections together. 
We also compute the dual principal graphs. 

In Section 8, we discuss several methods to obtain new fusion categories out of 
a given generalized Haagerup category, including de-equivariantization and equivariantization. 
In Section 9, we give solutions of the polynomial equations for abelian groups of small order. 
There exists a unique solution (up to equivalence in an appropriate sense) for $\Z_2\times \Z_2$, 
and it gives rise to a $3^{\Z_2\times \Z_2}$ subfactor. 
There exist two solutions for $\Z_4$, only one of which gives rise to a $3^{\Z_4}$ subfactor.

This work started with a conversation with Terry Gannon in 2010 asking whether the previous result on odd abelian 
groups can extend to more general groups, and the author is grateful to him. 
The author would like to thank Scott Morrison for his kind explanation of the use of formal codegrees, 
Victor Ostrik for providing an elementary proof of Lemma \ref{fixed-point-free}, 
and Vaughan Jones for drawing the author's attention to the spectral sequence for the cohomology of 
a semidirect product group.

\section{Preliminaries}
Our basic references are \cite{EGNO15} for fusion categories, \cite{EK98} for operator algebras and subfactors, 
and \cite{BKLR15} for the category of endomorphisms of von Neumann algebras. 
There are unfortunate discrepancies of terminology and notation in \cite{EGNO15} and \cite{BKLR15}. 
To avoid possible confusion, we use the symbol $*$ for the dual objects and the dual morphisms, and 
$\id_X$ for the identity morphism of an object $X$  only in  subsection \ref{GHC}, 
where general fusion categories are discussed. 
In the rest of the paper where only C$^*$-fusion categories are discussed, the symbol $*$ is reserved for 
the adjoint operators of the bounded operators acting on Hilbert spaces. 
Instead, we use $\overline{\sigma}$ for the dual object of $\sigma$, and we also use the term ``conjugate" instead of ``dual". 
Also the symbol $\id_M$ is reserved for the identity morphism of a von Neumann algebra $M$, 
playing the role of the unit object in $\End(M)$, and instead $1_\sigma$ is used for the identity morphism of 
an object $\sigma$. 

\subsection{Generalized Haagerup categories}\label{GHC}
The main purpose of this subsection is to set up an appropriate class of fusion categories for this work. 
We start with a little more general class than we need for the classification 
of $3^G$ subfactors, and it is a subclass of the so-called quadratic categories. 

A fusion category over the complex numbers $\C$ is a rigid semisimple $\C$-linear tensor category with 
finitely many simple objects and finite dimensional morphism spaces such that the unit object $\mathbf{1}$ is simple. 
Throughout the paper, we assume that fusion categories are strict. 
For a fusion category $\cC$, we denote by $\cO(\cC)$ the set of isomorphism classes of the simple objects in $\cC$. 
For an object $X\in \cC$, we denote by $[X]$ its isomorphism class. 

Let $\cC$ be a fusion category over the complex numbers $\C$. 
The isomorphism classes of the invertible objects of $\cC$ form a finite group, which we denote by $G$. 
We choose a representative from each class $g \in G$, and denote it by the same symbol $g$. 
We always assume $e=\mathbf{1}$. 
We say that $\cC$ is a quadratic category if there exists a non-invertible simple object $\rho$ such that 
every simple object of $\cC$ is isomorphic to an object in either $G$ or $G\otimes \rho \otimes G$. 

\begin{definition}
With the above notation, we say that $\cC$ is a quadratic category with $(G,\tau,m)$, where 
$\tau$ is a group automorphism of $G$ of period two and $m$ is a natural number, if $\rho$ is self-dual with 
$$\cO(\cC)=G\sqcup \{[g\otimes \rho]\}_{g\in G},$$
and they obey the following fusion rules: 
$$[g][h]=[gh],\quad g,h\in G,$$
$$[g] [\rho]=[\rho][g^\tau],\quad g\in G$$
$$[\rho]^2=[\mathbf{1}]+ \sum_{g\in G}m [g\otimes \rho].$$
\end{definition}

The Haagerup category is a quadratic category with $(\Z_3,-1,1)$. 
In fact, for odd groups there exists great restriction for the structure of the quadratic categories 
with $(G,\tau,m)$. 

\begin{theorem}\label{odd} Let $\cC$ be a spherical quadratic category with $(G,\tau,m)$.  
If $G$ is an odd group and $m$ is an odd number, then $G$ is abelian and $g^\tau=g^{-1}$ for any $g\in G$. 
\end{theorem}

To show the theorem, we first recall the notion of formal codegrees of a fusion category $\cC$ introduced by Ostrik \cite{O09}. 
The Grothendieck ring $K(\cC)$ of $\cC$ is the free module generated by $\cO(\cC)$ with a multiplication 
given by the monoidal product in $\cC$. 
Let $(\pi,V_\pi)$ be an irreducible representation of $K(\cC)$, where $V_\pi$ is a finite dimensional 
vector space over $\C$, and $\pi:K(\cC)\to \End(V_\pi)$ is a ring homomorphism. 
Then the formal codegree $f_\pi$ for $\pi$ is defined by 
$$f_\pi=\sum_{X\in \cO(\cC)}\mathrm{Tr}(\pi(X))\pi(X^*),$$
where $\mathrm{Tr}(\pi(X))$ is the trace of $\pi(X)$. 
Since $f_\pi$ commutes with $\pi(X)$ for every $X\in \cO(\cC)$, it is a scalar. 
Ostrik \cite[Theorem 2.13]{O09} showed that if $\cC$ is spherical, there exists a simple object 
in the Drinfeld center $\cZ(\cC)$ whose dimension is $\dim \cC/f_\pi$, where $\dim \cC$ is the global dimension 
of $\cC$. 
In particular, the number $\dim \cC/f_\pi$ is necessarily a cyclotomic integer.

\begin{lemma} Let the notation be as in Theorem \ref{odd} and assume that $G$ is an odd group and $m$ is an odd number. 
Then for any non-trivial irreducible representation $\pi$ of $G$, the two representations $\pi$ and $\pi\circ \tau$ are inequivalent.  
\end{lemma}

\begin{proof} 
Assume on the contrary that there is an irreducible non-trivial representation $(\pi_\pi,V_\pi)$ of $G$ such that 
$\pi$ is equivalent to $\pi\circ \tau$.  
Then there exists an invertible element $W\in \End(V_\pi)$ satisfying $\pi(g^\tau)=W\pi(g)W^{-1}$ for any $g\in G$. 
Since $\tau$ is of order two, $W^2$ is a scalar, and we may assume that $W^2=1$ holds multiplying $W$ 
by a scalar if necessary. 
Note that since $\pi$ is non-trivial, we have 
$$\sum_{g\in G}\pi(g)=0.$$
This enables us to introduce an irreducible representation $\pi'$ of $K(\cC)$ on $V_\pi$ by setting 
$\pi'(g)=\pi(g)$ for $g\in G$ and $\pi'([\rho])=W$. 
To compute the formal codegree of $\pi'$, 
we may assume that $(\pi_\pi,V_\pi)$ is a unitary representation. 
We choose an orthonormal basis $\{e_i\}_{i=1}^{\dim \pi}$, and express $\pi(g)$ and $W$ by matrices $(\pi(g)_{ij})$ and $(W_{ij})$. 
Note that $W$ is a self-adjoint unitary now. 
The Peter-Weyl theorem implies 
\begin{align*}
\lefteqn{(f_{\pi'})_{ij}=\sum_{g\in G}\mathrm{Tr}(\pi(g))\pi(g^{-1})_{ij}+\sum_{g\in G}\mathrm{Tr}(\pi(g)W)
(W\pi(g)^{-1})_{ij} }\\
 &=\sum_{g,k}\pi(g)_{kk}\overline{\pi(g)_{ji}}+\sum_{g,k,l,r}\pi(g)_{kl}W_{lk}W_{ir}\overline{\pi(g)_{jr}} \\
 &= \frac{|G|}{\dim \pi}\sum_{k}\delta_{k,i}\delta_{k,j}+
 \frac{|G|}{\dim \pi}\sum_{k,l,r}\delta_{k,j}\delta_{l,r}W_{lk}W_{ir} =\frac{2|G|}{\dim \pi}\delta_{ij},
 \end{align*}
and $f_{\pi'}=2|G|/\dim \pi$. 

On the other hand, since we assume that $\cC$ is spherical and $G$ is odd, the only possibilities of 
the dimensions of the simple objects in $\cC$ are $\dim g=1$ for $g\in G$, and 
$\dim g\otimes \rho=\frac{m|G|\pm \sqrt{m^2|G|^2+4}}{2}$, and so 
$$\dim \cC=|G|+|G|(\dim \rho)^2=|G|(2+m|G|\dim \rho).$$
This implies 
$$\frac{\dim \cC}{f_{\pi'}}=\dim \pi+\frac{m|G|\dim \pi\dim \rho}{2}.$$
Since $m|G|\dim \pi$ is odd, this cannot be an algebraic integer, and we get contradiction.   
\end{proof}

Recall that an automorphism of a finite group $G$ is called fixed-point-free if it has no fixed point in 
$G\setminus \{e\}$.  
It is known that $G$ allows a fixed-point-free automorphism of period two $\tau$ if and only if 
$G$ is abelian and $g^\tau=g^{-1}$ for any $g\in G$  (see \cite[Exercises 10.5.1]{R93}). 
We say that $\tau\in \Aut(G)$ is fixed-point-free on the dual of $G$ 
if for any non-trivial irreducible representation $\pi$, the two representation $\pi$ and 
$\pi\circ \tau$ are inequivalent. 
When $G$ is abelian and $\tau$ is of period two, the two definitions are equivalent 
because the latter is equivalent to the condition that $\tau$ acts on the dual group $\hat{G}$ by $-1$, 
which in tern is equivalent to the condition that $\tau$ acts on $G$ by $-1$. 
The proof of Theorem \ref{odd} follows from the following lemma, which states that the two definitions 
for a period two $\tau$ are always equivalent. 

\begin{lemma}\label{fixed-point-free} Let $G$ be a finite group, and let $\tau$ be an automorphism of $G$ of period two.  
If $\tau$ is fixed-point-free on the dual of $G$, the group $G$ is abelian and $g^\tau=g^{-1}$ 
for any $g$. 
\end{lemma}

\begin{proof} 
We first claim that the restriction of $\tau$ to any characteristic subgroup $N$ of $G$ is again fixed-point-free 
on the dual of $N$. 
Indeed, assume that there exists an non-trivial irreducible representation $\sigma$ of $N$ such that $\sigma$ and 
$\sigma\circ \tau$ are equivalent. 
Then thanks to the Frobenius reciprocity, for any irreducible representation $\pi$ of $G$, the multiplicity of 
$\pi$ in the induced representation $\Ind_N^G \sigma$ is the same as that of $\pi\circ \tau$ in $\Ind_N^G \sigma$. 
Since $\Ind_N^G \sigma$ does not contain the trivial representation, this implies that 
$$\dim \Ind_N^G \sigma=|G/N|\dim\sigma$$ 
is even, which contradict the assumption that $G$ is odd. 
Thus the claim holds. 

Let 
$$\{1,\pi_1,\pi_1\circ \tau, \pi_2,\pi_2\circ \tau,\ldots, \pi_k,\pi_k\circ \tau\}$$ 
be the irreducible representations of $G$. 
Then 
$$|G|=1+2\sum_{i=1}^k\dim \pi_i^2$$
is odd, and $G$ is solvable thanks to the Feit-Thompson theorem 
(see \cite[p.148]{R93} and references therein). 
Let 
$$G=G^{(0)}\rhd G^{(1)}\rhd\cdots \rhd G^{(l)}=\{e\}$$
be the derived series of $G$. 
If $G$ is abelian, the statement holds, and so we assume that $l\geq 2$ and get contradiction. 
For this purpose, it suffices to assume $l=2$ by replacing $G$ with $G^{(l-2)}$ because $G^{(l-2)}$ 
is a characteristic subgroup of $G$. 
Thus we assume that $[G,G]\neq \{e\}$ is abelian. 
Since $[G,G]$ is an abelian characteristic subgroup, the restriction of $\tau$ acts on $[G,G]$ by $-1$. 
On the other hand, since any representation of $G/[G,G]$ is regarded as a representation of $G$, 
the automorphism of $G/[G,G]$ induced by $\tau$ is also a fixed-point-free on the dual of $G/[G,G]$, and 
hence it acts by $-1$ for $G/[G,G]$ is abelian. 
This implies that for any $g\in G$, we have $g^\tau\in g^{-1}[G,G]$ and $gg^{\tau}\in [G,G]$. 
Since $\tau$ acts on $[G,G]$ by $-1$, we have 
$(gg^\tau)^\tau=(gg^\tau)^{-1}$ and we get $(g^2)^\tau=(g^2)^{-1}$. 
Since $G$ is an odd group, this implies that we have $h^\tau=h^{-1}$ for any $h\in G$, 
which contradicts the assumption that $G$ is non-abelian. 
Thus $l<2$ and $G$ is abelian.  
\end{proof}

Victor Ostrik kindly informed the author of the following elementary proof of the above lemma without using the 
Feit-Thompson theorem. 
We would like to thank him for his courtesy. 

\begin{proof}[Second proof of Lemma \ref{fixed-point-free}] 
It suffices to show that $\tau$ is a fixed-point-free automorphism. 
Let $2a$ be the number of the conjugacy classes in $G$ that are not fixed by $\tau$, and let 
$b$ be the number of non-trivial conjugacy classes that are fixed by $\tau$. 
Our goal is to show that $b=0$. 
Let 
$$\{1,\pi_1,\pi_1\circ \tau, \pi_2,\pi_2\circ \tau,\ldots, \pi_k,\pi_k\circ \tau\}$$ 
be the irreducible representations of $G$. 
Then we have $2k=2a+b$. 

Let $\tilde{G}$ be the semidirect product group $G\rtimes_\tau \Z_2$. 
Since $\tau$ is fixed-point-free on the dual of $G$, the group $\tilde{G}$ has two 1-dimensional 
representations, and all the other irreducible representations are of the form 
$\Ind_G^{\tilde{G}} \pi_i\cong\Ind_G^{\tilde{G}} \pi_i\circ\tau$. 
Thus $\tilde{G}$ has exactly $2+k=2+a+\frac{b}{2}$ irreducible representations. 
On the other hand, the number of the conjugacy classes in $\tilde{G}$ is larger than or equal to 
$2+a+b$, and $\tilde{G}$ has at least $2+a+b$ irreducible representations. 
Therefore we get $b=0$, and $\tau$ is a fixed-point-free automorphism. 
\end{proof}

Next we introduce cohomological invariants of a quadratic category $\cC$ with $(G,\tau,m)$ 
pursuing a similarity  between $\cC$ and the semidirect product group $G\rtimes_\tau \Z_2$, 
which was first observed by Evans-Gannon \cite{EG11} in the case of $G=\Z_n$ with odd $n$, $\tau=-1$, and $m=1$. 
Note that the $E_2$-term of the Lyndon-Hochschild-Serre spectral spectral sequence for the group cohomology 
$H^*(G\rtimes_\tau \Z_2,\C^\times)$ is given by $E^{p,q}_2=H^p(\Z_2,H^q(G,\C^\times))$, where the group 
$\Z_2$ acts on $H^q(G,\C^\times)$ through $\tau$. 
We start with 
$$E_2^{0,3}=H^0(\Z_2,H^3(G,\C^\times))=H^3(G,\C^\times)^\tau,$$ 
where $H^3(G,\C^\times)^\tau$ is the set of cohomology classes in $H^3(G,\C^\times)$ fixed by $\tau$.

For $g,h\in G$, we choose an isomorphism $v_{g,h}:gh\to g\otimes h$ with $v_{g,e}=v_{e,g}=\id_g$. 
Since both $(v_{g,h}\otimes \id_k)\circ v_{gh,k}$ and $(\id_g\otimes v_{h,k})\circ v_{g,hk}$ are isomorphisms 
from $ghk$ to $g\otimes h\otimes k$,  there exists $\omega(g,h,k)\in \C^\times$ satisfying 
\begin{equation}\label{C1}
(\id_g\otimes v_{h,k})\circ v_{g,hk}=\omega(g,h,k)(v_{g,h}\otimes \id_k)\circ v_{gh,k}.
\end{equation}
Thanks to the pentagon equation, we see that $\omega=\{\omega(g,h,k)\}_{g,h,k\in G}$ form a 3-cocycle in $Z^3(G,\C^\times)$, and 
we denote by $\fc^{0,3}(\cC)$ its cohomology class $[\omega]\in H^3(G,\C^\times)$, 
which is a well-known invariant of the fusion category $\cC$, or rather the fusion subcategory generated by $G$. 

For each $g\in G$, we choose an isomorphism  $w_g:g\otimes \rho\to \rho\otimes g^\tau$. 
Then we have two isomorphisms $\id_\rho\otimes v_{g^\tau,h^\tau}$ and 
$$(w_g\otimes \id_{h^{\tau}})\circ (\id_g\otimes w_h)\circ (v_{g,h}\otimes \id_\rho)\circ w_{gh}^{-1},$$
from $\rho\otimes g^\tau h^\tau$ to $\rho\otimes g^\tau\otimes h^\tau$, 
and there exists $\xi(g,h)\in \C^\times$ satisfying 
\begin{equation}\label{C2}
(w_g\otimes \id_{h^{\tau}})\circ (\id_g\otimes w_h)\circ (v_{g,h}\otimes \id_\rho)\circ w_{gh}^{-1}
=\xi(g,h) \id_\rho\otimes v_{g^\tau,h^\tau}.
\end{equation}

\begin{lemma} With the above notation, we have
\begin{equation}\label{C3}
\omega(g,h,k)=\omega(g^\tau,h^\tau,k^\tau)\xi(h,k)\xi(gh,k)^{-1}\xi(g,hk)\xi(g,h)^{-1}.
\end{equation}
In consequence $\fc^{0,3}(\cC)\in H^3(G,\C^\times)^\tau$. 
\end{lemma}

\begin{proof} 
From Eq.(\ref{C1}), we get 
$$\id_\rho\otimes ((\id_{g^\tau}\otimes v_{h^\tau,k^\tau})\circ v_{g^\tau,h^\tau k^\tau})
=\omega(g^\tau,h^\tau,k^\tau)
\id_\rho\otimes ((v_{g^\tau,h^\tau}\otimes \id_{k^\tau})\circ v_{g^\tau h^\tau,k^\tau}).$$
Computing the both sides using Eq.(\ref{C2}), we get the statement. 
\end{proof}

Now we assume $\fc^{0,3}(\cC)=0$, and we introduce an invariant 
$$\fc^{1,2}(\cC)\in H^1(\Z_2,H^2(G,\C^\times)).$$  
Since $\fc^{0,3}(\cC)=0$, we can choose $v_{g,h}$ so that the relation 
\begin{equation}\label{C4}
(\id_g\otimes v_{h,k})\circ v_{g,hk}=(v_{g,h}\otimes \id_k)\circ v_{gh,k}.
\end{equation}
holds. 
Then Eq.(\ref{C3}) shows that $\xi=\{\xi(g,h)\}_{g,h\in G}$ form 
a 2-cocycle in $Z^2(G,\C^\times)$. 
Choosing appropriate $w_g$, we may further assume that $\xi$ is normalized, that is,  
$\xi(g,g^{-1})=1$ for any $g\in G$, and in consequence $\xi(g,h)=\xi(h^{-1},g^{-1})^{-1}$. 
Since $\rho$ is self-dual, we can produce an isomorphism $w_g^*:(g^{\tau})^{-1}\otimes \rho\to \rho\otimes g^{-1}$ 
from $w_g$ by rigidity (see \cite[2.10]{EGNO15} for the definition), where we choose $g^*=g^{-1}$ with the evaluation and coevaluation 
maps given by $v_{g^{-1},g}^{-1}$ and $v_{g,g^{-1}}$ respectively. 
More concretely, we set 
\begin{align*}
\lefteqn{w_{g}^*=(({v_{{g^\tau}^{-1},g^\tau}}^{-1}\circ (\id_g\otimes \mathrm{ev}_\rho\otimes \id_{g^\tau}))
\otimes \id_\rho\otimes \id_{g^{-1}})} \\
&\circ (\id_{{g^{\tau}}^{-1}}\otimes \id_\rho\otimes((w_g\otimes \id_\rho\otimes \id_{g^{-1}})\circ  
(\id_g\otimes \mathrm{coev}_\rho\otimes \id_{g^{-1}})\circ v_{g,g^{-1}})),
\end{align*}
Thus there exists $\eta(g)\in \C^\times$ satisfying 
\begin{equation}\label{c12}w_{(g^\tau)^{-1}}^*=\eta(g)w_g. 
\end{equation}

\begin{lemma} With the above notation, we have 
\begin{equation}\label{C5}
\xi(g^\tau,h^\tau)\xi(g,h)=\eta(gh)\eta(g)^{-1}\eta(h)^{-1}.
\end{equation}
In consequence, the 2-cocycle $\xi\in Z^2(G,\C^\times)$ gives a class in $H^1(\Z_2,H^2(G,\C^\times))$. 
\end{lemma}

\begin{proof}
Since $\xi$ is normalized, it suffices to show
$$\xi((h^\tau)^{-1},(g^\tau)^{-1})\eta(gh)=\eta(g)\eta(h)\xi(g,h).$$
Indeed, with the notation $g^+=(g^\tau)^{-1}$, we have 
\begin{align*}
\lefteqn{\eta(g)\eta(h)(w_g\otimes\id_{h\tau})\circ (\id_g\otimes w_h)
=(w_{g+}^*\otimes \id_{h^\tau})\circ (\id_g\otimes w_{h^+}^*)} \\
 &=((w_{h+}\otimes \id_{g^{-1}})\circ (\id_{h^+}\otimes w_{g^+}))^* \\
 &=\xi(h^+,g^+)((\id_\rho\otimes v_{h^{-1},g^{-1}})\circ w_{(gh)^+}\circ (v_{h^{-1},g^{-1}}^{-1}\otimes \id_\rho))^*\\
 &=\xi(h^+,g^+)(\id_\rho\otimes v_{g,h})\circ w_{(gh)^{+}}^*\circ (v_{g,h}^{-1}\otimes \id_\rho)\\
 &=\xi(h^+,g^+)\eta(gh)(\id_\rho\otimes v_{g,h})\circ w_{gh}\circ (v_{g,h}^{-1}\otimes \id_\rho)\\
 &=\xi(h^+,g^+)\eta(gh)\xi(g,h)^{-1}(w_g\otimes\id_{h\tau})\circ (\id_g\otimes w_h). 
\end{align*}
\end{proof}

We denote by $\fc^{1,2}(\cC)$ the cohomology class in $H^1(\Z_2,H^2(G,\C^\times))$ given by $\xi$,  
which does not depend on either the choice of $v_{g,h}$ or that of $w_g$. 

Although we need only $\fc^{0,3}(\cC)$ and $\fc^{1,2}(\cC)$ for our purpose, 
we can proceed further under the additional assumption that $\cC$ is pivotal. 
Assume $\fc^{1,2}(\cC)=0$. 
Then we can choose $v_{g,h}$ and $w_g$ satisfying  
\begin{equation}\label{C6}
(w_g\otimes \id_{h^{\tau}})\circ (\id_g\otimes w_h)\circ (v_{g,h}\otimes \id_\rho)\circ w_{gh}^{-1}
=\id_\rho\otimes v_{g^\tau,h^\tau}.
\end{equation}
and in consequence Eq.(\ref{C5}) implies $\eta\in \Hom(G,\C^\times)$. 
Replacing $g$ with $g^+$ in $w_{g^+}^*=\eta(g)w_g$, we get 
$w_g^*=\eta(g^+)w_{g^+}$, and $w_g^{**}=\eta(g^+)\eta(g)w_g$. which shows $\eta\in \Hom(G,\C^\times)^\tau$. 
We still have freedom to replace $v_{g,h}$ with $\zeta(g,h)v_{g,h}$ and $w_g$ with $\mu(g)w_g$ 
where $\zeta\in Z^2(G,\C^\times)$ and $\mu(g)\in \C^\times$ satisfy $\zeta(g,e)=\zeta(e,g)=\mu(e)=1$ and 
$\mu(g)\mu(h)\mu(gh)^{-1}=\zeta(g,h)^{-1}\zeta(g^\tau,h^\tau)$. 
Since 
\begin{align*}
\lefteqn{w_{g^+}^*=((v_{g,g^{-1}}^{-1}\circ (\id_\rho\otimes \mathrm{ev}_\rho\otimes \id_{g^{-1}}))
\otimes \id_\rho\otimes \id_{g^\tau})}\\
&\circ (\id_\rho\otimes \id_\rho\otimes((w_{g^+}\otimes \id_\rho\otimes \id_g)\circ  
(\id_{g^+}\otimes \mathrm{coev}_\rho\otimes \id_{g^\tau})\circ v_{(g^\tau)^{-1},g^\tau})),
\end{align*}
this amounts to replacing $\eta(g)$ with 
$$\eta(g)\zeta(g,g^{-1})^{-1}\zeta((g^\tau)^{-1},g^\tau)\mu((g^\tau)^{-1})\mu(g)^{-1}.$$
Since the cocycle relation of $\zeta$ implies $\zeta(g,g^{-1})=\zeta(g^{-1},g)$, this is equal to 
$$\eta(g)\mu((g^\tau)^{-1})\mu(g^{-1}).$$ 
Note that we can identify 
$$H^2(\Z_2,H^1(G,\C^\times))=H^2(\Z_2,\Hom(G,\C^\times))$$ 
with
$$\Hom(G,\C^\times)^\tau/\{\chi \chi^\tau\in \Hom (G,\C^\times);\; \chi\in \Hom(G,\C^\times)\}.$$
On the other hand, we have 
$H^0(\Z_2,H^2(G,\C^\times))=H^2(G,\C^\times)^\tau$.  
Thus $\eta$ determines an element in 
$$\mathrm{coker}(H^0(\Z,H^2(G,\C^\times))\to H^2(\Z_2,H^1(G,\C^\times)),$$ 
which we denote by $\fc^{2,1}(\cC)$. 

Finally, we just mention that Longo \cite{L90} already pointed out that a right analogue of an element in 
$$H^3(\Z_2,H^0(G,\C^\times ))= H^3(\Z_2,\C^\times)\cong \Z_2$$ associated with  
the object $\rho$ should be, in modern term, the Frobenius-Schur indicators $\nu_{2,1}(\rho)\in \{1,-1\}$ 
(see \cite{NS07} for the definition). 
We set $\fc^{3,0}(\cC)=\nu_{2,1}(\rho)$.

In Lemma \ref{vanishing}, we show that a quadratic category $\cC$ with $(G,\tau,m=1)$ coming from 
a $3^G$ subfactor has trivial $\fc^{0,3}(\cC)$ and $\fc^{1,2}(\cC)$. 
To simplify the statements of our main results, we introduce the following class of quadratic categories. 
Our main goal in this paper is to classify them under the C$^*$-condition, which is probably not too modest 
a goal in view of Theorem \ref{odd} and Lemma \ref{vanishing}. 

\begin{definition} A generalized Haagerup category with a finite abelian group $G$ is 
a quadratic category $\cC$ with $(G,-1,1)$ satisfying $\fc^{0,3}(\cC)=0$ and $\fc^{1,2}(\cC)=0$. 
\end{definition}

When we need a quadratic category $\cC$ satisfying all the above conditions except for $m=1$ 
(as in the case of our new construction of the Asaeda-Haagerup subfactor in \cite{GIS15}), 
we could say that a fusion category $\cC$ is a generalized Haagerup category with higher multiplicity $m$. 
We could use the adjective ``twisted" to describe $\cC$ with non-trivial 
$\fc^{0,3}(\cC)$ or $\fc^{1,2}(\cC)$ (see \cite{MPS15}), though we do not need them in this paper.  
It is known that there exist a quadratic categories with $(\Z_3,-1,1)$ having non-trivial 
$\fc^{0,3}(\cC)\in H^3(G,\C^\times)$ (see \cite{EG11}, \cite[Example 12.14]{I15}).

\subsection{The category $\End(M)$} 
In this subsection, we partly follow \cite[Section 2]{I15} for presentation. 
For a Hilbert space $\cH$, we denote by $\B(\cH)$ the set of bounded operators on $\cH$, and by $\cU(\cH)$ the set of unitaries on $\cH$. 
The identity operator of $\cH$ is denoted by $1_{\cH}$ or simply by $1$. 
For a unital C$^*$-algebra $A$, we denote by $\cU(A)$ the set of unitaries in $A$. 
The unit of $A$ is denoted by $1_A$ or simply by $1$. 

Let $M$ be a properly infinite factor. 
Then the set of unital endomorphisms $\End(M)$ forms a tensor category with the monoidal product $\rho\otimes \sigma$ of two objects 
$\rho, \sigma\in \End(M)$ given by the composition $\rho\circ \sigma$, and the morphism space from $\rho$ to $\sigma$ 
given by 
$$\Hom_{\End(M)}(\rho,\sigma)=\{T\in M;\; T\rho(x)=\sigma(x)T,\;\forall x\in M\}.$$
For simplicity, we denote $(\rho,\sigma)=\Hom_{\End(M)}(\rho,\sigma)$. 
In this tensor category, the monoidal product $T_1\otimes T_2$ of two morphisms $T_i\in (\rho_i,\sigma_i)$, $i=1,2$, are given by 
$$T_1\rho_1(T_2)=\sigma_1(T_2)T_1\in (\rho_1\circ \rho_2,\sigma_1\circ \sigma_2).$$
This is graphically expressed as
$$
\begin{xy}(0,0)*+[F]{T_1},(10,7)*+[F]{T_2} 
\ar(0,15);(0,3)^<{\rho_1}
\ar(10,15);(10,10)^<{\rho_2}
\ar(0,-3);(0,-8)^>{\sigma_1}
\ar(10,4);(10,-8)^>{\sigma_2}
\end{xy}
=
\begin{xy}(0,7)*+[F]{T_1},(10,0)*+[F]{T_2} 
\ar(0,15);(0,10)^<{\rho_1}
\ar(10,15);(10,3)^<{\rho_2}
\ar(0,4);(0,-8)^>{\sigma_1}
\ar(10,-3);(10,-8)^>{\sigma_2}
\end{xy}
.$$
By definition, two objects $\rho,\sigma$ are equivalent if and only if there exists a unitary $U\in \cU(M)$ 
satisfying $\rho=\Ad U\circ \sigma$, where $\Ad U$ is the inner automorphism of $M$ given by $\Ad U (x)=UxU^{-1}$. 
The self-morphism space $(\rho,\rho)$ is nothing but the relative commutant $M\cap \rho(M)'$, 
and when this space consists of only scalars, we say that $\rho$ is irreducible (or simple). 

The morphism space $(\rho,\sigma)$ inherits the Banach space structure from $M$, and the $*$-operation 
of $M$ sends $(\rho,\sigma)$ to $(\sigma,\rho)$, which makes $\End(M)$ a C$^*$-tensor category 
(see \cite[Section 1]{BKLR15}). 
Moreover, if $\rho$ is irreducible, the space $(\rho,\sigma)$ is a Hilbert space with an inner product 
given by $T_1^*T_2=\inpr{T_1}{T_2}1_M$ for $T_1,T_2\in (\rho,\sigma)$. 
Throughout the paper, we assume that any functor between C$^*$-fusion categories preserves  
the $*$-structure. 

For $\rho\in \End(M)$, its dimension $d(\rho)$ is defined by $[M:\rho(M)]_0^{1/2}$, 
where $[M:\rho(M)]_0$ is the minimal index of $\rho(M)$ in $M$. 
We denote by $\End_0(M)$ the set of $\rho\in \End(M)$ with finite $d(\rho)$. 
The dimension function $\End_0(M)\ni \rho\mapsto d(\rho)$ is additive with respect to 
the direct sum operation and multiplicative with respect to the monoidal product operation. 
The tensor category $\End_0(M)$ is rigid in the following sense: for any $\rho\in \End_0(M)$, 
there exist $\brho\in \End_0(M)$, called the conjugate endomorphism of $\rho$, and two isometries 
$R_\rho\in (\id,\brho\circ \rho)$, $\overline{R}_\rho\in (\id,\rho\circ \brho)$ satisfying 
$$\overline{R}_\rho^*\rho(R_\rho)=R_\rho^*\brho(\overline{R}_\rho)=\frac{1}{d(\rho)}.$$
The evaluation morphism $\mathrm{ev}_\rho$ is identified with $\sqrt{d(\rho)}\overline{R}$, 
and the coevaluation morphism $\mathrm{coev}_\rho$ is identified with $\sqrt{d(\rho)} R^*$.

If we replace $\End(M)$ with the set of unital homomorphisms between two type III factors, 
the dimension function and conjugate morphisms still make sense, 
and we use the same notation as above (see \cite{I98}, \cite{BKLR15}).  

Every C$^*$-fusion category is realized as a category of bimodules of the hyperfinite II$_1$ factor 
(see \cite{HY00}), which implies by a tensor product trick, that every C$^*$-fusion category is realized 
as a subcategory of $\End_0(M)$ for any hyperfinite type III factor $M$. 
For uniqueness, we have the following statement, which is a consequence of Popa's classification 
theorem for amenable subfactors \cite{P95}. 
Recall that a monoidal functor from a strict fusion category $\cC$ to another strict fusion category $\cD$ is a pair $(F,L)$ 
consisting of a functor $F:\cC\to \cD$ and natural isomorphisms 
$$L_{\rho,\sigma}\in \Hom_\cD(F(\rho)\otimes F(\sigma), F(\rho\otimes \sigma))$$ 
satisfying 
$$L_{\rho\otimes \sigma,\tau}\circ (L_{\rho,\sigma}\otimes I_{F(\tau)})
=L_{\rho,\sigma\otimes \tau}\circ (I_{F(\rho)}\otimes L_{\sigma,\tau})$$
for any $\rho,\sigma,\tau\in \cC$ (see \cite[Definition 2.4.1]{EGNO15}).  
We may and do assume $F(\mathbf{1}_{\cC})=\mathbf{1}_{\cD}$ and $L_{\mathbf{1}_\cC,\rho}=L_{\rho,\mathbf{1}_\cC}=I_{F(\rho)}$.   
When $\cC$ and $\cD$ are C$^*$-categories, we further assume that $L_{\rho,\sigma}$ is a unitary.

\begin{theorem}[{\cite[Theorem 2.2]{I15}}]\label{uniqueness} Let $M$ and $P$ be hyperfinite type III$_1$ factors, 
and let $\cC$ and $\cD$ be C$^*$-fusion categories embedded in $\End(M)$ and $\End(P)$ respectively. 
Let $(F,L)$ be a monoidal functor from $\cC$ to $\cD$ that is an equivalence of the two C$^*$-fusion categories 
$\cC$ and $\cD$. 
Then there exists a surjective isomorphism $\Phi:M\to P$ and unitaries $U_\rho\in P$ for each object $\rho \in \cC$ 
satisfying  
$$F(\rho)=\Ad U_\rho \circ \Phi \circ\rho\circ\Phi^{-1},$$
$$F(X)=U_\sigma\Phi(X)U_\rho^*,\quad X\in (\rho,\sigma),$$
$$L_{\rho,\sigma}=U_{\rho\circ\sigma}\Phi\circ\rho\circ\Phi^{-1}(U_\sigma^*)U_\rho^*=
U_{\rho\circ\sigma}U_\rho^*F(\rho)(U_\sigma^*).$$
\end{theorem}

If $\rho$ is self-conjugate, we have $\overline{R}_\rho=\epsilon R_\rho$ 
with $\epsilon\in \{1,-1\}$. 
This sign $\epsilon$ can be identified with the Frobenius-Schur indicators $\nu_{2,1}(\rho)$ 
(see \cite{NS07} for the definition). 
We say that $\rho$ is real (or symmetrically self-dual) if $\epsilon=1$, and $\rho$ is pseudo-real if $\epsilon=-1$.  

\begin{lemma}\label{SC} Let $\rho\in \End(M)$ be a self-conjugate irreducible 
endomorphism of finite $d(\rho)$. 
If $\dim (\rho,\rho^2)=1$, then $\rho$ is real. 
\end{lemma}

\begin{proof} For $T\in (\rho,\rho^2)$, set $j(T)=\sqrt{d(\rho)}T^*\rho(R_\rho)$. 
Then $j:(\rho,\rho^2)\rightarrow (\rho,\rho^2)$ is an anti-unitary satisfying 
$$j^2(T)=d(\rho)\rho(R_\rho^*)T\rho(R_\rho)=d(\rho)\rho(R_\rho^*)\rho^2(R_\rho)T=\pm T.$$
Since $\dim (\rho,\rho^2)=1$, the case $j^2=-1$ never occurs, and 
$\rho$ is real.  
\end{proof}

\subsection{$G$-kernels and group actions}
Cohomological aspects of finite group actions on factors are well developed in \cite{C77}, \cite{J80}, \cite{S80}, 
and we summarize necessary facts for our purposes here. 

In the category $\End(M)$, an invertible object is nothing but an automorphism of $M$. 
Thus a finite group $G$ consisting of isomorphism classes of invertible objects is nothing but 
a finite subgroup of the outer automorphism group $\Out(M)$. 
Let $\T=\{z\in \C;\; |z|=1\}.$
Since we can always choose a unitary for an isomorphism between two invertible objects, it is natural 
for us to consider group cohomology with coefficient module $\T$ rather than $\C^\times$. 
Note that since $\C^\times=\T\times \R^+$ as trivial $G$-modules, we have 
$H^i(G,\C^\times)\cong H^i(G,\T)$ for $i\geq 1$. 

A $G$-kernel in $M$ is an injective homomorphism from $G$ into the outer automorphism group $\Out(M)$. 
For a $G$-kernel, we choose a lifting $\alpha:G\to \Aut(M)$, which is also called a $G$-kernel. 
Then there exists a unitary $V_{g,h}\in \cU(M)$ for each pair $g,h\in G$ satisfying 
$\alpha_g\circ \alpha_h=\Ad V_{g,h}\circ \alpha_{gh}$. 
By associativity $(\alpha_g\circ \alpha_h)\circ \alpha_k=\alpha_g\circ (\alpha_h\circ \alpha_k)$, 
we have $\Ad (V_{g,h}V_{gh,k})\circ \alpha_{ghk}=\Ad (\alpha_g(V_{h,k})V_{g,hk})\circ \alpha_{ghk},$
and there exists $\omega\in Z^3(G,\T)$ satisfying 
$$\alpha_g(V_{h,k})V_{g,hk}=\omega(g,h,k)V_{g,h}V_{gh,k}.$$ 
The cohomology class $[\omega]\in H^3(G,\T)$ is the exact obstruction for a $G$-kernel to lift 
to an genuine $G$-action, and it is also identified with the cohomology class in $H^3(G,\C^\times)$ 
defined by Eq.(\ref{C1}). 

When $[\omega]$ is trivia, we can choose $V_{g,h}$ so that the equality 
\begin{equation}
\alpha_g(V_{h,k})V_{g,hk}=V_{g,h}V_{gh,k}
\end{equation}
holds. 
The pair $(\alpha, V=\{V_{g,h}\})$ satisfying this relation is called a 2-cocycle action of $G$ on $M$. 
It is known that every 2-cocycle action of a finite group $G$ (with the assumption that $G\ni g\mapsto 
[\alpha_g]\in \Out(M)$ is injective) 
is equivalent to an action, that is, there exists a unitary $U_g\in \cU(M)$ for each $g\in G$ so that 
$\{\Ad U_g\circ \alpha_g\}_{g\in G}$ gives a $G$-action, and $U_g\alpha_g(U_h)V_{g,h}U_{gh}^*=1$.  

In the case of abstract fusion categories discussed in subsection \ref{GHC}, when Eq. (\ref{C4}) holds, 
the other isomorphisms satisfying the same relation are of the form $\zeta(g,h)v_{g,h}$ with 
$\zeta\in Z^2(G,\C^\times)$, and $H^2(G,\C^\times)$ naturally appears in the picture. 
The same mathematical fact takes a different (but of course equivalent) form in our case because 
of the following reason. 
In the case of $\End(M)$, since a $G$-action is a privileged lifting of a given $G$-kernel, 
we change the lifting $\alpha$ to be a $G$-action. 
For such $\alpha$, it is natural to consider only $1$ as an isomorphism from $\alpha_{gh}$ to $\alpha_{g}\circ \alpha_h$. 
Thus instead of considering different isomorphisms between fixed objects, 
we consider different $G$-actions that are lifting of the same $G$-kernel. 

Let $\beta$ be another $G$-action that is an inner perturbation of $\alpha$. 
Then there exists a unitary $U_g\in \cU(M)$ for each $g\in G$ satisfying 
$\beta_g=\Ad U_g\circ \alpha_g$. 
Since $\alpha$ and $\beta$ are $G$-actions, we have 
$$\beta_{gh}=\beta_g\circ \beta_h=\Ad (U_g\alpha_g(U_h))\circ \alpha_{gh}
=\Ad (U_g\alpha_g(U_h)U_{gh}^*)\circ \beta_{gh},$$
and there exists $\zeta\in Z^2(G,\T)$ satisfying $U_g\alpha_g(U_h)=\zeta(g,h)U_{gh}$. 
When $\zeta$ is a coboundary, we can choose $\{U_g\}_{g\in G}$ to satisfy the 1-cocycle relation 
$U_g\alpha_g(U_h)=U_{gh}$. 
In this case, it is known that $U=\{U_g\}_{g\in G}$ is a coboundary, that is, 
there exists a unitary $X\in \cU(M)$ satisfying $U_g=X^{-1}\alpha_g(X)$, and in consequence  
$\beta$ and $\alpha$ are inner conjugate, that is $\beta_g=\Ad X^{-1}\circ \alpha_g\circ \Ad X$. 
In summary, the inner conjugacy classes of the liftings of the same $G$-kernel to actions are 
in one-to-one correspondence with $H^2(G,\T)$, and the correspondence makes sense once a reference 
lifting $\alpha$ is chosen. 

\subsection{The Cuntz algebras}

One of the main tools in this note is the Cuntz algebra $\cO_n$, and we summarize the main feature of it here. 
Let $n$ be an integer larger than 1. 
The Cuntz algebra $\cO_n$ is the universal C$^*$-algebra with generators $\{S_i\}_{i=1}^n$ and relations 
$$S_i^*S_j=\delta_{i,j}1,$$
$$\sum_{i=1}^nS_iS_i^*=1.$$
The most peculiar property of the Cuntz algebra is that it is at the same time universal and simple (see \cite{Cu77}). 
Therefore if $\{T_i\}_{i=1}^n$ are noncommutative polynomials of the generators obeying the same relation 
as the defining relation, then there exists a unique endomorphism $\sigma\in \End(\cO_n)$ satisfying $\sigma(S_i)=T_i$.

\begin{lemma}[{\cite[Lemma 2.6]{I93}}]\label{CE} Let $\rho$ be a unital endomorphism of
the Cuntz algebra $\cO_n$ with the canonical generators
$\{S_1,S_2, \ldots S_n\}$.
We fix $1\leq i\leq n$ and set $T_j:=S_i^*\rho(S_j)S_i$.
If $\{T_1,T_2,\ldots, T_n\}$ satisfy the Cuntz algebra relation, then
$S_k^*\rho(x)S_i=0$ for $k\neq i$ and all $x\in \cO_n$.
In consequence, $\sigma(\cdot):=S_i^*\rho(\cdot)S_i$ is a unital
endomorphism and $S_i\in (\sigma,\rho)$.
\end{lemma}

\section{Polynomial equations for generalized Haagerup categories}\label{PEGHC}
In this section, we deduce polynomial equations for a C$^*$-generalized Haagerup category $\cC$ 
with a finite abelian group $G$. 
For $G$, we use additive notation. 
We set $G_2=\{g\in G;\; 2g=0\}$. 
We denote $n=\#G$ and 
$$d=d(\rho)=\frac{n+\sqrt{n^2+4}}{2}.$$ 

Let $M$ be the hyperfinite type III$_1$ factor. 
Then we may and do assume $\cC\subset \End(M)$.

\begin{definition} For a C$^*$-generalized Haagerup category $\cC\subset \End(M)$ with a finite abelian group $G$, 
we say that a pair $[\rho, \alpha]$ of $\rho\in \End (M)$ and an action $\alpha:G\to\Aut(M)$ satisfying  
$$\cO(\cC)=\{[\alpha_g]\}_{g\in G}\sqcup \{[\alpha_g][\rho]\}_{g\in G}$$ 
is a standard lifting of $\cC$ if $\rho$ and $\alpha$ satisfy the relation 
$\alpha_g\circ \rho=\rho\circ \alpha_{-g}$, 
and $\alpha$ restricted to $G_2$ acts on $(\rho,\rho^2)$ trivially. 
(We do not use the notation $(\rho,\alpha)$ for the pair of $\rho$ and $\alpha$ in order to avoid possible 
confusion with the intertwiner space $(\rho,\alpha_g)$.)

We say that two standard liftings $[\rho,\alpha]$ and $[\rho',\alpha']$ are conjugate (resp. inner conjugate) 
if there exist $\theta\in \Aut(M)$ (resp. an inner automorphism $\theta$) satisfying 
$\rho'=\theta\circ \rho\circ \theta^{-1}$ and $\alpha'_g=\theta\circ \alpha_g\circ \theta^{-1}$. 
\end{definition}

\begin{lemma} \label{standard} 
For a C$^*$-generalized Haagerup category $\cC\subset \End(M)$, there always exists a standard lifting. 
\end{lemma}

\begin{proof} We can choose $\rho\in \End(M)$ and $\alpha_g\in \Aut(M)$ with 
$$\cO(\cC)=\{[\alpha_g]\}_{g\in G}\sqcup \{[\alpha_g][\rho]\}_{g\in G}$$ satisfying the fusion rules 
$$[\alpha_g][\alpha_h]=[\alpha_{g+h}],$$
$$[\alpha_g][\rho]=[\rho][\alpha_{-g}],$$
$$[\rho^2]=[\id]+ \sum_{g\in G}[\alpha_g\circ\rho].$$
Since $c^{0,3}(\cC)$ and $c^{1,2}(\cC)$ are trivial, thanks to Eq.(\ref{C4}) and Eq.(\ref{C6}), 
there exist unitaries $V_{g,h}\in (\alpha_g\circ\alpha_h,\alpha_{g+h})$ 
and $W_g\in (\alpha_g\circ \rho,\rho\circ \alpha_{-g})$ satisfying 
$$\alpha_g(V_{h,k})V_{g,h+k}=V_{g,h}V_{g+h,k},$$
$$W_g \alpha_g(W_h)V_{g,h}W_{g+h}^{-1}
=\rho(V_{-g,-h}).$$
The first equation shows that the pair $(\alpha,\{V_{g,h}\}_{g,h\in G})$ is a cocycle action of $G$, and 
there exists a unitary $U_g\in \cU(M)$ for each $g\in G$ satisfying 
$V_{g,h}=\alpha_g(U_h^{-1})U_g^{-1}U_{g+h}$ and $\alpha'$ defined by 
$\alpha'_g=\Ad U_g\circ \alpha_g$ is a $G$-action. 
The second equation implies that if we set $W'_g=\rho(U_{-g})W_gU_g^{-1}$, then $W'_g\in 
(\alpha'_g\circ \rho,\rho\circ \alpha'_{-g})$ and $W'=\{W'_g\}_{g\in G}$ is an $\alpha'$ cocycle. 
Thus there exists a unitary $X\in \cU(M)$ satisfying $W_g=X^{-1}\alpha_g(X)$,
and we get
$$\alpha'_g\circ \Ad X\circ \rho=\Ad X\circ \rho \circ \alpha'_{-g}.$$
Setting $\rho'=\Ad X\circ \rho$, we get $\alpha'_g\circ \rho'=\rho'\circ \alpha'_{-g}$. 
To simplify the notation, we may and do assume that $\alpha$ is an action and 
$\rho$ and $\alpha$ satisfy the relation $\alpha_g\circ \rho=\rho\circ \alpha_{-g}$ 
from the beginning by replacing $\rho$ and $\alpha$ with $\rho'$ and $\alpha'$ respectively. 
 
Next we show that $\alpha$ restricted to $G_2$ globally fix $(\rho,\rho^2)$. 
Since $\dim (\rho,\rho^2)=1$. we can choose an isometry $T\in (\rho,\rho^2)$ with 
$(\rho,\rho^2)=\C T$. 
Then for any $g\in G$, we have 
$$\alpha_g(T)\rho(x)=\alpha_g(T\rho(\alpha_g(x)))=\alpha_g(\rho^2(\alpha_g(x))T)=\rho^2(\alpha_{2g}(x))\alpha_g(T),$$
and if $z\in G_2$, we get $\alpha_z(T)\in \C T$. 
Thus there exists a character $\chi\in \widehat{G_2}$ satisfying $\alpha_z(T)=\chi(z)T$. 
Note that $\chi(z)\in \{1,-1\}$. 
Since every character of $G_2$ extends to a character of $G$, we choose such an extension $\chi'\in \widehat{G}$. 

Let $Y\in \cU(M)$ be a unitary satisfying $\alpha_g(Y)=\chi'(g)Y$ for any $g\in G$. 
Let $\rho'=\Ad Y\circ \rho$, and let $T'=Y\rho(Y)TY^{-1}$. 
Then $Y'\in (\rho',\rho'^2)$ and $\alpha_g\circ\rho'=\rho'\circ \alpha_{-g}$. 
For $z\in G_2$, we have 
\begin{align*}
 \alpha_z(T')&=\alpha_z(Y\rho(Y)TY^{-1}))=\alpha_z(Y)\rho(\alpha_{-z}(Y))\alpha_z(T)\alpha_z(Y^{-1})) \\
 &=\chi(-z)\chi(z)T'=T'.
\end{align*}
Thus $[\alpha,\rho']$ is a standard lifting. 
\end{proof}

In what follows, we fix one standard lifting $[\rho,\alpha]$ for $\cC$ and obtain polynomial equations for it. 

We first choose an isometry $S\in (\id,\rho^2)$. 
Then we have $\alpha_g(S)\in (\id,\rho^2)$ because  
$$\alpha_g(S)x=\alpha_g(S\alpha_{-g}(x))=\alpha_g(\rho^2\alpha_{-g}(x)S)=\rho^2(x)\alpha_g(S).$$

\begin{lemma}\label{fixed} With the above notation, we have 
\begin{equation}\alpha_g(S)=S.
\end{equation}
\end{lemma}

\begin{proof} 
Since 
$$\dim(\id,(\alpha_g\rho)^2)=\dim(\id,\rho^2)=1,$$ 
$$\dim(\alpha_g\rho,(\alpha_g\rho)^2)=\dim(\alpha_g\rho,\rho^2)=1,$$ 
Lemma \ref{SC} implies that the endomorphism $\alpha_g\rho$ is real. 
Since $\alpha_{-g}(S)$ is a scalar multiple of $S$ and 
$$S^*\rho(\alpha_{-g}(S))=S^*\alpha_g\rho(S)=\frac{1}{d},$$
we get $\alpha_{-g}(S)=S$. 
\end{proof}

\begin{remark} The above lemma shows that $\fc^{2,1}(\cC)$ and $\fc^{3,0}(\cC)$ are trivial too.  
Indeed, since $[\rho,\alpha]$ is a standard lifting, we can choose $v_{g,h}$ and $w_g$ in 
subsection \ref{GHC} to be 1. 
Then $\eta(g)$ in Eq.(\ref{c12}) is given by $d\alpha_g(S^*\rho(\alpha_g(S)))=1$, 
and $\fc^{2,1}(\cC)$ is trivial. Since $\rho$ is real, $\fc^{3,0}(\cC)$ is trivial too.
\end{remark}

Now we examine the anti-unitaries on $(\alpha_g\rho,(\alpha_g\rho)^2)=(\alpha_g\rho,\rho^2)$ 
coming from the Frobenius reciprocity. 
For $T\in (\alpha_g\rho,(\alpha_g\rho)^2)=(\alpha_g\rho,\rho^2)$ we set 
\begin{equation}j_{1,g}(T)=\sqrt{d}T^*\alpha_g\rho(S)=\sqrt{d}T^*\rho(S),\end{equation}
\begin{equation}j_{2,g}(T)=\sqrt{d}\alpha_g\rho(T_g)^*S.\end{equation}
Then $j_{1,g}$ and $j_{2,g}$ are anti-unitaries of $(\alpha_g\rho,\rho^2)$ with 
$j_{1,g}^2=j_{2,g}^2=1$. 
We choose an isometry $T_g\in(\alpha_g\rho,\rho^2)$ satisfying $j_{1,g}(T_g)=T_g$, 
which is uniquely determined up to sign. 
Then $\{S\}\cup\{T_g\}_{g\in G}$ satisfy the Cuntz algebra relation. 
We denote by $\cO_{n+1}$ the C$^*$-algebra generated by these isometries. 
We have $\alpha_h(T_g)\in (\alpha_{g+2h}\rho,\rho^2)$ 
because  
$$\alpha_h(T_g)\alpha_{g+2h}\rho(x)=\alpha_h(T_g\alpha_g\rho\alpha_{-h}(x))=\alpha_h(\rho^2\alpha_{-h}(x)T_g)
=\rho^2(x)\alpha_h(T_g).$$
Moreover, since 
$$\alpha_h(T_g)=\alpha_h(j_{1,g}(T_g))=\sqrt{d}\alpha_h(T_g^*\rho(S))
=\sqrt{d}\alpha_h(T_g)^*\rho(\alpha_{-g}(S))=j_{1,g+2h}\alpha_h(T_g),$$  
we have 
\begin{equation}\alpha_h(T_g)=\epsilon_h(g)T_{g+2h},\end{equation}
with $\epsilon_h(g)\in \{1,-1\}$ satisfying the cocycle identity: 
\begin{equation}\epsilon_{h+k}(g)=\epsilon_h(g)\epsilon_k(g+2h).\end{equation}
Since $T_g$ is uniquely determined only up to sign, we have freedom to replace 
$T_g$ with $\delta_gT_g$ satisfying $\delta_g\in \{1,-1\}$.  
This amounts to replacing $\epsilon_h(g)$ with $\epsilon_{h}(g)\delta_g\delta_{g+2h}$, 
that is, multiplying by a coboundary term.   

There exists a character $\chi_g$ of $G_2$ for each $g\in G$ satisfying 
$\alpha_z(T_g)=\chi_g(z)T_g$ for all $z\in G_2$, or equivalently, $\epsilon_z(g)=\chi_g(z)$. 
Note that we have $\chi_0(z)=1$. 
Since $\alpha_z(\alpha_h(T_g))=\alpha_h(\alpha_z(T_g))$, we have $\chi_{g+2h}=\chi_g$ for any $g,h\in G$.

\begin{remark} Since $\alpha_z(T_0)=T_0$ for any $z\in G_2$, we may assume $\alpha_h(T_0)=T_{2h}$ and 
$\epsilon_h(0)=1$ for any $h\in G$. 
In a similar way, we can see that the cohomology class of 
the cocycle $\{\epsilon_h(g) \}_{g,h}$ is determined by $\{\chi_g\}_{g\in G}$. 
\end{remark}

Since $\dim(\alpha_g\rho,\rho^2)=1$, there exists $\eta_g\in \T$ satisfying $j_{2,g}(T_g)=\eta_gT_g$. 
Since 
$$j_{2,g+2h}(\alpha_h(T_g))=\sqrt{d}\alpha_{g+2h}\rho(\alpha_h(T_g)^*)S=\sqrt{d}\alpha_{g+h}\rho(T_g)^*S=\alpha_hj_{2,g}(T_g),$$ 
we have 
\begin{equation}\eta_{g+2h}=\eta_g.
\end{equation}
The unitary $j_{2,g}j_{1,g}$ is called the rotation map, and it does not depend on the choice of $S$. 
In our case, it reduces to the scalar $\eta_g$. 

Now we determine the form of $\rho$ on $S$ and $T_g$. 
We set $P=SS^*$ and $Q=\sum_{g\in G}T_gT_g^*$. 
In order to determine $\rho(T_g)$, it suffices to determine $\alpha_g\rho(T_g)$ as 
we have $\alpha_g\rho(T_g)=\rho\alpha_{-g}(T_g)=\epsilon_{-g}(g)\rho(T_{-g})$.  

\begin{lemma}\label{ST}
\begin{equation}
\rho(S)=\frac{1}{d}S+\frac{1}{\sqrt{d}}\sum_{g\in G}T_gT_g.
\end{equation} 
There exists $A_g(h,k)\in \C$ satisfying 
\begin{equation}
\alpha_g\rho(T_g)=\eta_gT_gSS^*+\frac{\overline{\eta_g}}{\sqrt{d}}ST_g^*+
\sum_{h,k\in G}A_g(h,k)T_{g+h}T_{g+h+k}T_{g+k}^*.
\end{equation}
\end{lemma}

\begin{proof}
Since $P+Q=1$, we get 
$$\rho(S)=(SS^*+\sum_{g\in G}T_gT_g^*)\rho(S)=\frac{1}{d}S+\frac{1}{\sqrt{d}}\sum_{g\in G}T_gj_{1,g}(T_g)
=\frac{1}{d}S+\frac{1}{\sqrt{d}}\sum_{g\in G}T_gT_g.$$ 
We compute each of $P\alpha_g\rho(T_g)$, $\alpha_g\rho(T_g)P$ and $Q\alpha_g\rho(T_g)Q$ now.
$$P\alpha_g\rho(T_g)=SS^*\alpha_g\rho(T_g)=\frac{1}{\sqrt{d}}Sj_{2,g}(T_g)^*=\frac{\overline{\eta_g}}{\sqrt{d}}ST_g^*,$$
\begin{align*}
\alpha_g\rho(T_g)P &=\alpha_g\rho(j_{1,g}(T_g))P=\sqrt{d}\alpha_g\rho(T_g^*\rho(S))P
=\sqrt{d}\alpha_g\rho(T_g^*)SSS^* \\
 &=j_{2,g}(T_g)SS^*=\eta_gT_gSS^*.
\end{align*}
For $h,k\in G$, we claim $T_{g+h}^*\alpha_g\rho(T_g)T_{g+k}\in (\alpha_{g+h+k},\rho^2)$. 
Indeed,  
\begin{align*}\lefteqn{
T_{g+h}^*\alpha_g\rho(T_g)T_{g+k}\alpha_{g+h+k}\rho(x) =T_{g+h}^*\alpha_g\rho(T_g)T_{g+k}\alpha_{g+k}\rho(\alpha_{-h}(x))} \\
 &=T_{g+h}^*\alpha_g\rho(T_g)\rho^2(\alpha_{-h}(x))T_{g+k}
 =T_{g+h}^*\alpha_g\rho(T_g\alpha_g\rho(\alpha_{-h}(x)))T_{g+k}\\
 &=T_{g+h}^*\alpha_g\rho(\rho^2(\alpha_{-h}(x))T_g)T_{g+k}
 =T_{g+h}^*\rho^3(\alpha_{-g-h}(x))\alpha_g\rho(T_g)T_{g+k}\\
 &=\alpha_{g+h}\rho^2\alpha_{-g-h}(x)T_{g+h}^*\alpha_g\rho(T_g)T_{g+k}
 =\rho^2(x)T_{g+h}^*\alpha_g\rho(T_g)T_{g+k}.
\end{align*}
Therefore there exists a scaler $A_g(h,k)\in \C$ satisfying 
$$Q\alpha_g\rho(T_g)Q=\sum_{h,k\in G}A_g(h,k)T_{g+h}T_{g+h+k}T_{g+k}^*.$$
This finishes the proof. 
\end{proof}

Now we examine how the choices of $S$ and $T_g$ effect on $(\epsilon_{h}(g),\eta_g,A_g(h,k))$. 
Let $c\in \T$, and let $c^{1/2}$ be one of its square root. 
If we replace $S$ with $cS$, then $j_{1,g}$ is replace by $cj_{1,g}$, and it fixes $c^{1/2}T_g$. 
The choices of $cS$ and $c^{1/2}T_g$ instead of $S$ and $T_g$ do not change 
$(\epsilon_{h}(g),\eta_g,A_g(h,k))$ at all. 
Therefore we don't need to think of a different choice of $S$ from $(\id_M,\rho^2)$, 
and we fix $S$.  
If we replace $T_g$ with $\delta_gT_g$ satisfying $\delta_g\in \{1,-1\}$, 
then $(\epsilon_{h}(g),\eta_g,A_g(h,k))$ is replaced with $(\epsilon'_{h}(g),\eta_g,A'_g(h,k))$
where 
\begin{equation}
\epsilon'_h(g)=\delta_g\delta_{g+2h}\epsilon_h(g),
\end{equation}
\begin{equation}
A'_{g}(h,k)=\delta_g\delta_{g+h}\delta_{g+k}\delta_{g+h+k}A_g(h,k). 
\end{equation}

\begin{definition}
We call the above transformation from a triplet $(\epsilon_{h}(g),\eta_g,A_g(h,k))$ to 
another triplet $(\epsilon'_{h}(g),\eta_g,A'_g(h,k))$ 
a gauge transformation by $\{\delta_g\}_{g\in G}$. 
We say that two triplets are gauge equivalent if they are transformed to each other by a gauge transformation. 
\end{definition}

For a gauge transformation by $\{\delta_g\}_{g\in G}$, we may always assume $\delta_0=1$. 

Note that thanks to Lemma \ref{ST}, any intertwiner between two endomorphisms obtained by composing endomorphisms in $\{\rho,\alpha_g\}$ 
in arbitrary times is a polynomial of $\{S,S^*,T_g,T_g^*\}$, and we can show as in \cite{SW02} that the 6j-symbols of the fusion category 
$\cC$ is completely determined by the numerical data $(\epsilon_{h}(g),\eta_g,A_g(h,k))$. 
Thus we obtain the following theorem.  

\begin{theorem} \label{mf} Let $\cC,\cC'\subset \End_0(M)$ be a generalized Haagerup categories with a finite abelian group $G$, and let 
$[\rho,\alpha]$ and $[\rho',\alpha']$ be standard lifting of $\cC$ and $\cC'$. 
If $[\rho,\alpha]$ and $[\rho',\alpha']$ have gauge equivalent numerical data, there exists a monoidal functor $(F,L)$ from 
$\cC$ to $\cC'$ with trivial $L$, which is an equivalence of the two fusion categories,  
satisfying $F(\rho)=\rho'$, $F(\alpha_g)=\alpha'_g$. 
\end{theorem}

\begin{remark}\label{choices} Our primary goal in this paper is to classify the $3^G$ subfactors, 
and for this goal we classify the generalized Haagerup categories $\cC$ with a finite abelian group $G$ and 
a distinguished simple object $\rho$ by the triplet $(\epsilon_{h}(g),\eta_g,A_g(h,k))$. 
To obtain the triplet from $\cC\subset \End(M)$, we made the following choices: 
\begin{itemize}
\item[(1)] identification of $G$ with the group of the invertible objects of $\cC$, 
\item[(2)] the standard lifting $[\rho,\alpha]$, 
\item[(3)] the isometry $S\in (\id,\rho^2)$, and 
\item[(4)] the isometry $T_g\in (\alpha_g\rho,\rho^2)$ satisfying $j_{1,g}(T_g)=T_g$. 
\end{itemize}
In view of Theorem \ref{uniqueness}, to classify $\cC$ with $G$ and a distinguished simple object $\rho$ 
by the triplet $(\epsilon_{h}(g),\eta_g,A_g(h,k))$, 
it is necessary and sufficient to describe how different choices in (1),(2),(3), and (4) transform the triplets. 
Different choices in (1) can be describe by the action of $\Aut(G)$, and those in (3) and (4) altogether 
can be described by the gauge transformations. 
Thus it is essential to describe how different choices of standard liftings transform the triplets, 
which involves $H^2(G,\T)$. 
We will show in Section \ref{classifiction} that the equivalence classes of $\cC$ with a distinguished object $\rho$ 
are in one-to-one correspondence with the $H^2(G,\T)\rtimes \Aut(G)$-orbits of the gauge equivalence classes of the triplets. 
To classify $\cC$ without specifying $\rho$, we still have freedom to replace $\rho$ with $\alpha_g\rho$, 
which makes an action of $G$ (in fact $G/2G$) on the gauge equivalence classes of the triplets.  
\end{remark}

We now deduce polynomial equations among $(\epsilon_{h}(g),\eta_g,A_g(h,k))$. 

\begin{lemma}[Orthogonality]\label{orthogonality}  
\begin{equation}
\sum_{h\in G}A_g(h,0)=-\frac{\overline{\eta_g}}{d},
\end{equation}
\begin{equation}
\sum_{h\in G}A_g(h-g,k)\overline{A_{g'}(h-g',k)}=\delta_{g,g'}-\frac{\overline{\eta_g}\eta_{g'}}{d}\delta_{k,0}.
\end{equation}
\end{lemma}

\begin{proof} $\rho(S)^*\alpha_g\rho(T_g)=\alpha_g\rho(S^*T_g)=0$ implies the first equation. 

$\alpha_{g'}\rho(T_{g'})^*\alpha_g\rho(T_g)=\epsilon_{-g'}(g')\epsilon_{-g}(g)\rho(T_{-g'}^*T_{-g})=\delta_{g,g'}$ 
implies the second. 
\end{proof}

$\alpha_g\rho=\rho\alpha_{-g}$ and $\alpha_h(T_g)=\epsilon_h(g)T_{g+2h}$ imply the following:

\begin{lemma}\label{alpha}
\begin{equation}
A_{g+2h}(p,q)=\epsilon_h(g)\epsilon_h(g+p)\epsilon_h(g+q)\epsilon_h(g+p+q)A_g(p,q),
\end{equation}
\end{lemma}

\begin{proof} Since 
$$\alpha_{g+2h}\rho(T_{g+2h})=\epsilon_h(g)\alpha_{g+2h}\rho\alpha_h(T_g)=\epsilon_h(g)\alpha_{g+h}\rho(T_g)
=\epsilon_h(g)\alpha_h(\alpha_g\rho(T_g)),$$
we get the statement. 
\end{proof}

$j_{1,g}(T_g)=T_g$ and $j_{2,g}(T_g)=\eta_gT_g$ imply the following:
\begin{lemma}\label{Frobenius} 
\begin{equation}
A_g(k,h)=\overline{A_g(h,k)},
\end{equation}
\begin{equation}
\eta_g^3=1,
\end{equation}
\begin{align}
A_g(h,k)&=\eta_g\epsilon_{-k}(g+h)\epsilon_{-k}(g+k)\epsilon_{-k}(g+h+k)A_g(-k,h-k)\\
&=\eta_g^2\epsilon_{-h}(g+h)\epsilon_{-h}(g+k)\epsilon_{-h}(g+h+k)A_g(k-h,-h).\nonumber
\end{align}
\end{lemma}

\begin{proof} 
Since 
\begin{align*}\lefteqn{
\alpha_g\rho(T_g) =\alpha_g\rho(j_{1,g}(T_g))=\sqrt{d}\alpha_g\rho(T_g^*\rho(S))}\\
&=\sqrt{d}\alpha_g\rho(T_g^*)(SSS^*+\sum_{h\in G}T_h\rho(S)T_h^*) \\
 &=j_{2,g}(T_g)SS^*+\sum_{h\in G}\frac{1}{d}\alpha_g\rho(T_g^*)T_hST_h^*+\frac{1}{\sqrt{d}}\sum_{h,k\in G}\alpha_g\rho(T_g^*)T_hT_kT_kT_h^*\\
 &=\eta_gT_gSS^*+\frac{\overline{\eta_g}}{\sqrt{d}}ST_g^*
 +\frac{1}{\sqrt{d}}\sum_{h,k\in G}\alpha_g\rho(T_g^*)T_{g+h}T_{g+h+k}T_{g+h+k}T_{g+h}^*\\
 &=\eta_gT_gSS^*+\frac{\overline{\eta_g}}{\sqrt{d}}ST_g^*
 +\frac{1}{\sqrt{d}}\sum_{h,k\in G}\overline{A_g(h,k)}T_{g+k}T_{g+h+k}T_{g+h}^*,
\end{align*}
the first equation holds. 

Since 
\begin{align*}\lefteqn{
\alpha_g\rho(T_g) =\overline{\eta_g}\alpha_g\rho(j_{2,g}(T_g))=\overline{\eta_g}\sqrt{d}\alpha_g\rho(\alpha_g\rho(T_g)^*S)
=\overline{\eta_g}\sqrt{d}\rho^2(T_g)^*\rho(S)}\\
 &= \overline{\eta_g}\sqrt{d}\rho^2(T_g)^*(\frac{1}{d}S+\frac{1}{\sqrt{d}}\sum_{k\in G}T_{g-k}T_{g-k})\\
 &=\frac{\overline{\eta_g}}{\sqrt{d}}ST_g^*+\overline{\eta_g}\sum_{k\in G}T_{g-k}\alpha_{g-k}\rho(T_g)^*T_{g-k}\\
 &=\frac{\overline{\eta_g}}{\sqrt{d}}ST_g^*+\overline{\eta_g}\sum_{k\in G}T_{g-k}\alpha_{-k}(\alpha_g\rho(T_g)^*\alpha_k(T_{g-k}))\\
 &=\frac{\overline{\eta_g}}{\sqrt{d}}ST_g^*
 +\overline{\eta_g}\sum_{k\in G}\epsilon_k(g-k)T_{g-k}\alpha_{-k}(\alpha_g\rho(T_g)^*T_{g+k})\\
 &=\frac{\overline{\eta_g}}{\sqrt{d}}ST_g^*+\overline{\eta_g}^2T_gSS^*
 +\overline{\eta_g}\sum_{h,k\in G}\overline{A_g(k,h)}\epsilon_k(g-k)T_{g-k}\alpha_{-k}(T_{g+h}T_{g+h+k}^*)\\
 &=\frac{\overline{\eta_g}}{\sqrt{d}}ST_g^*+\overline{\eta_g}^2T_gSS^*\\
 &+\overline{\eta_g}\sum_{h,k\in G}\overline{A_g(k,h)}\epsilon_k(g-k)\epsilon_{-k}(g+h)\epsilon_{-k}(g+h+k)
 T_{g-k}T_{g+h-2k}T_{g+h-k}^*, 
\end{align*}
and $\epsilon_k(g-k)=\epsilon_{-k}(g+k)$, we get 
$$A_g(-k,h-k)=\overline{A_g(k,h)}\overline{\eta_g}\epsilon_{-k}(g+k)\epsilon_{-k}(g+h)\epsilon_{-k}(g+h+k),$$
which finishes the statement.  
\end{proof}

\begin{remark} Assume that $h\in G$ satisfies $3h=0$ (e.g. $h=0$). 
Then we have 
$$A_g(h,2h)=\eta_g\epsilon_h(g)\epsilon_h(g+h)\epsilon_h(g+2h)A_g(h,2h).$$
This shows that $A_g(h,2h)\neq 0$ implies $\eta_g=1$.  
\end{remark}

From $S^*\rho^2(T_g)S=T_g$, we get 

\begin{lemma}\label{id}
\begin{equation}\label{I}
A_g(h,k)=A_{g+h+k}(h,k)\epsilon_h(g+k)\epsilon_k(g+h)\epsilon_{h+k}(g)\eta_{g+h}\overline{\eta_{g+k}}.
\end{equation}
\end{lemma}

\begin{proof} 
\begin{align*}
\lefteqn{S^*\rho^2(T_g)S=S^*\alpha_g\rho(\alpha_g\rho(T_g))S} \\
 &=\eta_gS^*\alpha_g\rho(T_gSS^*)S+\frac{\overline{\eta_g}}{\sqrt{d}}S^*\alpha_g\rho(ST_g^*)S+
 \sum_{h,k\in G}A_g(h,k)S^*\alpha_g\rho(T_{g+h}T_{g+h+k}T_{g+k}^*)S.
\end{align*}
The first term is 
$$\eta_gS^*\alpha_g\rho(T_gSS^*)S=\frac{1}{d\sqrt{d}}T_g^*\alpha_g\rho(S)=\frac{1}{d^2}T_g.$$
The second term is 
$$\frac{\overline{\eta_g}}{\sqrt{d}}S^*\alpha_g\rho(ST_g^*)S=\frac{1}{d^2}T_g.$$
The third term is 
\begin{align*}
\lefteqn{\sum_{h,k\in G}A_g(h,k)S^*\alpha_g\rho(T_{g+h}T_{g+h+k}T_{g+k}^*)S}\\
&=\sum_{h,k\in G}A_g(h,k)\alpha_{-h}(S^*\alpha_{g+h}\rho(T_{g+h}))\alpha_g\rho(T_{g+h+k})\alpha_{-k}(\alpha_{g+k}(T_{g+k}^*)S) \\
&=\frac{1}{d}\sum_{h,k\in G}A_g(h,k)\overline{\eta_{g+h}}\eta_{g+k}\alpha_{-h}(T_{g+h}^*)\alpha_g\rho(T_{g+h+k})\alpha_{-k}(T_{g+k}) \\
&=\frac{1}{d}\sum_{h,k\in G}A_g(h,k)\overline{\eta_{g+h}}\eta_{g+k}\epsilon_k(g+h)\epsilon_h(g+k)
\alpha_{-h-k}(T_{g+h+2k}^*\alpha_{g+h+k}\rho(T_{g+h+k})T_{g+2h+k}) \\
&=\frac{1}{d}\sum_{h,k\in G}A_g(h,k)A_{g+h+k}(k,h)\overline{\eta_{g+h}}\eta_{g+k}\epsilon_k(g+h)\epsilon_h(g+k)\alpha_{-h-k}(T_{g+2h+2k}) \\
&=\frac{1}{d}\sum_{h,k\in G}A_g(h,k)\overline{A_{g+h+k}(h,k)}\overline{\eta_{g+h}}\eta_{g+k}
\epsilon_k(g+h)\epsilon_h(g+k)\epsilon_{-h-k}(g+2h+2k)T_g \\
&=\frac{1}{d}\sum_{h,k\in G}A_g(h,k)\overline{A_{g+h+k}(h,k)}\overline{\eta_{g+h}}\eta_{g+k}
\epsilon_k(g+h)\epsilon_h(g+k)\epsilon_{h+k}(g)T_g.
\end{align*}
Since $d(1-\frac{2}{d^2})=n-\frac{1}{d}$, we get 
$$\sum_{h,k\in G}A_g(h,k)\overline{A_{g+h+k}(h,k)}\overline{\eta_{g+h}}\eta_{g+k}
\epsilon_k(g+h)\epsilon_h(g+k)\epsilon_{h+k}(g)=n-\frac{1}{d},$$
which implies
$$\sum_{g,h,k\in G}A_g(h,k)\overline{A_{g+h+k}(h,k)}\overline{\eta_{g+h}}\eta_{g+k}
\epsilon_k(g+h)\epsilon_h(g+k)\epsilon_{h+k}(g)=(n-\frac{1}{d})n.$$
On the other hand, we have   
$$\sum_{g,h,k\in G}|A_g(h,k)|^2=(n-\frac{1}{d})n,$$
from Lemma \ref{orthogonality}. 
This proves the statement. 
\end{proof}

We compute the both sides of the equation
\begin{equation}\label{ghk}
\alpha_g\rho(T_{g+h}^*\alpha_g\rho(T_g))T_{g+k}=\alpha_g\rho(T_{g+h}^*)T_{g+k}\alpha_{g+k}\rho(T_g),
\end{equation}
and obtain the following statement. 

\begin{lemma}\label{AAA=AA}
\begin{align}\label{A1}
A_g(h,k)&=A_{g+h}(h,k)\eta_g\eta_{g+k}\overline{\eta_{g+h}}\overline{\eta_{g+h+k}}\epsilon_h(g)\epsilon_h(g+k)\\
&=A_{g+k}(h,k)\overline{\eta_g\eta_{g+h}}\eta_{g+k}\eta_{g+h+k}\epsilon_k(g)\epsilon_k(g+h),\nonumber
\end{align}

\begin{align}\label{A2}
\lefteqn{
\sum_{l\in G}A_g(x+y,l)A_{g-p+x}(-x,l+p)A_{g-q+x+y}(-y,l+q)} \\
&=\eta_g\eta_{g+q+x}\eta_{g+p+q+y}\overline{\eta_{g+p}\eta_{g+x+y}\eta_{g+q+x+y}}\nonumber\\
&\times A_{g-p}(q+y,p+x+y)A_g(p+x,q+x+y)\nonumber \\ 
&\times \epsilon_p(g-p+x)\epsilon_{p+x}(g-p+q+y)\epsilon_q(g-q+x+y)\epsilon_{q+y}(g-q+x)\nonumber\\
&-\frac{\delta_{x,0}\delta_{y,0}}{d}\eta_g\eta_{g+p}\eta_{g+q}.\nonumber
\end{align}
\end{lemma}

\begin{proof}
The left-hand side of (\ref{ghk}) is 
\begin{align*}
\lefteqn{\delta_{h,0}\eta_g\alpha_g\rho(SS^*)T_{g+k}+\sum_{l\in G}A_g(h,l)\alpha_g\rho(T_{g+h+l}T_{g+l}^*)T_{g+k} } \\
 &=\frac{\delta_{h,0}\eta_g}{\sqrt{d}}\rho(S)T_{g+k}^*
 +\sum_{l\in G}A_g(h,l)\epsilon_l(g+k)\alpha_g\rho(T_{g+h+l})\alpha_{-l}(\alpha_{g+l}\rho(T_{g+l}^*)T_{g+k+2l}) \\
 &=\frac{\delta_{h,0}\eta_g}{\sqrt{d}}\rho(S)T_{g+k}^*
 +A_g(h,-k)\epsilon_{-k}(g+k)\overline{\eta_{g-k}}\alpha_g\rho(T_{g+h-k})SS^* \\
 &+\sum_{l,r\in G}A_g(h,l)\overline{A_{g+l}(k+l,r)}\epsilon_l(g+k)\alpha_g\rho(T_{g+h+l})\alpha_{-l}(T_{g+l+r}T_{g+k+2l+r}^*). \\
\end{align*}
The first term is 
$$\frac{\delta_{h,0}\eta_g}{d\sqrt{d}}ST_{g+k}^*+\frac{\delta_{h,0}\eta_g}{d}\sum_{s\in G}T_{g+s}T_{g+s}T_{g+k}^*.$$
The second term is 
\begin{align*}
\lefteqn{A_g(h,-k)\epsilon_{-k}(g+k)\overline{\eta_{g-k}}\alpha_{k-h}(\alpha_{g+h-k}\rho(T_{g+h-k})SS^*)} \\
 &=A_g(h,-k)\epsilon_{-k}(g+k)\epsilon_{k-h}(g+h-k)\overline{\eta_{g-k}}\eta_{g+h-k}T_{g-h+k}SS^*.
\end{align*}
The last term is 
\begin{align*}
\lefteqn{\sum_{l,r\in G}A_g(h,l)A_{g+l}(r,k+l)\epsilon_l(g+k)\epsilon_l(g+k+r)
 \alpha_{-h-l}(\alpha_{g+h+l}\rho(T_{g+h+l})\alpha_h(T_{g+l+r}))T_{g+k+r}^*}\\  
 &=\sum_{l,r\in G}A_g(h,l)A_{g+l}(r,k+l)\epsilon_l(g+k)\epsilon_l(g+k+r)\epsilon_h(g+l+r)\\
&\times \alpha_{-h-l}(\alpha_{g+h+l}\rho(T_{g+h+l})T_{g+2h+l+r})T_{g+k+r}^* \\
 &=\frac{1}{\sqrt{d}}\sum_{l\in G}A_g(h,l)A_{g+l}(-h,k+l)\epsilon_l(g+k)\epsilon_l(g-h+k)\epsilon_h(g-h+l)\overline{\eta_{g+h+l}}
 ST_{g-h+k}^*  \\
&+\sum_{l,r,s\in G}A_g(h,l)A_{g+l}(r,k+l)A_{g+h+l}(h+l+s,h+r)\\
&\times \epsilon_l(g+k)\epsilon_l(g+k+r)\epsilon_h(g+l+r)\alpha_{-h-l}(T_{g+2h+2l+s}T_{g+3h+2l+r+s})T_{g+k+r}^* \\
&=\frac{1}{\sqrt{d}}\sum_{l\in G}A_g(h,l)A_{g+l}(-h,k+l)\epsilon_l(g+k)\epsilon_l(g-h+k)\epsilon_h(g-h+l)\overline{\eta_{g+h+l}}
 ST_{g-h+k}^*  \\
&+\sum_{l,r,s\in G}A_g(h,l)A_{g+l}(r,k+l)A_{g+h+l}(h+l+s,h+r)\\
&\times \epsilon_l(g+k)\epsilon_l(g+k+r)\epsilon_h(g+l+r)
\epsilon_{h+l}(g+s)\epsilon_{h+l}(g+h+r+s) \\
&\times T_{g+s}T_{g+h+r+s}T_{g+k+r}^* \\
\end{align*}

On the other hand, the right-hand side of (\ref{ghk}) is  
\begin{align*}
\lefteqn{\alpha_g\rho(T_{g+h}^*)T_{g+k}\alpha_{g+k}\rho(T_g)=
\alpha_{-h}(\alpha_{g+h}\rho(T_{g+h}^*)\alpha_h(T_{g+k}))\alpha_{g+k}\rho(T_g)} \\
 &=\epsilon_h(g+k)\alpha_{-h}(\alpha_{g+h}\rho(T_{g+h}^*)T_{g+2h+k})\alpha_{g+k}\rho(T_g) \\
 &=\delta_{h+k,0}\epsilon_h(g-h)\overline{\eta_{g+h}}S\alpha_{-h}(S^*\alpha_g\rho(T_g)) \\
 &+\epsilon_h(g+k)\sum_{s\in G}\overline{A_{g+h}(h+k,s+h)}\alpha_{-h}(T_{g+2h+s}T_{g+3h+k+s}^*)\alpha_{g+k}\rho(T_g).
\end{align*}
The first term is
$$\frac{1}{\sqrt{d}}\delta_{h+k,0}\epsilon_h(g-h)\overline{\eta_{g+h}\eta_g}S\alpha_{-h}(T_g^*)
=\frac{1}{\sqrt{d}}\delta_{h+k,0}\epsilon_h(g-h)\epsilon_{-h}(g)\overline{\eta_{g+h}\eta_g}ST_{g-2h}^*.$$
The second term is 
\begin{align*}
\lefteqn{\epsilon_h(g+k)\sum_{s\in G}A_{g+h}(h+s,h+k)\epsilon_h(g+s)T_{g+s}\alpha_k(\alpha_{-h-k}(T_{g+3h+k+s}^*)\alpha_g\rho(T_g))}\\
 &=\epsilon_h(g+k)\sum_{s\in G}A_{g+h}(h+s,h+k)\epsilon_h(g+s)\epsilon_{h+k}(g+h-k+s)\\
&\times T_{g+s}\alpha_k(T_{g+h-k+s}^*\alpha_g\rho(T_g)) \\
 &=A_{g+h}(k,h+k)\epsilon_h(g+k)\epsilon_h(g-h+k)\epsilon_{h+k}(g)\eta_g T_{g-h+k}SS^* \\
 &+\epsilon_h(g+k)\sum_{s,r\in G}A_{g+h}(h+s,h+k)A_g(h-k+s,-k+r)\\
&\times \epsilon_h(g+s)\epsilon_{h+k}(g+h-k+s) T_{g+s}\alpha_k(T_{g+h-2k+s+r}T_{g-k+r}^*) \\
&=A_{g+h}(k,h+k)\epsilon_h(g+k)\epsilon_h(g-h+k)\epsilon_{h+k}(g)\eta_g T_{g-h+k}SS^* \\
 &+\epsilon_h(g+k)\sum_{s,r\in G}A_{g+h}(h+s,h+k)A_g(h-k+s,-k+r)\\
&\times \epsilon_h(g+s)\epsilon_{h+k}(g+h-k+s)\epsilon_{-k}(g+h+s+r)\epsilon_{-k}(g+k+r)\\
&\times T_{g+s}T_{g+h+s+r}T_{g+k+r}^*. \\
\end{align*}

Thus we obtain 
\begin{align}\label{A1'}
\lefteqn{A_g(h,-k)\epsilon_{-k}(g+k)\epsilon_{k-h}(g+h-k)\overline{\eta_{g-k}}\eta_{g+h-k}} \\
 &= A_{g+h}(k,h+k)\epsilon_h(g+k)\epsilon_h(g-h+k)\epsilon_{h+k}(g)\eta_g ,\nonumber
\end{align}
\begin{align}\label{A'}
\lefteqn{\sum_{l\in G}A_g(h,l)A_{g+l}(-h,k+l)\epsilon_l(g+k)\epsilon_l(g-h+k)\epsilon_h(g-h+l)\overline{\eta_{g+h+l}}} \\
 &=\delta_{h+k,0}\epsilon_h(g-h)\epsilon_{-h}(g)\overline{\eta_{g+h}\eta_g}-\frac{\delta_{h,0}\eta_g}{d},\qquad
 \qquad\qquad\qquad \nonumber
\end{align}
\begin{align}\label{A2'}
\lefteqn{\sum_{l\in G}A_g(h,l)A_{g+l}(r,k+l)A_{g+h+l}(h+l+s,h+r)}\\
&\times\epsilon_l(g+k)\epsilon_l(g+k+r)\epsilon_h(g+l+r)\epsilon_{h+l}(g+s)\epsilon_{h+l}(g+h+r+s)\nonumber\\
 &=A_{g+h}(h+s,h+k)A_g(h-k+s,-k+r)\epsilon_h(g+k)\epsilon_h(g+s)\nonumber \\
 &\times \epsilon_{h+k}(g+h-k+s)\epsilon_{-k}(g+h+s+r)\epsilon_{-k}(g+k+r) \nonumber\\
 &-\frac{\delta_{h,0}\delta_{r,0}\eta_g}{d}. \nonumber
\end{align}

Lemma \ref{Frobenius} implies  
$$
A_{g+h}(h,k)=\eta_{g+h}\epsilon_{-k}(g+2h)\epsilon_{-k}(g+h+k)\epsilon_{-k}(g+2h+k)A_{g+h}(-k,h-k),$$
and (\ref{A1'}) is equivalent to 
\begin{align*}
\lefteqn{A_g(h,k)}\\ 
&=A_{g+h}(-k,h-k)\epsilon_k(g-k)\epsilon_{-k-h}(g+h+k)\epsilon_h(g-h-k)\epsilon_{h-k}(g)\epsilon_h(g-k)\\
&\times \eta_g \eta_{g+k}\overline{\eta_{g+h+k}} \\
 &=A_{g+h}(-k,h-k)\epsilon_{h-k}(g+k)\epsilon_{-k}(g+h+k)\epsilon_{-k}(g+2h)\epsilon_h(g)\eta_g \eta_{g+k}
 \overline{\eta_{g+h+k}} \\
 &=A_{g+h}(h,k)\epsilon_{h-k}(g+k)\epsilon_{-k}(g+2h+k)\epsilon_h(g)\eta_g \eta_{g+k}\overline{\eta_{g+h}\eta_{g+h+k}} \\
 &=A_{g+h}(h,k)\epsilon_h(g+k)\epsilon_h(g)\eta_g \eta_{g+k}\overline{\eta_{g+h}\eta_{g+h+k}}. \\
\end{align*}
Therefore the first equation of (\ref{A1}) holds. 
The second equation follows from this and Lemma \ref{Frobenius}. 

Lemma \ref{Frobenius} and (\ref{A1}) imply
\begin{align*}
\lefteqn{A_{g+l}(-h,k+l)}\\
&=A_{g+l}(h+k+l,h)\overline{\eta_{g+l}}\epsilon_h(g-h+l)\epsilon_h(g+k+2l)\epsilon_h(g-h+k+2l) \\
 &=A_{g-h-k}(h+k+l,h)\eta_{g+l}\eta_{g+h+l}\overline{\eta_{g-k}\eta_{g-h-k}}\epsilon_{h+k+l}(g-h-k)\epsilon_{h+k+l}(g-k)\\
&\times \overline{\eta_{g+l}}\epsilon_h(g-h+l)\epsilon_h(g+k+2l)\epsilon_h(g-h+k+2l)\\
&=A_{g-h-k}(h+k+l,h)\eta_{g+h+l}\overline{\eta_{g-k}\eta_{g-h-k}}\epsilon_{k+l}(g-h-k)\epsilon_{k+l}(g-k)\epsilon_h(g-h+l).
 \end{align*}
Thus the left-hand side of (\ref{A'}) is 
\begin{align*}
\lefteqn{\sum_{l\in G}A_g(h,l)\overline{A_{g-h-k}(h,h+k+l)}\epsilon_k(g-k)\epsilon_k(g-h-k)\overline{\eta_{g-h-k}\eta_{g-k}}}\\
 &=(\delta_{h+k,0}-\frac{\delta_{h,0}\eta_g\overline{\eta_{g-h-k}}}{d}) 
 \epsilon_k(g-k)\epsilon_k(g-h-k)\overline{\eta_{g-h-k}\eta_{g-k}}\\
 &=\delta_{h+k,0}\epsilon_{-h}(g+h)\epsilon_{-h}(g)\overline{\eta_g\eta_{g+h}}
 -\frac{\delta_{h,0}\eta_g\overline{\eta_{g-k}^3}}{d}.
\end{align*}
This shows that (\ref{A'}) does not give any new condition.

By the change of variables $r=-x$, $h=x+y$, $k=p$, $s=q-x$ , 
(\ref{A2'}) becomes  
\begin{align}\label{A2"}
\lefteqn{\sum_{l\in G}A_g(x+y,l)A_{g+l}(-x,l+p)A_{g+l+x+y}(l+q+y,y)\epsilon_l(g+p)\epsilon_l(g+p-x)}\\
&\times \epsilon_{x+y}(g+l-x)\epsilon_{l+x+y}(g+q-x)\epsilon_{l+x+y}(g+q-x+y)\qquad\qquad\nonumber\\
&= A_{g+x+y}(q+y,p+x+y)A_g(-p+q+y,-p-x)\nonumber\\
&\times \epsilon_{x+y}(g+p)\epsilon_{x+y}(g+q-x)\epsilon_{p+x+y}(g-p+q+y)\nonumber\\
&\times \epsilon_{-p}(g+q-x+y)\epsilon_{-p}(g+p-x)- \frac{\delta_{x,0}\delta_{y,0}\eta_g}{d}.\nonumber
\end{align}

By (\ref{I}) we get 
\begin{align*}
\lefteqn{A_{g+l}(-x,l+p)} \\
 &=A_{g-p+x}(-x,l+p)\epsilon_{-x}(g+l+x)\epsilon_{l+p}(g-p)\epsilon_{l+p-x}(g-p+x)\eta_{g+l+x}\overline{\eta_{g-p}}.
\end{align*}
By (\ref{A1}) and Lemma \ref{Frobenius}, we get 
\begin{align*}
\lefteqn{A_{g+l+x+y}(l+q+y,y)=\eta_{g+l+x+y}A_{g+l+x+y}(-y,l+q)} \\
&\times \epsilon_{-y}(g+2l+q+x+2y)\epsilon_{-y}(g+l+x+2y)\epsilon_{-y}(g+2l+q+x+3y)\\
&= \eta_{g-q+x+y}\eta_{g-q+x}\overline{\eta_{g+l+x}}A_{g-q+x+y}(-y,l+q)\epsilon_{l+q}(g-q+x+y)\epsilon_{l+q}(g-q+x)\\
&\times \epsilon_{y}(g+2l+q+x)\epsilon_{y}(g+l+x)\epsilon_{y}(g+2l+q+x+y).
\end{align*}
Thus the left-hand side of (\ref{A2"}) is  
\begin{align*}
\lefteqn{\eta_{g+q+x+y}\eta_{g+q+x}\overline{\eta_{g+p}}\sum_{l\in G}A_g(x+y,l)A_{g-p+x}(-x,l+p)A_{g-q+x+y}(-y,l+q)}\\
&\times \epsilon_l(g+p)\epsilon_l(g+p-x)\epsilon_{x+y}(g+l-x)\epsilon_{l+x+y}(g+q-x)\epsilon_{l+x+y}(g+q-x+y)\\
&\times \epsilon_{-x}(g+l+x)\epsilon_{l+p}(g-p)\epsilon_{l+p-x}(g-p+x) \epsilon_{l+q}(g-q+x+y)\epsilon_{l+q}(g-q+x)\\
&\times \epsilon_{y}(g+2l+q+x)\epsilon_{y}(g+l+x)\epsilon_{y}(g+2l+q+x+y)\\
&= \eta_{g-q+x+y}\eta_{g-q+x}\overline{\eta_{g-p}}\epsilon_p(g-p)\epsilon_{p-x}(g-p+x)\\
&\times \sum_{l\in G}A_g(x+y,l)A_{g-p+x}(-x,l+p)A_{g-q+x+y}(-y,l+q)\\
&\times \epsilon_{x+y}(g+l-x)\epsilon_{-x}(g+l+x)\epsilon_{y}(g+l+x)\\
&\times \epsilon_{l+q}(g-q+x+y)\epsilon_{l+q}(g-q+x)\epsilon_{l+x}(g+q-x)\epsilon_{l+x}(g+q-x+y)\\
&= \eta_{g-q+x+y}\eta_{g-q+x}\overline{\eta_{g-p}}\epsilon_p(g-p)\epsilon_{p-x}(g-p+x)
\epsilon_{q-x}(g-q+x+y)\epsilon_{q-x}(g-q+x)\\
&\times \sum_{l\in G}A_g(x+y,l)A_{g-p+x}(-x,l+p)A_{g-q+x+y}(-y,l+q)\\
\end{align*}

By (\ref{A1}) and Lemma \ref{Frobenius}, we have 
\begin{align*}
\lefteqn{A_{g+x+y}(q+y,p+x+y)=A_{g-p}(q+y,p+x+y)} \\
 &\times\eta_{g-p}\eta_{g-p+q+y}\overline{\eta_{g+x+y}\eta_{g+q+x+2y}}\epsilon_{p+x+y}(g-p)\epsilon_{p+x+y}(g-p+q+y),
\end{align*}
\begin{align*}
\lefteqn{A_g(-p+q+y,-p-x)= A_g(p+x,q+x+y)} \\
 &\times\eta_g\epsilon_{p+x}(g-p+q+y)\epsilon_{p+x}(g-p-x)\epsilon_{p+x}(g-2p+q-x+y).
\end{align*}
Thus the right-hand side of (\ref{A2"}) is 
\begin{align*}
\lefteqn{\eta_g\eta_{g+p}\eta_{g+p+q+y}\overline{\eta_{g+x+y}\eta_{g+q+x}} A_{g-p}(q+y,p+x+y)A_g(p+x,q+x+y)} \\ 
&\times \epsilon_{x+y}(g+p)\epsilon_{x+y}(g+q-x)\epsilon_{p+x+y}(g-p+q+y)\\
&\times \epsilon_{-p}(g+q-x+y)\epsilon_{-p}(g+p-x)  \epsilon_{p+x+y}(g-p)\epsilon_{p+x+y}(g-p+q+y)\\
&\times \epsilon_{p+x}(g-p+q+y)\epsilon_{p+x}(g-p-x)\epsilon_{p+x}(g-2p+q-x+y)
- \frac{\delta_{h,0}\delta_{r,0}\eta_g}{d}\\
&= \eta_g\eta_{g+p} \eta_{g+p+q+y}\overline{\eta_{g+x+y}\eta_{g+q+x}}
A_{g-p}(q+y,p+x+y)A_g(p+x,q+x+y) \\ 
&\times    \epsilon_p(g-p)\epsilon_{p+x}(g-p+q+y) 
 \epsilon_x(g+p-x)\epsilon_x(g+q-x+y) \epsilon_{x+y}(g+q-x)\\
& -\frac{\delta_{h,0}\delta_{r,0}\eta_g}{d}.\\
\end{align*}

Therefore (\ref{A2'}) is equivalent to 
\begin{align*}
\lefteqn{
\sum_{l\in G}A_g(x+y,l)A_{g-p+x}(-x,l+p)A_{g-q+x+y}(-y,l+q)} \\
&= \eta_g\eta_{g+p} \eta_{g+p+q+y}\overline{\eta_{g+x+y}\eta_{g+q+x}}\eta_{g+p}\overline{\eta_{g+q+x+y}\eta_{g+q+x}}\\
&\times A_{g-p}(q+y,p+x+y)A_g(p+x,q+x+y) \\ 
&\times     \epsilon_p(g-p)\epsilon_{p+x}(g-p+q+y) \epsilon_x(g+p-x)\epsilon_x(g+q-x+y) \epsilon_{x+y}(g+q-x)\\
&\times \epsilon_p(g-p)\epsilon_{p-x}(g-p+x)\epsilon_{q-x}(g-q+x+y)\epsilon_{q-x}(g-q+x)\\
&-\frac{\delta_{x,0}\delta_{y,0}}{d}\eta_g\eta_{g+p}\overline{\eta_{g+q}^2}\\
&=\eta_g\eta_{g+q+x}\eta_{g+p+q+y}\overline{\eta_{g+p}\eta_{g+x+y}\eta_{g+q+x+y}}\\
&\times A_{g-p}(q+y,p+x+y)A_g(p+x,q+x+y) \\ 
&\times \epsilon_p(g-p+x)\epsilon_{p+x}(g-p+q+y)\epsilon_q(g-q+x+y)\epsilon_{q+y}(g-q+x)\\
&-\frac{\delta_{x,0}\delta_{y,0}}{d}\eta_g\eta_{g+p}\eta_{g+q},\\
\end{align*}
and (\ref{A2}) holds. 
\end{proof}

\begin{remark} (\ref{I}) follows from (\ref{A1}). 
\end{remark}
\section{Reconstruction}\label{Reconstruction}
In this section, we discuss how to recover the C$^*$-generalized Haagerup category from a solution 
of the polynomial equations we got in the previous section. 

Assume that $G$ is a finite abelian group of order $n$. 
We set $d=\frac{n+\sqrt{n^2+4}}{2}$. 
We consider $\epsilon_h(g)\in \{1,-1\}$, $\eta_g\in \T$, $A_g(h,k)\in \C$ satisfying the following condition: 
\begin{equation}\label{cocycle}\epsilon_{h+k}(g)=\epsilon_h(g)\epsilon_k(g+2h),\quad \epsilon_h(0)=1,
\end{equation}
\begin{equation}\label{R1}
\eta_{g+2h}=\eta_g,\quad \eta_g^3=1,
\end{equation}
\begin{equation}\label{O1}
\sum_{h\in G}A_g(h,0)=-\frac{\overline{\eta_g}}{d},
\end{equation}
\begin{equation}\label{O2}
\sum_{h\in G}A_g(h-g,k)\overline{A_{g'}(h-g',k)}=\delta_{g,g'}-\frac{\overline{\eta_g}\eta_{g'}}{d}\delta_{k,0},
\end{equation}
\begin{equation}\label{2hshift}
A_{g+2h}(p,q)=\epsilon_h(g)\epsilon_h(g+p)\epsilon_h(g+q)\epsilon_h(g+p+q)A_g(p,q),
\end{equation}
\begin{equation}\label{CC}
A_g(k,h)=\overline{A_g(h,k)},
\end{equation}
\begin{align}\label{R2}
A_g(h,k)&=A_g(-k,h-k)\eta_g\epsilon_{-k}(g+h)\epsilon_{-k}(g+k)\epsilon_{-k}(g+h+k)\\
&=A_g(k-h,-h)\overline{\eta_g}\epsilon_{-h}(g+h)\epsilon_{-h}(g+k)\epsilon_{-h}(g+h+k),\nonumber
\end{align}
\begin{align}\label{hkshift}
A_g(h,k)&=A_{g+h}(h,k)\eta_g\eta_{g+k}\overline{\eta_{g+h}}\overline{\eta_{g+h+k}}\epsilon_h(g)\epsilon_h(g+k)\\
&=A_{g+k}(h,k)\overline{\eta_g\eta_{g+h}}\eta_{g+k}\eta_{g+h+k}\epsilon_k(g)\epsilon_k(g+h),\nonumber
\end{align}
\begin{align}\label{AAA}
\lefteqn{
\sum_{l\in G}A_g(x+y,l)A_{g-p+x}(-x,l+p)A_{g-q+x+y}(-y,l+q)} \\
&=A_g(p+x,q+x+y)A_{g-p}(q+y,p+x+y)\nonumber\\
&\times \eta_g\eta_{g+q+x}\eta_{g+p+q+y}\overline{\eta_{g+p}\eta_{g+x+y}\eta_{g+q+x+y}}\nonumber \\ 
&\times \epsilon_p(g-p+x)\epsilon_{p+x}(g-p+q+y)\epsilon_q(g-q+x+y)\epsilon_{q+y}(g-q+x)\nonumber\\
&-\frac{\delta_{x,0}\delta_{y,0}}{d}\eta_g\eta_{g+p}\eta_{g+q}.\nonumber
\end{align}

We denote by $\cO_{n+1}$ the Cuntz algebra with the canonical generators $\{S\}\cup\{T_g\}_{g\in G}$. 
We can introduce a $G$-action $\alpha$ on $\cO_{n+1}$ and an endomorphism $\rho$ of $\cO_{n+1}$ 
satisfying $\alpha_g\rho=\rho\alpha_{-g}$ by 
$$\alpha_h(S)=S,\quad \alpha_h(T_g)=\epsilon_h(g)T_{g+2h},$$
$$\rho(S)=\frac{1}{d}S+\frac{1}{\sqrt{d}}\sum_{g\in G}T_gT_g,$$
$$\alpha_g\rho(T_g)=\eta_gT_gSS^*+\frac{\overline{\eta_g}}{\sqrt{d}}ST_g^*+
\sum_{h,k\in G}A_g(h,k)T_{g+h}T_{g+h+k}T_{g+k}^*.$$

\begin{theorem} $S\in (\id,\rho^2)$, $T_g\in (\alpha_g\rho,\rho^2)$. 
\end{theorem}

\begin{proof} Direct computation shows $S^*\rho^2(S)S=S$, and the proof of Lemma \ref{id} shows $S^*\rho^2(T_g)S=T_g$. 
Thus Lemma \ref{CE} implies $S\in (\id,\rho^2)$. 

Direct computation shows 
$$\rho(S^*\rho(S))T_{g+k}=\alpha_{g+k}\rho(S)^*T_{g+k}\alpha_{g+k}\rho(S),$$
$$\alpha_g\rho(T_g^*\alpha_g\rho(S))T_{g+k}=\alpha_g\rho(T_g)^*T_{g+k}\alpha_{g+k}\rho(S),$$
$$\alpha_g\rho(S^*\alpha_g\rho(T_g))T_{g+k}=\alpha_g\rho(S^*)T_{g+k}\alpha_{g+k}\rho(T_g),$$
and the proof of Lemma \ref{AAA=AA} shows  
$$\alpha_g\rho(T_{g+h}^*\alpha_g\rho(T_g))T_{g+k}=\alpha_g\rho(S^*)T_{g+k}\alpha_{g+k}\rho(T_g).$$
Thus we have
\begin{align*}
\lefteqn{\rho^2(S)T_{h}=\rho(SS^*+\sum_{g\in G}T_{-g}T_{-g}^*)\rho^2(S)T_h} \\
 &=\rho(S)\rho(S^*\rho(S))T_h+\sum_{g\in G}\alpha_g\rho(T_gT_g^*\alpha_g\rho(S))T_h\\
 &=\rho(S)\rho(S^*)T_h\alpha_h\rho(S)+\sum_{g\in G}\alpha_g\rho(T_gT_g^*)T_h\alpha_h\rho(S)\\
 &=\rho(SS^*+\sum_{g\in G}T_{-g}T_{-g}^*)T_h\alpha_h\rho(S)=T_h\alpha_h\rho(S),
\end{align*}
\begin{align*}
\lefteqn{\rho^2(T_g)T_{g+k} =\alpha_g\rho(SS^*+\sum_{h\in G}T_{g+h}T_{g+h}^*) \rho^2(T_g)T_{g+k}} \\
&=\alpha_g\rho(SS^*\alpha_g\rho(T_g))T_{g+k}+\sum_{h\in G}\alpha_g\rho(T_{g+h}T_{g+h}^*\alpha_g\rho(T_g))T_{g+k}\\
 &=\alpha_g\rho(SS^*)T_{g+k}\alpha_{g+k}\rho(T_g)+\sum_{h\in G}\alpha_g\rho(T_{g+h}T_{g+h}^*)T_{g+k}\alpha_{g+k}\rho(T_g)\\
 &=T_{g+k}\alpha_{g+k}\rho(T_g),
\end{align*}
and Lemma \ref{CE} implies $T_{g+k}\in (\alpha_{g+k}\rho,\rho^2)$. 
\end{proof}

As in \cite{I93} and \cite[Appendix]{I15}, we introduce a weighted gauge action $\gamma$ on $\cO_{n+1}$ 
by $\gamma_t(S)=e^{2i t}S$, and $\gamma_t(T_g)=e^{i t}T_g$. 
Then $\gamma$ commutes with $\alpha_g$ and $\rho$. 
There exists a unique KMS-state for $\gamma$, and $\alpha_g$ and $\rho$ extend to the weak closure of $\cO_{n+1}$ 
in the GNS representation of the KMS state, which is the hyperfinite type III$_{1/d}$ factor. 
Taking the tensor product with the hyperfinite type III$_1$ factor, we get $\alpha$ and $\rho$ acting 
on the the hyperfinite type III$_1$ factor. 

Let $N=\{h\in G_2|\; \chi_g(h)=1\,\quad \forall g\in G\}$. 
If $N=\{0\}$, then we can recover a C$^*$-generalized Haagerup category having the same solution 
of the polynomial equations with the original category $\cC$ in this way. 
If $N\neq \{0\}$, the action $\alpha$ is not faithful, and the resulting fusion category is 
de-equivariantization of the original category $\cC$ by $N$. 
In Appendix, we will give a free product trick to recover the original category in this case.

\section{The classification of generalized Haagerup categories}\label{classifiction}
In this section, we deduce a complete classification invariant for C$^*$-generalized Haagerup categories 
with a fixed finite abelian group $G$ following our observation in Remark \ref{choices}. 
Throughout this section, we make the same assumption as in Section \ref{PEGHC}. 

We start from an easy case. 

\begin{lemma} If $H^2(G,\T)$ is trivial, there exists a one-to-one correspondence between 
the equivalence classes of C$^*$-generalized Haagerup categories with a finite abelian group 
$G$ and a distinguished simple object $\rho$,  
and the $\Aut(G)$-orbits of the gauge equivalence classes of the solutions $(\epsilon_{h}(g),\eta_g,A_g(h,k))$ 
of the polynomial equations Eq.(\ref{cocycle})-Eq.(\ref{AAA}). 
\end{lemma}

\begin{proof} Thanks to Theorem \ref{uniqueness} and Remark \ref{choices}, it suffices to show that 
two different standard liftings of a generalized Haagerup category $\cC\subset \End_0(M)$ 
give the same gauge equivalence class of the solutions. 
Assume that $[\rho,\alpha]$ and $[\rho',\alpha']$ are standard liftings of $\cC$ with $[\rho]=[\rho']$ and 
$[\alpha_g]=[\alpha'_g]$. 
Then there exist unitaries $U_g\in \cU(M)$ and $V\in \cU(M)$ satisfying 
$\rho'=\Ad V\circ \rho$ and $\alpha'_g=\Ad U_g\circ \rho$. 
Since $H^2(G,\T)$ is trivial, we may further assume that $\{U_g\}_{g\in G}$ is an $\alpha$-cocycle. 
Since it is a coboundary, the two actions $\alpha$ and $\alpha'$ are inner conjugate. 
Since inner conjugate standard liftings give gauge equivalent triplets, by conjugating $[\rho',\alpha']$ 
by an inner automorphism, we may assume $\alpha'=\alpha$. 

Since $\alpha_g\circ \rho=\rho\circ \alpha_{-g}$ and $\alpha_g\circ \rho'=\rho'\circ \alpha_{-g}$, 
we see that $\alpha_g(V)$ is proportional to $V$, and there exists a character $\chi\in \widehat{G}$ 
satisfying $\alpha_g(V)=\chi(g)V$. 
Since $V\rho(V)T_0V^*\in (\rho',\rho'^2)$ and $[\rho',\alpha]$ is a standard lifting, the character is 
trivial on $G_2$. 
From this condition, we can choose a character $\chi_0\in \widehat{G}$ satisfying $\chi_0^2=\chi$ 
(use the fundamental theorem of finitely generated abelian groups). 

Let $S'=V\rho(V)S\in (\id_M,\rho'^2)$. 
Then the two anti-unitaries on $(\alpha_g\rho',\rho'^2)$ take the forms 
$$j'_{1,g}T'=\sqrt{d}T'^*\rho'(S'),$$ 
$$j'_{2,g}T'=\sqrt{d}\alpha_g\rho'(T')^*S'.$$ 
Let $T'_g=\chi_0(g)^{-1}V\rho(V)T_gV^{-1}\in (\alpha_g\rho',\rho'^2)$. 
Then we have $j'_{1,g}T_g'=T'_g$, $\alpha_h(T'_g)=\epsilon_h(g)T'_{g+2h}$, and $j_{2,g}T'_g=\eta_gT'_g$. 
Moreover, 
\begin{align*}
\lefteqn{\alpha_g\rho'(T'_g)=\chi_0(g)^{-1}V\alpha_g\rho(V\rho(V)T_gV^{-1})V^{-1}} \\
 &=\chi_0(g)V\rho(V)\rho^2(V)\alpha_g\rho(T_g)\rho(V^{-1})V^{-1} \\
 &=\chi_0(g)V\rho(V)\rho^2(V) 
 \big(\eta_gT_gSS^*+\frac{\overline{\eta_g}}{\sqrt{d}}ST_g^*+
\sum_{h,k\in G}A_g(h,k)T_{g+h}T_{g+h+k}T_{g+k}^*\big)\\
&\times \rho(V^{-1})V^{-1} \\
 &=\eta_gT'_gS'S'^*+\frac{\overline{\eta_g}}{\sqrt{d}}S'{T'_g}^*+
\sum_{h,k\in G}A_g(h,k)T'_{g+h}T'_{g+h+k}{T'_{g+k}}^*.
\end{align*}
This shows the statement. 
\end{proof}

Now we proceed to the general case. 
The above computation shows that if $[\rho',\alpha']$ is another standard lifting with 
$[\rho']=[\rho]$ and $\alpha'$ inner conjugate to $\alpha$, 
then the change from $[\rho,\alpha]$ to $[\rho',\alpha']$ leave the gauge equivalence class of the triplet 
$(\epsilon_{h}(g),\eta_g,A_g(h,k))$ invariant. 
This means that in general, the effect of replacing $[\rho,\alpha]$ with $[\rho',\alpha']$ satisfying 
$[\rho]=[\rho']$ and $[\alpha'_g]=[\alpha_g]$ depends only on 
the cohomology class in $H^2(G,\T)$ determined by the difference of $\alpha'$ from $\alpha$.  
We give an explicit description of the action of the cohomology group $H^2(G,\T)$ on the gauge equivalence 
classes of the solutions of the polynomial equations. 

Recall that a cocycle $\omega\in Z^2(G,\T)$ is normalized if it satisfies 
$\omega(g,0)=\omega(0,g)=\omega(g,-g)=1$ for all $g\in G$. 
The cocycle relation implies that any normalized cocycle satisfies $\omega(g,h)^{-1}=\omega(-h,-g)$. 
It is known that every cohomology class in $H^2(G,\T)$ is represented by a normalized cocycle. 

\begin{lemma} For a finite abelian group $G$, 
every cohomology class in $H^2(G,\T)$ is represented by a normalized cocycle $\omega$ satisfying 
\begin{equation}\label{2cocycle}
\omega(g,h)\omega(h,g)=\omega(g,h)\overline{\omega(-g,-h)}=1. 
\end{equation}
\end{lemma}

\begin{proof} 
Since the group automorphism $G\ni g\mapsto -g\in G$ acts on $H^2(G,\T)$ trivially, 
the two cocycles $\omega(g,h)$ and $\omega(-g,-h)$ 
are cohomologous, and there exists $\mu:G\rightarrow \T$ satisfying 
$$\omega(g,h)\omega(h,g)=\omega(g,h)\overline{\omega(-g,-h)}=\mu(g)\mu(h)\overline{\mu(g+h)}.$$
Since the left hand side is normalized, we have $\mu(0)=\mu(g)\mu(-g)=1$. 
Note that $\mu$ restricted to $G_2$ is a character, 
and there exists a character $\chi\in \widehat{G}$ extending $\mu|_{G_2}$. 
By replacing $\mu(g)$ with $\mu(g)\overline{\chi(g)}$ if necessary, 
we may and do assume $\mu(z)=1$ for all $z\in G_2$. 
Thus we can choose a square root $\mu(g)^{1/2}$ for each $g\in G$ satisfying 
$\mu(0)^{1/2}=\mu(g)^{1/2}\mu(-g)^{1/2}=1$. 
Replacing $\omega(g,h)$ with $\omega(g,h)\overline{\mu(g)^{1/2}\mu(h)^{1/2}}\mu(g+h)^{1/2}$, 
we get a normalized cocycle satisfying Eq.(\ref{2cocycle}).  
\end{proof}

In what follows, we assume that $\omega$ is a normalized cocycle satisfying Eq.(\ref{2cocycle}). 
We choose unitaries $\{U_g\}_{g\in G}$ in $M$ satisfying $U_g\alpha_g(U_h)=\omega(g,h)U_{g+h}$ with $U_0=1$, 
and set $\beta_g=\Ad U_g\circ \alpha_g$. 
Then $\beta$ is an outer action of $G$ on $M$, and we have
$$\rho\circ \beta_{-g}=\Ad(\rho(U_{-g})U_g^{-1})\circ \beta_g\circ \rho.$$
Since 
\begin{align*}
\lefteqn{\rho(U_{-g})U_g^{-1}\beta_g(\rho(U_{-h})U_h^{-1})
 =\rho(U_{-g})\alpha_g(\rho(U_{-h})U_h^{-1})U_g^{-1}}\\
 &=\rho(U_{-g}\alpha_{-g}(U_{-h}))\alpha_g(U_h^{-1})U_g^{-1}
 =\omega(-g,-h)\overline{\omega(g,h)}\rho(U_{-g-h})U_{g+h}^{-1}\\
 &=\rho(U_{-g-h})U_{g+h}^{-1},
\end{align*}
the family $\{\rho(U_{-g})U_g^{-1}\}_{g\in G}$ forms a $\beta$-cocycle. 
Thus there exists a unitary $V\in \cU(M)$ satisfying $\rho(U_{-g})U_g^{-1}=V^{-1}\beta_g(V)$,  
and so $U_g\alpha_g(V)=V\rho(U_{-g})$. 
Let $\sigma=\Ad V\circ \rho$. 
Then we have $\beta_g\circ \sigma=\sigma\circ \beta_{-g}$. 

\begin{lemma}\label{newstandard} The pair $[\sigma,\beta]$ is a standard lifting. 
\end{lemma}

\begin{proof} It suffices to show that $G_2$ acts on $(\sigma,\sigma^2)$ trivially. 
Note that we have $V\rho(V)T_0V^{-1}\in (\sigma,\sigma^2)$, and for $z\in G_2$ we have 
\begin{align*}
\lefteqn{\beta_h(V\rho(V)T_0V^{-1})=U_h\alpha_h(V\rho(V)T_0V^{-1})U_h^{-1}}\\
&=U_h\alpha_h(V)\rho(\alpha_{-h}(V))T_{2h}\alpha_h(V^{-1})U_h^{-1}
=V\rho(U_{-h})\rho(U_{-h}^{-1}V\rho(U_h))T_{2h}\rho(U_{-h}^{-1})V^{-1}\\
&=V\rho(V)\rho^2(U_h)T_{2h}\rho(U_{-h}^{-1})V^{-1}
=V\rho(V)T_{2h}\rho(\alpha_{-2h}(U_{h})U_{-h}^{-1})V^{-1}\\
&=V\rho(V)T_{2h}\rho(U_{-2h}^{-1}U_{-2h}\alpha_{-2h}(U_{h})U_{-h}^{-1})V^{-1}
=\omega(-2h,h)V\rho(V)T_{2h}\rho(U_{-2h}^{-1})V^{-1}.\\
 &=\omega(-2h,h)V\rho(V)T_{2h}\alpha_{2h}(V^{-1})U_{2h}^{-1}.\\
\end{align*}
Thus for $z\in G_2$ we get $\beta_h(V\rho(V)T_0V^{-1})=V\rho(V)T_0V^{-1}$.
\end{proof}

\begin{remark} Thanks to the cocycle relation and normalization, we have 
$$\omega(-h,-h)\omega(-2h,h)=\omega(-h,h)\omega(-h,0)=1,$$ 
and we have $\omega(-2h,h)=\omega(h,h)$. 
Thanks to Eq.(\ref{2cocycle}) and normalization, we have $\omega(h,h)=\omega(-h,-h)\in \{1,-1\}$. 
The above computation implies 
$$\beta_h(V\rho(V)T_0V^{-1})=\omega(h,h) V\rho(V)T_{2h}\alpha_{2h}(V^{-1})U_{2h}^{-1}.$$
Since the left-hand side depends only on the class $h+G_2$, so does $\omega(h,h)$. 
Thus we can choose $\mu:G\to \{1,-1\}$ satisfying $\mu(2h)=\omega(h,h)$.  
\end{remark}

Let $S'=V\rho(V)S$. 
Then $S'$ is an isometry in $(\id,\sigma^2)$, and we define two anti-unitaries on $(\beta_g\sigma,\sigma^2)$ by 
$$j'_{1,g}T=\sqrt{d}T^*\sigma(S')=\sqrt{d}T^*V\rho(V\rho(V)S)V^*,$$
$$j'_{2,g}T=\sqrt{d}\beta_g\sigma(T^*)S'=\sqrt{d}U_g\alpha_g(V)\alpha_g\rho(T^*)\alpha_g(V^{-1})U_g^{-1}V\rho(V)S.$$
Let $T'_g=\mu(g)V\rho(V)T_g\alpha_g(V^{-1})U_g^{-1}$.  
Then $T'_g$ is an isometry in $(\beta_g\sigma,\sigma^2)$ satisfying 
$\beta_g(T'_0)=T'_{2g}$.

\begin{lemma} With the above notation, we have $j'_{1,g}T'_g=T'_g$. 
\end{lemma}

\begin{proof}
We have  
\begin{align*}
j'_{1,g}T'_g&=\mu(g)\sqrt{d} U_g\alpha_g(V)T_g^*\rho(V^{-1})V^{-1} V\rho(V\rho(V)S)V^{-1}\\
 &=\mu(g)\sqrt{d} U_g\alpha_g(V)T_g^*\rho^2(V)\rho(S)V^{-1}
 =\mu(g)\sqrt{d} U_g\alpha_g(V)\rho(\alpha_{-g}(V))T_g^*\rho(S)V^{-1}\\
 &=\mu(g)V\rho(U_{-g})\rho(U_{-g}^{-1}V\rho(U_g))T_gV^{-1}
 =\mu(g)V\rho(V)\rho^2(U_g)T_gV^{-1}\\
 &=\mu(g)V\rho(V)T_g\rho(\alpha_{-g}(U_g))  V^{-1},
\end{align*}
and
$$\rho(\alpha_{-g}(U_g))  V^{-1}=\rho(U_{-g}^{-1}U_{-g}\alpha_{-g}(U_g))V^{-1}
= \omega(-g,g)\rho(U_{-g}^{-1})V^{-1}=\alpha_g(V^{-1})U_g^{-1},$$
which shows $j'_{1,g}T'_g=T'_g$. 
\end{proof}

Now we define $\epsilon'_h(g)$ and $A'_g(h,k)$ by $\beta_h(T_g')=\epsilon'_h(g)T'_{g+2h}$ and 
$$\beta_g\circ\sigma(T'_g)=\eta_gT'_gS'{S'}^*+\frac{\overline{\eta_g}}{\sqrt{d}}S'{T'}_g^*+
\sum_{h,k\in G}A'_g(h,k)T'_{g+h}T'_{g+h+k}{T'}_{g+k}^*.$$

\begin{theorem}\label{H2action} The action of the cohomology group $H^2(G,\T)$ on the gauge equivalence classes 
of the solutions $(\epsilon_{h}(g),\eta_g,A_g(h,k))$ of the polynomial equations Eq.(\ref{cocycle})-Eq.(\ref{AAA}) 
is given as follows. 
Let $\omega\in Z^2(G,\T)$ be a normalized cocycle satisfying Eq.(\ref{2cocycle}).
We choose $\mu:G\to\{1,-1\}$ satisfying $\mu(2g)=\omega(g,g)$ for any $g\in G$. 
Then $[\omega]$ transforms $[(\epsilon_{h}(g),\eta_g,A_g(h,k))]$ to $[(\epsilon'_{h}(g),\eta_g,A'_g(h,k))]$ 
with 
\begin{align*}
\lefteqn{\epsilon'_h(g)=\epsilon_h(g)\mu(g)\mu(g+2h)\overline{\omega(h,g)}\overline{\omega(g+h,h)}} \\
 &=\epsilon_h(g)\mu(g)\mu(g+2h)\mu(2h) b_\omega(g,h)\overline{\omega(g,2h)}.
\end{align*}
$$A'_g(h,k)=A_g(h,k)\mu(g+k)\mu(g)\mu(g+h+k)\mu(g+h)\overline{\omega(g+k,h)}\overline{\omega(h,g)},$$
where $b_\omega(g,h)=\omega(g,h)\overline{\omega(h,g)}$ is the antisymmetric bicharacter associated with 
the 2-cocycle $\omega$. 
In particular, we have $\epsilon'_z(g)=\epsilon_z(g)b_\omega(g,z)$ for $z\in G_2$.
\end{theorem}

\begin{proof} 
We first compute $\beta_h(T'_g)$ as 
\begin{align*}
 \beta_h(T'_g)&=\mu(g)U_h\alpha_h(V\rho(V)T_g\alpha_g(V^{-1})U_g^{-1})U_h^{-1} \\
 &=\mu(g)\epsilon_h(g)U_h\alpha_h(V)\rho(\alpha_{-h}(V))T_{g+2h}\alpha_{g+h}(V^{-1})\alpha_h(U_g^{-1})U_h^{-1} \\
 &=\mu(g)\epsilon_h(g)V\rho(U_{-h})\rho(U_{-h}^{-1}V\rho(U_h))T_{g+2h}\alpha_{g+h}(V^{-1})
 \overline{\omega(h,g)}U_{g+h}^{-1} \\
 &=\mu(g)\epsilon_h(g)\overline{\omega(h,g)}V\rho(V)T_{g+2h}\alpha_{g+2h}\circ\rho(U_h)\alpha_{g+h}(V^{-1})
 U_{g+h}^{-1}\\
 &=\mu(g)\epsilon_h(g)\overline{\omega(h,g)}V\rho(V)T_{g+2h}\alpha_{g+h}(\rho(\alpha_{-h}(U_h))V^{-1}) U_{g+h}^{-1}.
\end{align*}
Here we have 
\begin{align*}
\lefteqn{\alpha_{g+h}(\rho(\alpha_{-h}(U_h))V^{-1}) U_{g+h}^{-1}
=\alpha_{g+h}(\rho(U_{-h}^{-1}U_{-h}\alpha_{-h}(U_h))V^{-1}) U_{g+h}^{-1} } \\
 &=\omega(-h,h)\alpha_{g+h}(\rho(U_{-h}^{-1})V^{-1}) U_{g+h}^{-1} 
 =\alpha_{g+h}(\alpha_h(V^{-1})U_h^{-1})) U_{g+h}^{-1}\\
 &=\overline{\omega(g+h,h)}\alpha_{g+2h}(V^{-1})U_{g+2h}^{-1}.
\end{align*}
This shows 
$$\beta_h(T'_g)=\epsilon_h(g)\mu(g)\overline{\omega(h,g)}\overline{\omega(g+h,h)}V\rho(V)T_{g+2h}\alpha_{g+2h}(V^{-1})U_{g+2h}^{-1},$$
and so 
\begin{align*}
\lefteqn{\epsilon'_h(g)=\epsilon_h(g)\mu(g)\mu(g+2h)\overline{\omega(h,g)}\overline{\omega(g+h,h)}} \\
 &=\epsilon_h(g)\mu(g)\mu(g+2h)b_\omega(g,h)\overline{\omega(g,h)}\overline{\omega(g+h,h)} \\
 &=\epsilon_h(g)\mu(g)\mu(g+2h)b_\omega(g,h)\overline{\omega(h,h)}\overline{\omega(g,2h)}.
\end{align*}

For $A'_g(h,k)$, we have  
\begin{align*}
\lefteqn{{T'_{g+h}}^*\beta_g\circ \sigma(T'_g)T'_{g+k} }\\
 &=\mu(g+h)\mu(g+k)U_{g+h}\alpha_{g+h}(V)T_{g+h}^*\Ad(\rho(V^{-1})V^{-1}U_g\alpha_g(V))\circ\alpha_g\circ
 \rho(T'_g)\\
 &\times T_{g+k}\alpha_{g+k}(V^{-1})U_{g+k}^{-1} \\
 &=\mu(g+h)\mu(g+k)U_{g+h}\alpha_{g+h}(V)T_{g+h}^*\\
 &\times \Ad(\rho(V^{-1}U_{-g}))\circ\rho(\alpha_{-g}(T'_g))
 T_{g+k}\alpha_{g+k}(V^{-1})U_{g+k}^{-1} \\
 &=\mu(g+h)\mu(g+k)U_{g+h}\alpha_{g+h}(V)T_{g+h}^*\rho(V^{-1}U_{-g}\alpha_{-g}(T'_g)U_{-g}^{-1}V)
 T_{g+k}\alpha_{g+k}(V^{-1})U_{g+k}^{-1}.
\end{align*}
Here we have 
\begin{align*}
\lefteqn{V^{-1}U_{-g}\alpha_{-g}(T'_g)U_{-g}^{-1}V
=\mu(g)V^{-1}U_{-g}\alpha_{-g}(V\rho(V)T_g\alpha_g(V^{-1})U_g^{-1})U_{-g}^{-1}V} \\
 &=\mu(g)V^{-1}U_{-g}\alpha_{-g}(V)\rho(\alpha_g(V))\alpha_{-g}(T_g)V^{-1}\alpha_{-g}(U_g^{-1})U_{-g}^{-1}V \\
 &=\mu(g)\overline{\omega(-g,g)}V^{-1}U_{-g}\alpha_{-g}(V)\rho(\alpha_g(V))\alpha_{-g}(T_g)\\
 &=\mu(g)\rho(U_g)\rho(U_g^{-1}V\rho(U_{-g}))\alpha_{-g}(T_g)\\
 &=\mu(g)\rho(V\rho(U_{-g}))\alpha_{-g}(T_g)=\mu(g)\rho(U_g\alpha_g(V))\alpha_{-g}(T_g).
\end{align*}
Thus we get 
\begin{align*}
\lefteqn{{T'_{g+h}}^*\beta_g\circ \sigma(T'_g)T'_{g+k}=\mu(g+h)\mu(g+k)\mu(g)} \\
 &\times U_{g+h}\alpha_{g+h}(V)T_{g+h}^*\rho^2(V\rho(U_{-g}))\alpha_g\circ\rho(T_g)T_{g+k}\alpha_{g+k}(V^{-1})U_{g+k}^{-1}\\
 &=\mu(g+h)\mu(g+k)\mu(g)\\ 
 &\times U_{g+h}\alpha_{g+h}(V)\alpha_{g+h}\circ\rho(V\rho(U_{-g}))T_{g+h}^*\alpha_g\circ\rho(T_g)T_{g+k}\alpha_{g+k}(V^{-1})U_{g+k}^{-1}\\
 &=\mu(g+h)\mu(g+k)\mu(g)A_g(h,k)\\ 
 &\times U_{g+h}\alpha_{g+h}(V)\rho(\alpha_{-g-h}(V))\rho^2(\alpha_{g+h}(U_{-g})) T_{g+h+k}
 \alpha_{g+k}(V^{-1})U_{g+k}^{-1}\\
 &=\mu(g+h)\mu(g+k)\mu(g)A_g(h,k)\\ 
 &\times V\rho(U_{-g-h})\rho(U_{-g-h}^{-1}V\rho(U_{g+h}))\rho^2(\alpha_{g+h}(U_{-g})) T_{g+h+k}\alpha_{g+k}(V^{-1})U_{g+k}^{-1}\\
 &= \mu(g+h)\mu(g+k)\mu(g)\omega(g+h,-g)A_g(h,k)\\
 &\times V\rho(V)\rho^2(U_h) T_{g+h+k}\alpha_{g+k}(V^{-1})U_{g+k}^{-1}\\
 &= \mu(g+h)\mu(g+k)\mu(g)\omega(g+h,-g)A_g(h,k)\\
 &\times V\rho(V)T_{g+h+k}\alpha_{g+h+k}\circ\rho(U_h)\alpha_{g+k}(V^{-1})U_{g+k}^{-1}.
\end{align*}
Here we have 
\begin{align*}\lefteqn{
\alpha_{g+h+k}\circ\rho(U_h)\alpha_{g+k}(V^{-1})U_{g+k}^{-1}
 =\alpha_{g+h+k}(V^{-1}U_{-h}\alpha_{-h}(V))\alpha_{g+k}(V^{-1})U_{g+k}^{-1}   }\\
 &=\alpha_{g+h+k}(V^{-1})\alpha_{g+h+k}(U_{-h})U_{g+k}^{-1}
 =\alpha_{g+h+k}(V^{-1})\alpha_{g+k}(U_h^{-1}U_h\alpha_h(U_{-h}))U_{g+k}^{-1}\\
 &=\omega(h,-h)\alpha_{g+h+k}(V^{-1})\alpha_{g+k}(U_h^{-1})U_{g+k}^{-1}
 =\overline{\omega(g+k,h)}\alpha_{g+h+k}(V^{-1})U_{g+h+k}^{-1}.
\end{align*}
Therefore we get 
\begin{align*}
\lefteqn{{T'_{g+h}}^*\beta_g\circ \sigma(T'_g)T'_{g+k}} \\
&=\mu(g+h)\mu(g+k)\mu(g)\omega(g+h,-g)\overline{\omega(g+k,h)}A_g(h,k)\\
 &\times V\rho(V)T_{g+h+k}\alpha_{g+h+k}(V^{-1})U_{g+h+k}^{-1}\\
&=\mu(g+h+k)\mu(g+h)\mu(g+k)\mu(g)\omega(g+h,-g)\overline{\omega(g+k,h)}A_g(h,k)T'_{g+h+k}.
\end{align*}
The cocycle relation implies $\omega(h,g)\omega(h+g,-g)=\omega(g,-g)=1$, and 
we get the statement. 
\end{proof}

In summary we obtain the following theorem. 

\begin{theorem}\label{distinguished} The equivalence classes of C$^*$-generalized Haagerup categories with a finite abelian group 
$G$ and a distinguished object $\rho$ are in one-to-one correspondence with the $H^2(G,\T)\rtimes \Aut(G)$-orbits of the gauge equivalence classes 
of solutions $(\epsilon_{h}(g),\eta_g,A_g(h,k))$ of the polynomial equations Eq.(\ref{cocycle})-Eq.(\ref{AAA}), 
where the action of $H^2(G,\T)$ is given as in Theorem \ref{H2action}. 
\end{theorem}

To classify generalized Haagerup categories with $G$ without specifying a distinguished simple object $\rho$, 
we still have freedom to reparametrize the simple objects; namely replacing $\rho$ with $\alpha_p\rho$ 
gives rise to an action of $G$ (in fact $G/2G$, see below) on the gauge equivalence classes of 
the solutions of the polynomial equations. 

We fix $p\in G$. 
Then $(\alpha_p\rho,(\alpha_p\rho)^2)=(\alpha_p\rho,\rho^2)$ and $G_2$ acts on $(\alpha_p\rho,(\alpha_p\rho)^2)$ 
by the character $\chi_p$. 
We choose an extension $\chi\in \widehat{G}$ of $\chi_p$ and choose a unitary $V\in \cU(M)$ satisfying 
$\alpha_g(V)=\chi(g)V$. 
Let $\sigma=\alpha_p\circ \Ad V\circ \rho$. 
Then the pair $[\sigma,\alpha]$ is a standard lifting. 
Let $S'=V\rho(V)S$. 
Then $S'$ is an isometry in $(\id,\sigma^2)$, and we define two anti-unitaries on $(\alpha_g\sigma,\sigma^2)$ by 
$$j'_{1,g}T=\sqrt{d}T^*\sigma(S')=\sqrt{d}T^*V\rho(V\rho(V)S)V^*,$$
$$j'_{2,g}T=\sqrt{d}\alpha_g\sigma(T^*)S'=\sqrt{d}V\alpha_{p+g}\rho(T^*)\rho(V)S.$$

For each $g$, we choose a square root of $\chi(g)$ and denote it by $\nu(g)$. 
Let $T'_g=\overline{\nu(p+g)}V\rho(V)T_{p+g}V^*$.  
Then $T'_g$ is an isometry of $(\alpha_g\sigma,\sigma^2)$ satisfying $j_{1,g}T'_g=T'_g$. 
As before we define $\epsilon'_h(g)$, $\eta'_g$ and $A'_g(h,k)$ by $\alpha_h(T_g')=\epsilon'_h(g)T'_{g+2h}$, 
$j'_{2,g}T'_g=\eta'_gT'_g$, and 
$$\alpha_g\circ\sigma(T'_g)=\eta_gT'_gS'{S'}^*+\frac{\overline{\eta_g}}{\sqrt{d}}S'{T'}_g^*+
\sum_{h,k\in G}A'_g(h,k)T'_{g+h}T'_{g+h+k}{T'}_{g+k}^*.$$
We can choose $\nu$ to satisfy $\nu(p+2h)=\chi(h)\epsilon_h(p)\nu(p)$ for any $h\in G$, 
which makes $\epsilon'_h(0)=1$. 

\begin{lemma} \label{translation} Let the notation be as above. Then 
$$\epsilon'_h(g)=\epsilon_h(p+g)\nu(p+g+2h)\overline{\nu(p+g)\chi(h)},$$
$$\eta'_g=\eta_{p+g},$$
$$A'_g(h,k)=A_{p+g}(h,k)\nu(p+g+h+k)\nu(p+g+h)\overline{\nu(p+g+k)\nu(p+g)\chi(h)}.$$
\end{lemma}

When $p=2q\in 2G$, the character $\chi_{p}$ is trivial, and we can choose $\chi=\nu=1$. 
Then we have 
$$\epsilon'_h(g)=\epsilon_h(2q+g)=\epsilon_{h+q}(g)\epsilon_q(g)=\epsilon_h(g)\epsilon_q(g+2h)\epsilon_q(g),$$
$$\eta'_g=\eta_{2q+g}=\eta_g,$$
$$A'_g(h,k)=A_{p+g}(h,k)=A_g(h,k)\epsilon_q(g)\epsilon_q(g+h)\epsilon_q(g+k)\epsilon_q(g+h+k).$$
This shows that $2G$ acts trivially on the gauge equivalence classes of the solutions. 
This corresponds to the fact that the action of $2G$ comes from the inner automorphisms 
$\alpha_q\otimes \cdot \otimes \alpha_{-q}$ of the category $\cC$. 

In summary, we get the following classification result.

\begin{theorem} Let $G$ be a finite abelian group and let 
$$\Gamma=(H^2(G,\T)\times G/2G)\rtimes \Aut(G).$$
The equivalence classes of C$^*$-generalized Haagerup categories with $G$ are in one-to-one correspondence with 
the $\Gamma$-orbits of the gauge equivalence classes of solutions $(\epsilon_{h}(g),\eta_g,A_g(h,k))$ of the polynomial equations 
Eq.(\ref{cocycle})-Eq.(\ref{AAA}). 
\end{theorem}

In view of the above argument and Theorem \ref{mf}, we get 

\begin{theorem}  Let $G$ and $\Gamma$ be as above, and let $\cC$ be a generalized Haagerup category with $G$ 
having a solution $(\epsilon_{h}(g),\eta_g,A_g(h,k))$ of the polynomial equations Eq.(\ref{cocycle})-Eq.(\ref{AAA}). 
We assume $A_g(h,k)\neq 0$ for any $g,h,k\in G$. 
Then the outer automorphism group $\Out(\cC)$ of $\cC$ is the stabilizer subgroup $\Gamma_0$ of 
$\Gamma$ for the gauge equivalence class of the solution $(\epsilon_{h}(g),\eta_g,A_g(h,k))$.
\end{theorem}

\begin{proof} Let $[\rho,\alpha]$ be a standard lifting of $\cC\subset \End_0(M)$, and let 
$S\in (\id,\rho^2)$ and $T_g\in (\alpha_g,\rho^2)$ be isometries for which we have the solution 
$(\epsilon_{h}(g),\eta_g,A_g(h,k))$ of the polynomial equations 
Eq.(\ref{cocycle})-Eq.(\ref{AAA}). 
Let $(F,L)$ be a monoidal functor from $\cC$ to itself giving an element in $\Out(\cC)$. 
Thanks to Theorem \ref{uniqueness}, we may assume, up to natural transformation, that there exists 
$\Phi\in \Aut(M)$ such that $L$ is trivial and $F$ is given by $F(\sigma)=\Phi\circ \sigma \circ \Phi^{-1}$ 
for objects $\sigma$ and by $F(X)=\Phi(X)$ for morphisms $X$. 
Thus the pair $[F(\rho),F(\alpha_\cdot)]$ is a standard lifting of $\cC$ too, and there exist $p\in G$, 
$\theta\in \Aut(G)$, $\omega\in Z^2(g,\T)$, $V\in \cU(M)$, and $U_g\in \cU(M)$ such that 
$$F(\rho)=\Ad V\circ \alpha_{\theta(p)}\circ \rho,$$
$$F(\alpha_g)=\Ad U_{\theta(g)}\circ \alpha_{\theta(g)},$$
$$U_g\alpha_g(U_h)=\omega(g,h)U_{g+h}.$$
We define a homomorphism $\pi:\Out(\cC)\to \Gamma_0$ sending $[(F,L)]$ to 
$([\omega],p+2G,\theta)$. 
Thanks to Theorem \ref{mf}, it is a surjection. 

Assume $[(F,L)]\in \ker \pi$. 
Perturbing $(F,L)$ by an inner automorphism of $\cC$ if necessary, we may assume $p=0$, $\theta=\id$, and $\omega=1$. 
This implies that $\{U_g\}_{g\in G}$ is an $\alpha$-cocycle, which is always a coboundary, and so we may assume $U_g=1$ by perturbing 
$\Phi$ with an inner automorphism of $M$. 
Thus 
$$F(\rho)=\Phi\circ \rho\circ \Phi^{-1}=\Ad V\circ \rho,$$
$$F(\alpha_g)=\Phi\circ \alpha_g\circ \Phi^{-1}=\alpha_g,$$
and 
$$F(S)\in (\id,F(\rho)^2)=\C V\rho(V)S,$$
$$F(T_g)\in (F(\alpha_g)F(\rho),F(\rho)^2)=\C V\rho(V)T_g\alpha_g(V^*).$$
Note that $[F(\rho),F(\alpha_\cdot)]$ is a standard lifting of $\cC$, 
Since $F(\alpha_g)\circ F(\rho)=F(\rho)\circ F(\alpha_{-g})$, we see that $\alpha_g(V)$ is a multiple of $V$, 
Since $F(\alpha_z)(F(T_0))=F(T_0)$ for any $z\in G_2$, we have $\alpha_z(V)=V$. 
Moreover as in the proof of Lemma \ref{standard}, we may further assume that $\alpha_g(V)=V$ for any $g\in G$ 
by perturbing $\Phi$ with an inner automorphism of $M$ if necessary. 

By replacing $V$ with a multiple of $V$, we may assume $F(S)=V\rho(V)S$. 
We still have freedom to replace $V$ with $-V$ maintaining this equality. 
Since $F(T_g)$ is proportional to $V\rho(V)T_gV^*$, there exists $c_g\in \T$ satisfying 
$F(T_g)=c(g)V\rho(V)T_gV^*$. 
Applying $\Phi$ to the both sides of the two equations in Lemma \ref{ST}, we get 
$c(g)\in \{1,-1\}$, and 
\begin{equation}\label{E51}
A_g(h,k)=c(g)c(g+h)c(g+k)c(g+h+k)A_g(h,k).
\end{equation}
Since $A_g(h,k)\neq 0$, the map $c:G\to \{1,-1\}$ is a group homomorphism. 
Since $F(\alpha_h)(F(T_g))=\epsilon_h(g)F(T_{g+2h})$, we get 
\begin{equation} \label{E52}
c(g+2h)=c(g).
\end{equation}
By replacing $V$ with $-V$ if necessary, we may further assume 
\begin{equation}\label{E53}
c(0)=1.
\end{equation}

Now it is easy to construct a natural transformation to make $(F,L)$ equivalent to 
the identify functor as in the proof of \cite[Theorem 13.3]{I15}.  
\end{proof}

\begin{remark} 
In general, we can show that there exists an exact sequence 
$$0\to \ker \pi \to \Out(\cC)\to \Gamma_0\to 0,$$
where $\ker \pi$ is isomorphic to 
the quotient group of the set of maps $c:G\to \{1,-1\}$ satisfying Eq.(\ref{E51}),(\ref{E52}),(\ref{E53}) 
modulo the subgroup 
$$\{c\in \Hom(G,\{1,-1\});\;c(g+2h)=c(g),\; c(0)=1\}.$$
To the best knowledge of the author, there is no known example not satisfying 
the assumption $A_g(h,k)\neq 0$ for any $g,h,k\in G$. 
\end{remark}

When $G$ is an odd abelian group, the polynomial equations were already obtained in \cite{I01} 
under the additional assumption that $\id \oplus \rho$ has a $Q$-system, 
and they were extensively studied in \cite{EG11}.  
They take a very simple form, and so does the $H^2(G,\T)$-action (and $G/2G$ is trivial). 
Since $G_2$ is trivial, we can take $\epsilon_h(g)=1$, and the gauge freedom disappears. 
Since $2G=G$, neither $A_g(h,k)$ nor $\eta_g$ depends on $g$, and we denote them 
by $A(h,k)$ and $\eta$ respectively. 
The polynomial equations are reduced to 
\begin{equation}\label{R1'}
\eta^3=1,
\end{equation}
\begin{equation}\label{O1'}
\sum_{h\in G}A(h,0)=-\frac{\overline{\eta}}{d},
\end{equation}
\begin{equation}\label{O2'}
\sum_{h\in G}A(h-g,k)\overline{A(h-g',k)}=\delta_{g,g'}-\frac{\delta_{k,0}}{d},
\end{equation}
\begin{equation}\label{CC'}
A(k,h)=\overline{A(h,k)},
\end{equation}
\begin{equation}\label{R2'}
A(h,k)=A(-k,h-k)\eta=A(k-h,-h)\overline{\eta},
\end{equation}
\begin{align}\label{AAA'}
\lefteqn{
\sum_{l\in G}A(x+y,l)A(-x,l+p)A(-y,l+q)} \\
&=A(p+x,q+x+y)A(q+y,p+x+y)-\frac{\delta_{x,0}\delta_{y,0}}{d}.\nonumber
\end{align}

Since $G$ is an odd abelian group, every cohomology class in $H^2(G,\T)$ can be represented by a anti-symmetric 
bicharacter $\omega$. 
Under this assumption, the action of $H^2(G,\T)$ is now given by 
\begin{equation}A'(h,k)=A(h,k)\omega(h,k).\end{equation}

The smallest odd abelian group $G$ with non-trivial $H^2(G,\T)$ is $\Z_3\times \Z_3$. 
However, Evans-Gannon \cite{EG11} showed that there is no generalized Haagerup category 
for $\Z_3\times \Z_3$ with a $Q$-system for $\id\oplus \rho$. 
We do not know if there exists a solution of Eq.(\ref{R1'})-Eq.(\ref{AAA'}) for $\Z_3\times \Z_3$ 
without a $Q$-system for $\id\oplus \rho$. 

\section{$Q$-systems}
\subsection{When $\id\oplus \rho$ has a $Q$-system}
When a C$^*$-generalized Haagerup category comes from a $3^G$ subfactor, the object 
$\id\oplus \rho$ has a $Q$-system. 
It is shown in \cite[Section 7]{I01} that $\id\oplus \rho$ has a $Q$-system if and only if 
\begin{equation}\label{Q1}
A_0(h,0)=\delta_{h,0}-\frac{1}{d-1}.
\end{equation}

As we will see later, it is often the case that any solution of 
(\ref{cocycle})-(\ref{hkshift}) and (\ref{Q1}) in fact satisfies 
\begin{equation}\label{Q2}
A_g(h,0)=\delta_{h,0}-\frac{1}{d-1}, 
\end{equation}
for any $g,h\in G$ (e.g $G=\Z_4, \Z_2\times \Z_2$). 
In other words, once $\id\oplus \rho$ has a $Q$-system, so does any other $\id\oplus \alpha_g\rho$ 
in this case.

As in \cite[Lemma 7.3]{I01}, we will simplify part of (\ref{AAA}) under the assumption Eq.(\ref{Q2}). 

\begin{lemma} If $(\epsilon,\eta,A)$ is a solution of (\ref{cocycle})-(\ref{hkshift}), the following holds:
\begin{align}\label{O3}
\lefteqn{\sum_{l\in G}A_{g-p+x}(-x,l+p)A_{g-q}(x,l+q)} \\
 &=\delta_{p-q+x,0}\overline{\eta_{g+q}}\epsilon_x(g-p-x)-\frac{\delta_{x,0}}{d}\eta_{g+p}\eta_{g+q}.  \nonumber
\end{align}
\end{lemma}

\begin{proof}
By (\ref{R2}) and (\ref{2hshift}), we have 
\begin{align*}
\lefteqn{A_{g-p+x}(-x,l+p)} \\
 &= A_{g-p+x}(l+p+x,x)\overline{\eta_{g-p+x}}\epsilon_x(g-p)\epsilon_x(g+l+x)\epsilon_x(g+l)\\
 &= A_{g-p-x}(l+p+x,x)\overline{\eta_{g-p+x}}\epsilon_x(g-p-x).
\end{align*}
Therefore by (\ref{CC}) and (\ref{O2}), we get the statement. 
\end{proof}

\begin{lemma} Let $(\epsilon,\eta,A)$ be a solution of (\ref{cocycle})-(\ref{hkshift}), and 
let $g\in G$. 
We assume that (\ref{Q2}) holds for any $h\in G$. 
Then (\ref{AAA}) with $x+y=0$ is equivalent to 
\begin{align}\label{AA-AA}
\lefteqn{A_g(-x,p)A_g(x,q)\epsilon_x(g-x)\epsilon_x(g+p-x)} \\
 &-A_g(p+x,q)A_g(q-x,p)\epsilon_x(g+q-x)\epsilon_x(g+p+q-x) \nonumber \\
 &=\frac{\delta_{p-q+x,0}}{d-1}\epsilon_x(g+p)-\frac{\delta_{x,0}}{d-1}. \nonumber
\end{align}
In particular,  (\ref{AAA}) with $x=y=0$ is equivalent to 
\begin{equation}\label{|A|}
|A_g(p,q)|^2=\delta_{p,0}\delta_{q,0}-\frac{\delta_{p,0}+\delta_{q,0}+\delta_{p,q}}{d-1}+\frac{d}{(d-1)^2}.
\end{equation}
\end{lemma}

\begin{proof} Note that since $A_g(0,0)\neq 0$, we have $\eta_g=1$. 
By (\ref{Q2}) and (\ref{O3}), the left-hand side of (\ref{AAA}) with $x+y=0$ is  
\begin{align*}
\lefteqn{\sum_{l\in G}(\delta_{l,0}-\frac{1}{d-1})A_{g-p+x}(-x,l+p)A_{g-q}(x,l+q)
}\\
 &=A_{g-p+x}(-x,p)A_{g-q}(x,q)-\frac{\delta_{p-x+q,0}\overline{\eta_{g+q}}\epsilon_x(g-p-x)}{d-1}
 +\frac{\delta_{x,0}\eta_{g+p}\eta_{g+q}}{d(d-1)}. \\
\end{align*}
The right-hand side is 
\begin{align*}
\lefteqn{A_g(p+x,q)A_{g-p}(q-x,p)\eta_{g+q+x}\eta_{g+p+q-x}\overline{\eta_{g+p}\eta_{g+q}}} \\
 &\times \epsilon_p(g-p+x) \epsilon_{p+x}(g-p+q-x) \epsilon_q(g-q)\epsilon_{q-x}(g-q+x) 
 -\frac{\delta_{x,0}\eta_{g+p}\eta_{g+q}}{d},
\end{align*}
and so \begin{align*}
\lefteqn{A_{g-p+x}(-x,p)A_{g-q}(x,q)} \\
 &-A_g(p+x,q)A_{g-p}(q-x,p)\eta_{g+q+x}\eta_{g+p+q+x}\overline{\eta_{g+p}\eta_{g+q}} \\
 &\times \epsilon_p(g-p+x) \epsilon_{p+x}(g-p+q-x) \epsilon_q(g-q)\epsilon_{q-x}(g-q+x)\\
 &=\frac{\delta_{p-x+q,0}\overline{\eta_{g+q}}\epsilon_x(g-p-x)}{d-1}- \frac{\delta_{x,0}\eta_{g+p}\eta_{g+q}}{d-1}.\\
\end{align*}
By (\ref{R2}),
\begin{align*}\lefteqn{A_{g-p+x}(-x,p)=A_g(-x,p)\eta_{g-p}\overline{\eta_{g+x}}\epsilon_{-x}(g+x)\epsilon_p(g-p)\epsilon_{p-x}(g-p+x)}\\
&=A_g(-x,p)\eta_{g+p}\overline{\eta_{g+x}}
\epsilon_x(g-x)\epsilon_x(g+p-x)\epsilon_p(g-p)\epsilon_p(g-p+x),
\end{align*}
$$A_{g-q}(x,q)=A_g(x,q)\eta_{g+x}\overline{\eta_{g-q}\eta_{g-q+x}}\epsilon_q(g-q)\epsilon_q(g-q+x),$$
$$A_{g-p}(q-x,p)=A_g(q-x,p)\eta_{g+q-x}\overline{\eta_{g-p}\eta_{g-p+q-x}}\epsilon_p(g-p)\epsilon_p(g-p+q-x).$$
Therefore we get (\ref{AA-AA}). 
By setting $x=0$, we get (\ref{|A|}). 
\end{proof}

\subsection{In the case without $Q$-systems}
When Eq.(\ref{Q2}) is not satisfied, we still have a counterpart of Eq.(\ref{|A|}).  

Let 
$$x_{g,h}=\epsilon_{-h}(g)A_g(-h,-h)=\eta_gA_g(h,0)=\eta_g^{-1}A_g(0,h)\in \R.$$
Then Eq.(\ref{R2}) implies 
\begin{equation}
\eta_gx_{g,0}=x_{g,0},
\end{equation}
and we see that $\eta_g\neq 1$ could occur only if $x_{g,0}=0$. 
Eq.(\ref{2hshift}) and (\ref{hkshift}) imply    
\begin{equation}
x_{g,h}=x_{g+2l,h}=x_{g+h,h}. 
\end{equation}
Eq.(\ref{O1}), (\ref{O2}), and (\ref{AAA}) with $x=y=p=0$ imply 
\begin{equation}\label{Deg1}
\sum_{l\in G}x_{g,l}=-\frac{1}{d},
\end{equation}
\begin{equation}\label{Deg2}
\sum_{l\in G}x_{g,l-g}x_{g',l-g'}=\delta_{g,g'}-\frac{1}{d},
\end{equation}
\begin{equation}\label{Deg3}
\sum_{l\in G}x_{g,l}^2x_{g-h,l+h}=x_{g,h}^2-\frac{1}{d}.
\end{equation}
We first solve this system of equations. 
With a solution, the absolute value of $A_g(h,k)$ is determined by 
\begin{equation}\label{|A|1}
\sum_{l\in G}x_{g,l}x_{g-h,l+h}x_{g-k,l+k}=|A_g(h,k)|^2-\frac{1}{d},
\end{equation}
where we used (\ref{AAA}) with $x=y=0$. 

When $G$ is odd, we may assume $x_{g,h}=x_{0,h}$, which we denote by $x_h$. 
Then these equations become 
\begin{equation}
\eta_gx_0=x_0,
\end{equation}
\begin{equation}\label{Deg1'}
\sum_{h\in G}x_h=-\frac{1}{d},
\end{equation}
\begin{equation}\label{Deg2'}
\sum_{h\in G}x_{h}x_{h+g}=\delta_{g,0}-\frac{1}{d},
\end{equation}
\begin{equation}\label{Deg3'}
\sum_{h\in G}x_h^2x_{h+g}=x_g^2-\frac{1}{d},
\end{equation}
\begin{equation}
\sum_{l\in G}x_lx_{l+h}x_{l+k}=|A(h,k)|^2-\frac{1}{d}.
\end{equation}

\begin{remark} \label{|A|2} From the above relation, we can see that the existence of a $Q$-system for 
$\id\oplus \rho$ solely implies that Eq.(\ref{|A|}) holds for $g=0$. 
Indeed, Eq.(\ref{Q1}) implies 
\begin{align*}
\lefteqn{|A_0(h,k)|^2=\frac{1}{d}+\sum_{l\in G}x_{0,l}x_{-h,l+h}x_{-k,l+k}} \\
 &=\frac{1}{d} +\sum_{l\in G}(\delta_{l,0}-\frac{1}{d-1})x_{-h,l+h}x_{-k,l+k}\\
 &=\frac{1}{d} +x_{-h,h}x_{-k,k}-\frac{1}{d-1}\sum_{l\in G}x_{-h,l+h}x_{-k,l+k}\\
 &=\frac{1}{d} +x_{0,h}x_{0,k}-\frac{1}{d-1}(\delta_{h,k}-\frac{1}{d}) \\
 &=\delta_{h,0}\delta_{k,0}-\frac{\delta_{h,0}+\delta_{k,0}+\delta_{h,k}}{d-1}+\frac{d}{(d-1)^2}
 \end{align*}
\end{remark}
\section{The structure of the $3^G$ subfactors}\label{subfactors} 
In this section, we prove our classification theorems for the $3^G$ subfactors using 
results obtained so far. 

\begin{lemma}\label{vanishing} 
Let $G$ be a (not necessarily commutative) finite group. 
If a quadratic category $\cC$ with $(G,\tau,1)$ comes from a $3^G$ subfactor, 
then $\fc^{0,3}(\cC)$ and $\fc^{1,2}(\cC)$ are trivial. 
\end{lemma}

\begin{proof} 
Let $N\subset M$ be a $3^G$ subfactor realizing the quadratic category $\cC$ as $M-M$ bimodules. 
We may assume that $M$ is of type III. 
Let $\iota:N\hookrightarrow M$ be the inclusion map. 
Then, we have $[\iota\biota]=[\id]+[\rho]$, and $\rho$ generates $\cC$. 
Thus we have $\cO(\cC)=\{[\rho]\}\sqcup \{[\alpha_g][\rho]\}_{g\in G}$ with the fusion rules 
$$[\alpha_g][\alpha_h]=[\alpha_{gh}],$$
$$[\alpha_g][\rho]=[\rho][\alpha_{g^\tau}],$$
$$[\rho]^2=[\id]+\sum_{g\in G}[\alpha_g][\rho].$$
Since the principal graph is $3^G$, we have a homomorphism $\kappa:N\to M$ with irreducible image satisfying 
$$[\rho][\iota]=[\iota]+[\kappa],$$
$$[\kappa\biota]=\sum_{g\in G}[\alpha_g][\rho],$$
(see Fig.\ref{Haagerup}). 
By symmetry, we have $[\alpha_g][\kappa]=[\kappa]$ for any $g\in G$, and we can choose a representative $\alpha_g$ 
satisfying $\alpha_g\circ \kappa=\kappa$. 
These representatives $\{\alpha_g\}_{g\in G}$ form a group and the class $\fc^{0,3}(\cC)$ vanishes. 

There exist unitaries $U_g\in \cU(M)$ and $\omega\in Z^2(G,\T)$ satisfying 
$$\rho\circ \alpha_{g^\tau}=\Ad U_g\circ \alpha_g\circ \rho,$$
$$U_g\alpha_h(U_h)=\omega(g,h)U_{gh}.$$
To show that the class $\fc^{1,2}(\cC)$ vanishes, it suffices to show that the cohomology class 
$[\omega]\in H^2(G,\T)$ is trivial. 
Let $\cH=(\kappa,\rho\kappa)$. 
We introduce a projective representation of $G$ on $\cH$. 
For $X\in \cH$, 
\begin{align*}
\lefteqn{U_g\alpha_g(X)\kappa(x)=U_g\alpha_g(X\kappa(x))=U_g\alpha_g(\rho\kappa(x)X)} \\
 &=U_g\alpha_g\rho\kappa(x)\alpha_g(X)=\rho\alpha_{g^\tau}\kappa(x)U_g\alpha_g(X)=\rho\kappa(x)U_g\alpha_g(X).
\end{align*}
Thus we get a linear transformation of $\cH$ mapping $X$ to $U_g\alpha_g(X)$, which is denoted by $V_g$. 
For $g,h\in G$, we have 
$$V_gV_hX=U_g\alpha_g(U_h\alpha_h(X))=U_g\alpha_g(U_h)\alpha_{gh}(X)=\omega(g,h)U_{gh}\alpha_{gh}(X)=\omega(g,h)V_{gh}X,$$
and $\{V_g\}_{g\in G}$ form a projective representation of $G$. 
On the other hand, since $[\rho][\iota]=[\iota]+[\kappa]$, we have  
\begin{align*}
\lefteqn{\dim \cH=\dim(\kappa,\rho\rho\iota)-\dim(\kappa,\rho\iota)} \\
 &=\dim(\kappa,(\id\oplus \bigoplus_{g\in G}\alpha_g\rho)\iota)-1
 =\sum_{g\in G}(\dim(\kappa,\alpha_g\iota)+\dim(\kappa,\alpha_g\kappa))-1 \\
 &=|G|-1.
\end{align*}
Thus we get 
$$\det V_g \det V_h=\omega(g,h)^{|G|-1}\det V_{gh},$$ 
and $(|G|-1)[\omega]=0$ in $H^2(G,\T)$, which implies $[\omega]=0$ 
(see \cite[Corollary 10.2]{B94}). 
\end{proof}

\begin{theorem} Let $G$ be a finite abelian group. 
The $3^G$ subfactors giving quadratic categories with $(G,-1,1)$ are completely classified by 
the $H^2(G,\T)\rtimes \Aut(G)$-orbits of the gauge equivalence classes of the solutions of 
Eq.(\ref{cocycle})-(\ref{AAA}) and (\ref{Q1}). 
\end{theorem}

\begin{proof} The statement follows from Theorem \ref{distinguished} and Lemma \ref{vanishing}. 
\end{proof}

\begin{theorem} Let $G$ be a finite group with odd order. 
Then a $3^G$ subfactor exists only if $G$ is abelian. 
Moreover, the $3^G$ subfactors are completely classified by 
the $H^2(G,\T)\rtimes \Aut(G)$-orbits of the solutions of 
Eq.(\ref{O1'})-(\ref{AAA'}) and (\ref{Q1}). 
\end{theorem}

\begin{proof} The statement follows from Theorem \ref{odd}, Theorem \ref{distinguished}, and Lemma \ref{vanishing}. 
\end{proof}

In the rest of this section, we give an algorithm to compute the other principal graph and the fusion rules 
of the dual category for the $3^G$ subfactors. 

Let $M$, $\cC$, $G$, $\alpha$, $\rho$, and $(\epsilon_h(g),\eta_g,A_g(h,k))$ be as in Section \ref{PEGHC}, 
and we assume Eq.(\ref{Q1}) holds. 
We describe the $Q$-system for $\id\oplus \rho$ now (see \cite{GI08}). 
Let $M^G$ be the fixed point subalgebra of $M$ under the $G$-action $\alpha$. 
We choose two isometries $V_0,V_1\in M^G$ with $V_0V_0^*+V_1V_1^*=1$, and set 
$$\gamma(x)=V_0xV_0^*+V_1\rho(x)V_1^*.$$
Let 
\begin{align*}
\lefteqn{W=\frac{1}{\sqrt{d+1}}V_0+\frac{1}{\sqrt{d+1}}V_1\rho(V_0)V_1^*} \\
 &+\sqrt{\frac{d}{d+1}}V_1\rho(V_1)SV_0^*+\sqrt{\frac{d-1}{d+1}}V_1\rho(V_1)T_0V_1^*.
\end{align*}
Then $(\gamma,V_0,W)$ is a $Q$-system, and the corresponding $3^G$ subfactor $N$ is given by 
$$N=\{x\in M;\; Wx=\gamma(x)W,\; Wx^*=\gamma(x^*)W\}.$$
Let $\iota:N\hookrightarrow M$ be the inclusion map. 
Then $\biota$ is identified with $\gamma$ regarded as a map from $M$ into $N$. 

We denote by $\cD$ the fusion category in $\End_0(N)$ arising from the subfactor $N\subset M$, 
which is the dual category of $\cC$ with respect to the above $Q$-system. 
Note that the set $\{[\alpha_g\iota]\}_{g\in G}\sqcup \{[\kappa]\}$ exhausts the equivalence classes 
of irreducible $M-N$ sectors associated with the subfactor $N\subset M$. 
To determine the irreducible objects in $\cD$, it suffices to compute the irreducible decomposition of 
$\biota\alpha_g\iota$ and $\biota\kappa$. 
Since 
$$\dim(\biota\iota,\biota\iota)=\dim(\iota\biota,\iota\biota)=2,$$ 
there exists an irreducible $\hrho\in \cD$ with 
$[\biota\iota]=[\id]+[\hrho]$. 
From the principal graph $3^G$, we get $[\iota\hrho]=[\iota]+[\kappa]$. 
We have $d(\hrho)=d$, $d(\iota)=\sqrt{d+1}$, and $d(\kappa)=(d-1)\sqrt{d+1}$. 

For $z\in G_2$, we have $\alpha_z\gamma=\gamma\alpha_z$ and $\alpha_z(W)=W$. 
Thus $\alpha_z$ globally preserves $N$, and we denote by $\beta_z$ the restriction of $\alpha_z$ to $N$. 
By definition, we have $\alpha_z\iota=\iota\beta_z$, 
and we get  
$$[\biota\alpha_z\iota]=[\biota\iota\beta_z]=[\beta_z]+[\hrho][\beta_z].$$ 
Thus $\beta_z,\hrho\beta_z\in \cD$. 
For $z_1,z_2\in G_2$ with $z_1\neq z_2$, we have 
$$\dim(\biota\alpha_{z_1}\iota,\biota\alpha_{z_2}\iota)=\dim(\iota\biota\alpha_{z_1},\alpha_{z_2}\iota\biota)=
\dim(\alpha_{z_1},\alpha_{z_2})+\dim(\rho\alpha_{z_1},\alpha_{z_2}\rho)=0,$$
which implies $[\beta_{z_1}]\neq[\beta_{z_2}]$ and $[\hrho\beta_{z_1}]\neq [\hrho\beta_{z_2}]$. 
In particular $\beta$ is an outer action of $G_2$ on $N$. 
Since 
$$[\iota][\hrho\beta_z]=[\iota\beta_z]+[\kappa\beta_z],$$
$\kappa\beta_z$ is equivalent to either $\kappa$ or $\alpha_g\iota$. 
Comparing dimensions, we get $[\kappa\beta_z]=[\kappa]$, 
and $\dim(\kappa,\iota\hrho\beta_z)=1$. 

For $g,h\in G\setminus G_2$, we have 
$$\dim(\biota\alpha_g\iota,\biota\alpha_h\iota)=\dim(\iota\biota\alpha_g,\alpha_h\iota\biota)
=\dim(\alpha_g,\alpha_h)+\dim(\rho\alpha_g,\alpha_h\rho)=\delta_{g,h}+\delta_{g,-h}.$$
Thus $\biota\alpha_g\iota$ is irreducible and $[\biota\alpha_g\iota]=[\biota\alpha_h\iota]$ if and only 
if either $g=h$ or $g=-h$. 
We introduce an equivalence relation in $G\setminus G_2$ as $g\sim h$ if and only if either $g=h$ or $g=-h$, 
and choose a transversal $J_0\subset G$ of the equivalence relation. 
For $g\in J_0$, we set $\pi_g=\biota\alpha_g\iota\in \cD$. 
Then 
$$[\iota\pi_g]=[\alpha_g\iota]+[\rho\alpha_g\iota]=[\alpha_g\iota]+[\alpha_{-g}\rho\iota]=
[\alpha_g\iota]+[\alpha_{-g}\iota]+[\kappa],$$
and $\dim(\kappa,\iota\pi_g)=1$. 

It remains to compute the irreducible decomposition of $\biota\kappa$. 
By Frobenius reciprocity, we get $\dim(\biota\kappa,\hrho\beta_z)=\dim(\biota\kappa,\pi_g)=1$. 
Thus there exists mutually inequivalent irreducible endomorphisms $\{\sigma_j\}_{j\in J_1}\subset \cD$ and 
natural numbers $n_j$ satisfying 
\begin{equation}\label{dualgraph}
[\biota\kappa]=\sum_{z\in G_2}[\hrho\beta_z]+\sum_{g\in J_0}[\pi_g]+\sum_{j\in J_1}n_j[\sigma_j].
\end{equation}
Since 
$$\dim(\iota\sigma_j,\alpha_g\iota)=\dim(\sigma_j,\biota\alpha_g\iota)=0,$$
we have $[\iota\sigma_j]=n_j[\kappa]$, and 
$$d(\sigma_j)=\frac{n_jd(\kappa)}{d(\iota)}=n_j(d-1).$$ 
Comparing the dimensions of the both sides of Eq.(\ref{dualgraph}), we get 
$$\sum_{j\in J_1}n_j^2=\frac{|G|-|G_2|}{2}.$$
Summing up the above argument, we get 

\begin{prop} With the above notation, 
$$\cO(\cD)=\{[\beta_z]\}_{z\in G_2}\sqcup \{[\pi_g]\}_{g\in J_0}\sqcup\{[\sigma_j]\}_{j\in J_0},$$
\end{prop}

Our goal is to show 

\begin{theorem} \label{mfree} $n_j=1$ for any $j\in J_1$. 
\end{theorem}

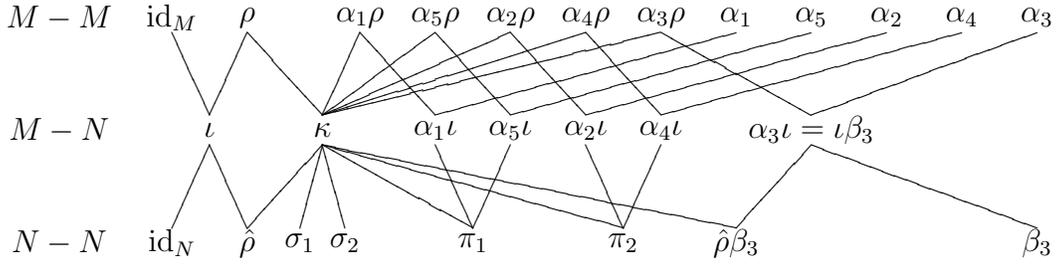
\begin{figure}[h]
${}$
\begin{xy} 
(0,15)*{M-M},(15,15)*{\id_M},(25,15)*{\rho},(40,15)*{\alpha_1\rho},(50,15)*{\alpha_5\rho},(60,15)*{\alpha_2\rho},(70,15)*{\alpha_4\rho},
(80,15)*{\alpha_3\rho},(90,15)*{\alpha_1},(100,15)*{\alpha_5},(110,15)*{\alpha_2},(120,15)*{\alpha_4},(130,15)*{\alpha_3},
(0,0)*{M-N},(20,0)*{\iota},(35,0)*{\kappa},(50,0)*{\alpha_1\iota},(60,0)*{\alpha_5\iota},(70,0)*{\alpha_2\iota},(80,0)*{\alpha_4\iota},
(100,0)*{\alpha_3\iota=\iota\beta_3},
(0,-15)*{N-N},(15,-15)*{\id_N},(25,-15)*{\hat{\rho}},(32,-15)*{\sigma_1},(38,-15)*{\sigma_2},(55,-15)*{\pi_1},(75,-15)*{\pi_2},(90,-15)*{\hrho\beta_3},
(130,-15)*{\beta_3},
{(15,13) \ar @{-} (20,2)},{(20,2) \ar @{-} (25,13)},{(25,13) \ar @{-} (35,2)},{(35,2) \ar @{-} (40,13)},
{(35,2) \ar @{-} (50,13)},{(35,2) \ar @{-} (60,13)},{(35,2) \ar @{-} (70,13)},{(35,2) \ar @{-} (80,13)},
{(40,13) \ar @{-} (50,2)},{(50,13) \ar @{-} (60,2)},{(60,13) \ar @{-} (70,2)},{(70,13) \ar @{-} (80,2)},{(80,13) \ar @{-} (100,2)},
{(50,2) \ar @{-} (90,13)},{(60,2) \ar @{-} (100,13)},{(70,2) \ar @{-} (110,13)},{(80,2) \ar @{-} (120,13)},{(100,2) \ar @{-} (130,13)},
{(15,-13) \ar @{-} (20,-2)},{(20,-2) \ar @{-} (25,-13)},{(25,-13) \ar @{-} (35,-2)},{(35,-2) \ar @{-} (32,-13)},{(35,-2) \ar @{-} (38,-13)},
{(35,-2) \ar @{-} (55,-13)},{(35,-2) \ar @{-} (75,-13)},{(35,-2) \ar @{-} (90,-13)},
{(50,-2) \ar @{-} (55,-13)},{(60,-2) \ar @{-} (55,-13)},{(70,-2) \ar @{-} (75,-13)},{(80,-2) \ar @{-} (75,-13)},
{(100,-2) \ar @{-} (90,-13)},{(100,-2) \ar @{-} (130,-13)},
\end{xy}
\caption{The principal graphs of the $3^{\Z_6}$ subfactor}
\end{figure}

To prove the above theorem, we consider a finite dimensional C$^*$-algebra 
$\cA=(\biota\rho\iota,\biota\rho\iota)$. 
Since 
$$[\biota\rho\iota]=[\biota\iota]+[\biota\kappa]=[\id]+2[\hrho]+\sum_{z\in G_2\setminus\{0\}}[\hrho\beta_z]
+\sum_{g\in J_0}[\pi_g]+\sum_{j\in J_1}n_j[\sigma_j],$$
it is isomorphic to 
$$\C\oplus M_2(\C)\oplus \C^{G_2\setminus\{0\}}\oplus \C^{J_0}\oplus \bigoplus_{j\in J_1}M_{n_j}(\C).$$
Let $[\cA,\cA]$ be the linear span of $\{xy-yx\in \cA;\; x,y\in \cA\}$. 
To show that $n_j=1$ for any $j\in J_1$, it suffices to show $\dim [\cA,\cA]<6$. 

In what follows we use graphical expression of intertwines as follows: 

$\begin{xy}(-5,-5)*{\rho},(5,-5)*{\rho},
{(-5,-3) \ar @(u,u)@{-} (5,-3)},
\end{xy}=\sqrt{d}S$,  
$\begin{xy}(-5,5)*{\rho},(5,5)*{\rho},
{(-5,3) \ar @(d,d)@{-} (5,3)},
\end{xy}=\sqrt{d}S^*$,  
$\begin{xy}(-5,-5)*{\rho},(5,-5)*{\rho},(0,5)*{\rho},
{(-5,-3) \ar @{-} (0,0)},
{(5,-3) \ar @{-} (0,0)},
{(0,3) \ar @{-} (0,0)},
\end{xy}=T_0$,  
$\begin{xy}(-5,5)*{\rho},(5,5)*{\rho},(0,-5)*{\rho},
{(-5,3) \ar @{-} (0,0)},
{(5,3) \ar @{-} (0,0)},
{(0,-3) \ar @{-} (0,0)},
\end{xy}=T_0^*$, 

$\begin{xy}(-5,-5)*{\iota},(5,-5)*{\biota},
{(-5,-3) \ar @(u,u)@{-} (5,-3)},
\end{xy}=\sqrt{d(\iota)}V_0$,  
$\begin{xy}(-5,5)*{\iota},(5,5)*{\biota},
{(-5,3) \ar @(d,d)@{-} (5,3)},
\end{xy}=\sqrt{d(\iota)}V_0^*$,  
$\begin{xy}(-5,-5)*{\biota},(5,-5)*{\iota},
{(-5,-3) \ar @(u,u)@{-} (5,-3)},
\end{xy}=\sqrt{d(\iota)}W$,  
$\begin{xy}(-5,5)*{\biota},(5,5)*{\iota},
{(-5,3) \ar @(d,d)@{-} (5,3)},
\end{xy}=\sqrt{d(\iota)}W$,

$\begin{xy}(-5,-5)*{\iota},(5,-5)*{\biota},(0,5)*{\rho},
{(-5,-3) \ar @{-} (0,0)},
{(5,-3) \ar @{-} (0,0)},
{(0,3) \ar @{-} (0,0)},
\end{xy}=V_1$,  
$\begin{xy}(-5,5)*{\iota},(5,5)*{\biota},(0,-5)*{\rho},
{(-5,3) \ar @{-} (0,0)},
{(5,3) \ar @{-} (0,0)},
{(0,-3) \ar @{-} (0,0)}.
\end{xy}=V_1^*$,

Since $j_{1,0}(T_0)=j_{2,0}(T_0)=T_0$, we have 
$$\begin{xy}(-10,5)*{\rho},(0,5)*{\rho},(5,-7)*{\rho},
{(0,0) \ar @(ld,d)@{-} (-10,3)},
{(0,3) \ar @{-} (0,0)},
{(5,-5) \ar @{-} (0,0)},
\end{xy}=
\begin{xy}(10,5)*{\rho},(0,5)*{\rho},(-5,-7)*{\rho},
{(0,0) \ar @(rd,d)@{-} (10,3)},
{(0,3) \ar @{-} (0,0)},
{(-5,-5) \ar @{-} (0,0)},
\end{xy}=
\begin{xy}(-5,5)*{\rho},(5,5)*{\rho},(0,-7)*{\rho},
{(-5,3) \ar @{-} (0,0)},
{(5,3) \ar @{-} (0,0)},
{(0,-5) \ar @{-} (0,0)},
\end{xy},
$$
$$\begin{xy}(-10,-5)*{\rho},(0,-5)*{\rho},(5,7)*{\rho},
{(0,0) \ar @(lu,u)@{-} (-10,-3)},
{(0,-3) \ar @{-} (0,0)},
{(5,5) \ar @{-} (0,0)},
\end{xy}=
\begin{xy}(10,-5)*{\rho},(0,-5)*{\rho},(-5,7)*{\rho},
{(0,0) \ar @(ru,u)@{-} (10,-3)},
{(0,-3) \ar @{-} (0,0)},
{(-5,5) \ar @{-} (0,0)},
\end{xy}=
\begin{xy}(-5,-5)*{\rho},(5,-5)*{\rho},(0,7)*{\rho},
{(-5,-3) \ar @{-} (0,0)},
{(5,-3) \ar @{-} (0,0)},
{(0,5) \ar @{-} (0,0)},
\end{xy}.
$$

Note that we have 
$\begin{xy}(-5,5)*{\rho},(5,5)*{\rho},(-5,-5)*{\rho},(5,-5)*{\rho},
{(-5,3) \ar @{-} (-5,-3)},{(5,3) \ar @{-} (5,-3)},{(-5,-1) \ar @{-} (5,1)^{\rho}},
\end{xy}=T_0^*\rho(T_0)$ and 
$\begin{xy}(-5,5)*{\rho},(5,5)*{\rho},(-5,-5)*{\rho},(5,-5)*{\rho},
{(-5,3) \ar @{-} (-5,-3)},{(5,3) \ar @{-} (5,-3)},{(-5,1) \ar @{-} (5,-1)^{\rho}},
\end{xy}=\rho(T_0)^*T_0$. 
On the other hand, Lemma \ref{ST} and Eq.(\ref{Q1}) imply
\begin{align*}
\lefteqn{T_0^*\rho(T_0)=\rho(T_0)^*T_0} \\
 &=SS^*+T_0T_0^*-\frac{1}{d-1}\sum_{g\in G}T_gT_g^*=-\frac{1}{d-1}+\frac{d}{d-1}SS^*+T_0T_0^*.
\end{align*}
Thus it makes sense to express the above two diagrams simply by  
$\begin{xy}(-5,5)*{\rho},(5,5)*{\rho},(-5,-5)*{\rho},(5,-5)*{\rho},
{(-5,3) \ar @{-} (-5,-3)},{(5,3) \ar @{-} (5,-3)},{(-5,0) \ar @{-} (5,0)^{\rho}},
\end{xy}$, or by the letter $\mathrm{H}$,  and we get the relation 
$$\frac{1}{d-1}
\begin{xy}(-5,12)*{\rho},(5,12)*{\rho},(-5,-12)*{\rho},(5,-12)*{\rho},
{(-5,10) \ar @{-} (-5,-10)},{(5,10) \ar @{-} (5,-10)}
\end{xy}
+\begin{xy}(-5,12)*{\rho},(5,12)*{\rho},(-5,-12)*{\rho},(5,-12)*{\rho},
{(-5,10) \ar @{-} (-5,-10)},{(5,10) \ar @{-} (5,-10)},{(-5,0) \ar @{-} (5,0)^{\rho}}
\end{xy}
=\frac{1}{d-1}
\begin{xy}(-5,12)*{\rho},(5,12)*{\rho},(-5,-12)*{\rho},(5,-12)*{\rho},
{(-5,10) \ar @(d,d)@{-}  (5,10)},{(-5,-10) \ar @(u,u)@{-}  (5,-10)}\end{xy}
+\begin{xy}(-5,12)*{\rho},(5,12)*{\rho},(-5,-12)*{\rho},(5,-12)*{\rho},
{(-5,10) \ar @{-} (0,5)},{(5,10) \ar @{-} (0,5)},
{(0,-5) \ar @{-} (-5,-10)},{(0,-5) \ar @{-} (5,-10)},{(0,5) \ar @{-} (0,-5)^\rho}
\end{xy}.
$$
We frequently use this relation without mentioning it. 

We define 
$\begin{xy}(0,5)*{\iota},(-5,-5)*{\rho},(0,-5)*{\iota},
{(0,3) \ar @{-} (0,-3)},{(0,0) \ar @{-} (-5,-3)},
\end{xy}:=
\begin{xy}(-5,5)*{\iota},(0,-5)*{\rho},(10,-5)*{\iota},
{(-5,3) \ar @{-} (0,0)},{(0,0) \ar @(ru,u)@{-} (10,-3)},
{(0,0) \ar @{-} (0,-3)}
\end{xy}=\sqrt{d(\iota)}V_1^*W$, and its conjugate 
$\begin{xy}(0,-5)*{\iota},(-5,5)*{\rho},(0,5)*{\iota},
{(0,-3) \ar @{-} (0,3)},{(0,0) \ar @{-} (-5,3)}\end{xy}:=
(\begin{xy}(0,5)*{\iota},(-5,-5)*{\rho},(0,-5)*{\iota},
{(0,3) \ar @{-} (0,-3)},{(0,0) \ar @{-} (-5,-3)}\end{xy})^*.$
Then direct computation shows $\sqrt{d}S^*\rho(V_1^*W)=W^*V_1$, and we get the relation 
$\begin{xy}(10,5)*{\iota},(0,5)*{\rho},(10,-5)*{\iota},
{(0,3) \ar @(d,ld)@{-} (10,0)},{(10,3) \ar @{-} (10,-3)}\end{xy}
=\begin{xy}(0,-5)*{\iota},(-5,5)*{\rho},(0,5)*{\iota},
{(0,-3) \ar @{-} (0,3)},{(0,0) \ar @{-} (-5,3)}\end{xy}$. 
In a similar way, we define 
$\begin{xy}(0,5)*{\biota},(5,-5)*{\rho},(0,-5)*{\biota},
{(0,3) \ar @{-} (0,-3)},{(0,0) \ar @{-} (5,-3)},
\end{xy}:=\sqrt{d(\iota)}\gamma(V_0)^*W$, and its conjugate 
$\begin{xy}(0,-5)*{\biota},(5,5)*{\rho},(0,5)*{\biota},
{(0,-3) \ar @{-} (0,3)},{(0,0) \ar @{-} (5,3)}\end{xy}:=
(\begin{xy}(0,5)*{\biota},(5,-5)*{\rho},(0,-5)*{\biota},
{(0,3) \ar @{-} (0,-3)},{(0,0) \ar @{-} (5,-3)},
\end{xy})^*$. 
Then we have 
$\begin{xy}(-10,5)*{\biota},(0,5)*{\rho},(-10,-5)*{\biota},
{(0,3) \ar @(d,rd)@{-} (-10,0)},{(-10,3) \ar @{-} (-10,-3)}\end{xy}
=\begin{xy}(0,-5)*{\biota},(5,5)*{\rho},(0,5)*{\biota},
{(0,-3) \ar @{-} (0,3)},{(0,0) \ar @{-} (5,3)}\end{xy}$. 

\begin{lemma}\label{reduction} With the above notation, we have 
\begin{enumerate}
 \item $\begin{xy}(0,8)*{\iota},(0,-8)*{\iota},
{(0,4) \ar @(l,l)@{-} (0,-4)_\rho},{(0,6) \ar @{-} (0,-6)}\end{xy}=\frac{d}{d(\iota)}\;
\begin{xy}(0,8)*{\iota},(0,-8)*{\iota},
{(0,6) \ar @{-} (0,-6)}\end{xy}$.

 \item $\begin{xy}(0,-8)*{\rho},(5,8)*{\iota},(5,-8)*{\iota},
{(5,6) \ar @{-} (5,-6)},{(5,4) \ar @{-} (0,-3)_\rho}, 
{(5,1) \ar @{-} (0,-3)^\rho},
{(0,-3) \ar @{-} (0,-6)}
\end{xy}=\sqrt{\frac{d-1}{d(\iota)}}
\begin{xy}(0,-8)*{\rho},(5,8)*{\iota},(5,-8)*{\iota},
{(5,6) \ar @{-} (5,-6)},{(5,0) \ar @{-} (0,-6)}\end{xy}$.

 \item $\begin{xy}(0,-8)*{\rho},(5,-8)*{\rho},(10,8)*{\iota},(10,-8)*{\iota},
{(10,6) \ar @{-} (10,-6)},{(10,3) \ar @{-} (0,-6)}, 
{(10,-1) \ar @{-} (5,-6)}\end{xy}
=\frac{1}{d(\iota)} 
\begin{xy}(0,-8)*{\rho},(5,-8)*{\rho},(10,8)*{\iota},(10,-8)*{\iota},
{(10,6) \ar @{-} (10,-6)},{(0,-6) \ar @(u,u)@{-} (5,-6)}\end{xy}+\sqrt{\frac{d-1}{d(\iota)}}
\begin{xy}(0,-8)*{\rho},(6,-8)*{\rho},(10,8)*{\iota},(10,-8)*{\iota},
{(10,6) \ar @{-} (10,-6)},{(3,-3) \ar @{-} (0,-6)},{(3,-3) \ar @{-} (6,-6)},
{(10,3) \ar @{-} (3,-3)_\rho},\end{xy}$.

 \item $\begin{xy}(0,8)*{\biota},(0,-8)*{\biota},
{(0,4) \ar @(r,r)@{-} (0,-4)_\rho},{(0,6) \ar @{-} (0,-6)}\end{xy}=\frac{d}{d(\iota)}\;
\begin{xy}(0,8)*{\biota},(0,-8)*{\biota},
{(0,6) \ar @{-} (0,-6)}\end{xy}$.

 \item $\begin{xy}(0,-8)*{\rho},(-5,8)*{\biota},(-5,-8)*{\biota},
{(-5,6) \ar @{-} (-5,-6)},{(-5,4) \ar @{-} (0,-3)_\rho}, 
{(-5,1) \ar @{-} (0,-3)^\rho},
{(0,-3) \ar @{-} (0,-6)}
\end{xy}=\sqrt{\frac{d-1}{d(\iota)}}\;
\begin{xy}(0,-8)*{\rho},(-5,8)*{\biota},(-5,-8)*{\biota},
{(-5,6) \ar @{-} (-5,-6)},{(-5,0) \ar @{-} (0,-6)}\end{xy}$.

 \item $\begin{xy}(0,-8)*{\rho},(-5,-8)*{\rho},(-10,8)*{\biota},(-10,-8)*{\biota},
{(-10,6) \ar @{-} (-10,-6)},{(-10,3) \ar @{-} (0,-6)}, 
{(-10,-1) \ar @{-} (-5,-6)}\end{xy}
=\frac{1}{d(\iota)} \;
\begin{xy}(0,-8)*{\rho},(-5,-8)*{\rho},(-10,8)*{\biota},(-10,-8)*{\biota},
{(-10,6) \ar @{-} (-10,-6)},{(0,-6) \ar @(u,u)@{-} (-5,-6)}\end{xy}+\sqrt{\frac{d-1}{d(\iota)}}\;
\begin{xy}(0,-8)*{\rho},(-6,-8)*{\rho},(-10,8)*{\biota},(-10,-8)*{\biota},
{(-10,6) \ar @{-} (-10,-6)},{(-3,-3) \ar @{-} (0,-6)},{(-3,-3) \ar @{-} (-6,-6)},
{(-10,3) \ar @{-} (-3,-3)_\rho},\end{xy}$.

\end{enumerate}
\end{lemma}

\begin{proof} We show only (1),(2),(3) because the rest follows from similar arguments. 

(1) The left-hand side is 
$$d(\iota)W^*V_1V_1^*W=d(\iota)W^*(1-V_0V_0^*)W=d(\iota)(1-\frac{1}{d(\iota)^2})=\frac{d}{d(\iota)}.$$

(2) The left-hand side is 
\begin{align*}
\lefteqn{d(\iota)T_0^*\rho(V_1^*W)V_1^*W=d(\iota)T_0^*\rho(V_1^*)V_1^*\gamma(W)W} \\
 &=d(\iota)T_0^*\rho(V_1^*)V_1^*WW=d(\iota)\sqrt{\frac{d-1}{d+1}}V_1^*W=\sqrt{d-1}V_1^*W.\\
\end{align*}
(3) The statement follows from (1),(2) and $1=SS^*+\sum_{g\in G}T_gT_g^*$ with 
consideration that $(\alpha_g\iota,\iota)=\{0\}$ for any $g\in G\setminus \{0\}$. 
\end{proof}

For $X\in (\rho^2,\rho^2)$, we set 
$f(X)=\begin{xy}(0,0)*+[F]{X},
(-7,10)*{\biota},(0,10)*{\rho},(7,10)*{\iota},
(-7,-10)*{\biota},(0,-10)*{\rho},(7,-10)*{\iota},
{(-7,8) \ar @{-} (-7,-8)},{(7,8) \ar @{-} (7,-8)},
{(0,8) \ar @{-} (1,3)},{(-1,-3) \ar @{-} (0,-8)},
{(-7,6) \ar @{-} (-1,3)^\rho},{(1,-3) \ar @{-} (7,-6)_\rho}
\end{xy}\in \cA$. 
We define $X_1,X_2,X_3,X_4,X_5\in \cA$ by 
$X_1=\begin{xy}(-5,8)*{\biota},(0,8)*{\rho},(5,8)*{\iota},
(-5,-8)*{\biota},(0,-8)*{\rho},(5,-8)*{\iota},
{(-5,6) \ar @{-} (-5,-6)},{(0,6) \ar @{-} (0,-6)},{(5,6) \ar @{-} (5,-6)},
{(0,0) \ar @{-} (5,0)^\rho}\end{xy}, 
X_2=\begin{xy}(-5,8)*{\biota},(0,8)*{\rho},(5,8)*{\iota},
(-5,-8)*{\biota},(0,-8)*{\rho},(5,-8)*{\iota},
{(-5,6) \ar @{-} (-5,-6)},{(0,6) \ar @{-} (0,-6)},{(5,6) \ar @{-} (5,-6)},
{(-5,0) \ar @{-} (0,0)^\rho},\end{xy}$,
$X_3=f(1)=\begin{xy}(-5,5)*{\biota},(-5,-5)*{\biota},(0,5)*{\rho},(0,-5)*{\rho},(5,5)*{\iota},(5,-5)*{\iota},
{(-5,3) \ar @{-} (-5,-3)},{(5,3) \ar @{-} (5,-3)},{(-5,0) \ar @{-} (0,-3)},{(0,3) \ar @{-} (5,0)},
\end{xy}$, 
$X_4=f(dSS^*)=\begin{xy}(-5,5)*{\biota},(-5,-5)*{\biota},(0,5)*{\rho},(0,-5)*{\rho},(5,5)*{\iota},(5,-5)*{\iota},
{(-5,3) \ar @{-} (-5,-3)},{(5,3) \ar @{-} (5,-3)},{(-5,0) \ar @{-} (0,3)},{(0,-3) \ar @{-} (5,0)},\end{xy}$, 
$X_5=f(T_0T_0^*)=
\begin{xy}(-5,5)*{\biota},(-5,-5)*{\biota},(0,5)*{\rho},(0,-5)*{\rho},(5,5)*{\iota},(5,-5)*{\iota},
{(-5,3) \ar @{-} (-5,-3)},{(5,3) \ar @{-} (5,-3)},{(0,3) \ar @{-} (0,-3)},
{(-5,1) \ar @{-} (0,1)},{(0,-1) \ar @{-} (5,-1)},
\end{xy}$. 
Let $\cB$ be the linear span of $\{1,X_1,X_2,X_3,X_4,X_5\}$. 
In view of the linear isomorphism, 
$$(\biota\rho\iota,\biota\rho\iota)\cong (\gamma\rho,\rho\gamma)=((\id\oplus \rho)\rho,\rho(\id\oplus \rho)),$$
we see that $\cB$ and $\{f(T_gT_g^*)\}_{g\in G\setminus\{0\}}$ linearly span $\cA$. 
It is easy to show $X_1^*=X_1$, $X_2^*=X_2$, $X_3^*=X_4$, 
$$X_5^*=f(\mathrm{H})=-\frac{1}{d-1}X_3+\frac{1}{d-1}X_4+X_5,$$
and $\cB$ is closed under the adjoint operation. 

In what follows we compute the multiplication table of $\cA$ modulo $\cB$ using Lemma \ref{reduction} (3),(6). 
For $X_1^2$, we have  
$$
X_1^2=\begin{xy}(-5,6)*{\biota},(-5,-6)*{\biota},(0,6)*{\rho},(0,-6)*{\rho},(5,6)*{\iota},(5,-6)*{\iota},
{(-5,4) \ar @{-} (-5,-4)},
{(5,4) \ar @{-} (5,2)},{(5,2) \ar @{-} (5,-2)^\iota},{(5,-2) \ar @{-} (5,-4)},
{(0,4) \ar @{-} (0,2)},{(0,2) \ar @{-} (0,-2)_\rho},{(0,-2) \ar @{-} (0,-4)},
{(0,2) \ar @{-} (5,2)^\rho},{(0,-2) \ar @{-} (5,-2)_\rho},
\end{xy} =\frac{1}{d(\iota)}\; 
 \begin{xy}(-5,6)*{\biota},(-5,-6)*{\biota},(0,6)*{\rho},(0,-6)*{\rho},(12,6)*{\iota},(12,-6)*{\iota},
{(-5,4) \ar @{-} (-5,-4)},
{(12,4) \ar @{-} (12,-4)},
{(0,4) \ar @{-} (0,2)},{(0,2) \ar @{-} (0,-2)_\rho},{(0,-2) \ar @{-} (0,-4)},
{(0,2) \ar @(r,r)@{-} (0,-2)_\rho},
\end{xy}
+\sqrt{\frac{d-1}{d(\iota)}}\;
\begin{xy}(-5,6)*{\biota},(-5,-6)*{\biota},(0,6)*{\rho},(0,-6)*{\rho},(10,6)*{\iota},(10,-6)*{\iota},
{(-5,4) \ar @{-} (-5,-4)},
{(10,4) \ar @{-} (10,-4)},
{(0,4) \ar @{-} (0,2)},{(0,2) \ar @{-} (0,-2)_\rho},{(0,-2) \ar @{-} (0,-4)},
{(0,2) \ar @{-} (5,0)^\rho},{(0,-2) \ar @{-} (5,0)_\rho},
{(5,0) \ar @{-} (10,0)^\rho}
\end{xy}
=\frac{1}{d(\iota)}+\frac{d-2}{\sqrt{d(\iota)(d-1)}}X_1, $$
where we used 
$$\begin{xy}(0,10)*{\rho},(-5,-7)*{\rho},(5,-7)*{\rho},
{(0,8) \ar @{-} (0,5)},{(0,5) \ar @{-} (-5,-2)_\rho},{(0,5) \ar @{-} (5,-2)^\rho},{(-5,-2) \ar @{-} (5,-2)_\rho},
{(-5,-2) \ar @{-} (-5,-5)},{(5,-2) \ar @{-} (5,-5)}
\end{xy}=T_0^*\rho(T_0)T_0=\frac{d-2}{d-1}T_0=\frac{d-2}{d-1}
\begin{xy}(0,10)*{\rho},(-5,-7)*{\rho},(5,-7)*{\rho},
{(0,8) \ar @{-} (0,1)},{(0,1) \ar @{-} (-5,-5)}, {(0,1) \ar @{-} (5,-5)}\end{xy}.$$
In particular $X_1^2\in \cB$. 
In a similar way, we have $X_2^2\in \cB$. 
From the definition of $X_5$, we get $X_1X_2=X_5\in \cB$, and so $X_2X_1=(X_1X_2)^*=X_5^*\in \cB$. 
We can easily show $X_iX_j\in \cB$ for $1\leq i,j\leq 4$ by the same type of computation as above.  

\begin{lemma} \label{product} 
Let the notation be as above, and let $Q_g=T_gT_g^*$. 
Then 
\begin{itemize}
\item[(1)] $X_1f(Q_g)\in \cB+\sqrt{\frac{d-1}{d(\iota)}}f(\mathrm{H}Q_g)$,
\item[(2)] $f(Q_g)X_1\in \cB+\sqrt{\frac{d-1}{d(\iota)}}f(\rho(T_0)^*Q_g\rho(T_0))$,
\item[(3)] $X_2f(Q_g)\in \cB+\sqrt{\frac{d-1}{d(\iota)}}f(T_0^*\rho(Q_g)T_0)$,
\item[(4)] $f(Q_g)X_2\in \cB+\sqrt{\frac{d-1}{d(\iota)}}f(Q_g\mathrm{H})$,
\item[(5)] $X_3f(Q_g)\in \cB$,
\item[(6)] $f(Q_g)X_3\in \cB$,
\item[(7)] $X_4f(Q_g)\in \cB$,
\item[(8)] $f(Q_g)X_4\in \cB$,
\item[(9)] $f(Q_{g_1})f(Q_{g_2})\in \cB+f(\rho(T_0)^*Q_{g_1}\rho(Q_{g_2})T_0)$. 
\end{itemize}
In particular $\cB$ is a $*$-subalgebra of $\cB$. 
\end{lemma}

\begin{proof} In the proof we suppress the labeling of edges unless there is possibility of confusion. 
The left external edges are always labeled by $\biota$, the right external edges are always labeled by $\iota$, 
and internal edges are always labeled by $\rho$. 
(1)-(6) are easy to verify. 
From (5) we get $X_3\cA\subset \cB$, and since $X_4=X_3^*$, we get (7). 
(8) follows from (6) in the same way. 

For $g_1,g_2\in G$, 
\begin{align*}
f(Q_{g_1})f(Q_{g_2})&=
\begin{xy}(0,7)*+[F]{Q_{g_2}},(0,-7)*+[F]{Q_{g_1}},
{(-10,16) \ar @{-} (-10,-16)},{(10,16) \ar @{-} (10,-16)},
{(0,16) \ar @{-} (2,11)},{(-10,13) \ar @{-} (-2,11)},
{(-2,3) \ar @{-} (2,-3)},
{(-10,0) \ar @{-} (-2,-3)},{(2,3) \ar @{-} (10,0)},
{(0,-16) \ar @{-} (-2,-11)},{(10,-13) \ar @{-} (2,-11)},
\end{xy}
\in \cB+\frac{d-1}{d(\iota)}\;
 \begin{xy}(0,7)*+[F]{Q_{g_2}},(0,-7)*+[F]{Q_{g_1}},
{(-11,16) \ar @{-} (-11,-16)},{(11,16) \ar @{-} (11,-16)},
{(0,16) \ar @{-} (2,11)},{(-11,13) \ar @{-} (-2,11)},
{(-2,3) \ar @{-} (2,-3)},
{(-9,12) \ar @{-} (-2,-3)},{(2,3) \ar @{-} (9,-12)},
{(0,-16) \ar @{-} (-2,-11)},{(11,-13) \ar @{-} (2,-11)},
\end{xy} \\
&=\cB+\frac{d-1}{d(\iota)}f(
\begin{xy} (0,6)*+[F]{Q_{g_2}},(0,-6)*+[F]{Q_{g_1}},
{(-2,16) \ar @{-} (-2,9)},{(2,16) \ar @{-} (2,9)},
{(-2,13) \ar @(ld,lu)@{-} (-2,-2)},{(2,-13) \ar @(ru,rd)@{-} (2,2)},{(-2,2) \ar @{-} (2,-2)},
{(2,-16) \ar @{-} (2,-9)},{(-2,-16) \ar @{-} (-2,-9)}
\end{xy})=\cB +f(\rho(T_0)^*Q_{g_1}\rho(Q_{g_2})T_0),  
\end{align*}
and we obtain (9). 

Putting $g=g_1=g_2=0$ in the above, we see that $\cB$ is a subalgebra. 
\end{proof}

\begin{lemma} The minimal central projection $z(\id)$ of $\cA$ corresponding to the irreducible 
component $\id$ contained in $\biota\rho\iota$ belongs to $\cB$. 
\end{lemma}

\begin{proof}
Since $\begin{xy}(-10,-5)*{\biota},(10,-5)*{\iota},(0,-5)*{\rho},
{(-10,-3) \ar @(u,u)@{-} (10,-3)},
{(0,-3) \ar @{-} (-5,5)}
\end{xy}\in(\id,\biota\rho\iota)$, and 
$$\begin{xy}(-10,0)*{\biota},(10,0)*{\iota},(0,0)*{\rho},
{(-10,2) \ar @(u,u)@{-} (10,2)},
{(0,2) \ar @{-} (-5,10)},
{(-10,-2) \ar @(d,d)@{-} (10,-2)},
{(0,-2) \ar @{-} (-5,-10)},
\end{xy}=\frac{d}{d(\iota)}\;
\begin{xy}(-10,0)*{\biota},(10,0)*{\iota},
{(-10,2) \ar @(u,u)@{-} (10,2)},
{(-10,-2) \ar @(d,d)@{-} (10,-2)},
\end{xy}=d,
$$
we have 
$$z(\id)=\frac{1}{d}\;
\begin{xy}(-10,-13)*{\biota},(10,-13)*{\iota},(0,-13)*{\rho},
{(-10,-11) \ar @(u,u)@{-} (10,-11)},
{(0,-11) \ar @{-} (-5,-3)},
(-10,13)*{\biota},(10,13)*{\iota},(0,13)*{\rho},
{(-10,11) \ar @(d,d)@{-} (10,11)},
{(0,11) \ar @{-} (-5,3)}.
\end{xy}$$
Using 
$$\begin{xy}
(-5,-13)*{\iota},(5,-13)*{\biota},
(-5,13)*{\iota},(5,13)*{\biota},
{(-5,11) \ar @{-} (-5,-11)},
{(5,11) \ar @{-} (5,-11)},
\end{xy}
=\frac{1}{d(\iota)}
\begin{xy}(-10,-13)*{\iota},(10,-13)*{\biota},
{(-10,-11) \ar @(u,u)@{-} (10,-11)},
(-10,13)*{\iota},(10,13)*{\biota},
{(-10,11) \ar @(d,d)@{-} (10,11)},
\end{xy}+
\begin{xy}
(-5,-13)*{\iota},(5,-13)*{\biota},(-5,13)*{\iota},(5,13)*{\biota},
{(-5,11) \ar @{-} (0,4)},{(5,11) \ar @{-} (0,4)},
{(0,4) \ar @{-} (0,-4)^\rho},
{(-5,-11) \ar @{-} (0,-4)},{(5,-11) \ar @{-} (0,-4)},
\end{xy},$$
we get \begin{align*}
\begin{xy}(-10,-13)*{\biota},(10,-13)*{\iota},(0,-13)*{\rho},
{(-10,-11) \ar @(u,u)@{-} (10,-11)},
{(0,-11) \ar @{-} (-5,-3)},
(-10,13)*{\biota},(10,13)*{\iota},(0,13)*{\rho},
{(-10,11) \ar @(d,d)@{-} (10,11)},
{(0,11) \ar @{-} (-5,3)}.
\end{xy}
&=\begin{xy}(-10,-13)*{\biota},(10,-13)*{\iota},(0,-13)*{\rho},
(-10,13)*{\biota},(10,13)*{\iota},(0,13)*{\rho},
{(-10,11) \ar @(d,d)@{-} (10,11)},
{(0,11) \ar @{-} (-5,3)}.
{(-10,-11) \ar @(ru,l)@{-} (20,5)},
{(20,5) \ar @(r,ru) @{-} (10,-11)},
{(0,-11) \ar @{-} (-4,-5)},
\end{xy} 
=\frac{1}{d(\iota)}
\begin{xy}(-5,7)*{\biota},(-5,-7)*{\biota},(0,7)*{\rho},(0,-7)*{\rho},(5,7)*{\iota},(5,-7)*{\iota},
{(-5,5) \ar @{-} (-5,-5)},{(5,5) \ar @{-} (5,-5)},
{(0,5) \ar @{-} (-5,2)},{(0,-5) \ar @{-} (-5,-2)},
\end{xy}+
\begin{xy}(-5,7)*{\biota},(-5,-7)*{\biota},(0,7)*{\rho},(0,-7)*{\rho},(5,7)*{\iota},(5,-7)*{\iota},
{(-5,5) \ar @{-} (-5,-5)},{(5,5) \ar @{-} (5,-5)},
{(0,5) \ar @{-} (-5,2)},{(0,-5) \ar @{-} (-5,-2)},
{(-5,0) \ar @{-} (5,0)}
\end{xy}\\
&\in \cB+\frac{1}{d(\iota)}\;
\begin{xy}(-5,7)*{\biota},(-5,-7)*{\biota},(0,7)*{\rho},(0,-7)*{\rho},(5,7)*{\iota},(5,-7)*{\iota},
{(-5,5) \ar @{-} (-5,-5)},{(5,5) \ar @{-} (5,-5)},
{(0,5) \ar @{-} (5,0)},{(-5,0) \ar @{-} (0,-5)},
\end{xy}+\sqrt{\frac{d-1}{d(\iota)}}\;
\begin{xy}(-5,7)*{\biota},(-5,-7)*{\biota},(0,7)*{\rho},(0,-7)*{\rho},(5,7)*{\iota},(5,-7)*{\iota},
{(-5,5) \ar @{-} (-5,-5)},{(5,5) \ar @{-} (5,-5)},
{(0,5) \ar @{-} (0,0)},{(0,-5) \ar @{-} (-5,-2)},
{(-5,0) \ar @{-} (5,0)}\end{xy}\\
&=\cB+\frac{1}{d(\iota)}\sqrt{\frac{d-1}{d(\iota)}}\;
\begin{xy}(-5,7)*{\biota},(-5,-7)*{\biota},(0,7)*{\rho},(0,-7)*{\rho},(5,7)*{\iota},(5,-7)*{\iota},
{(-5,5) \ar @{-} (-5,-5)},{(5,5) \ar @{-} (5,-5)},
{(0,5) \ar @{-} (0,-5)},{(0,0) \ar @{-} (5,0)},\end{xy}
+\frac{d-1}{d(\iota)}\;
\begin{xy}(-5,7)*{\biota},(-5,-7)*{\biota},(0,7)*{\rho},(0,-7)*{\rho},(5,7)*{\iota},(5,-7)*{\iota},
{(-5,5) \ar @{-} (-5,-5)},{(5,5) \ar @{-} (5,-5)},
{(0,5) \ar @{-} (2,0)},{(0,-5) \ar @{-} (-2,0)},
{(-5,0) \ar @{-} (5,0)}\end{xy}
=\cB+\frac{d-1}{d(\iota)}\;f(\mathrm{H})=\cB,
\end{align*}
which shows the statement. 
\end{proof}

\begin{proof}[Proof of Theorem \ref{mfree}] 
We first claim that it suffices to show $[\cA,\cA]\subset \cB$. 
Indeed, we have already observed that if $\dim [\cA,\cA]\leq 5$, we are done. 
On the other hand, since $z(\id)$ is a central projection, we have $z(\id)[\cA,\cA]=0$. 
Thus if $[\cA,\cA]\subset \cB$, we get 
$$[\cA,\cA]=(1-z(\id))[\cA,\cA]\subset (1-z(\id))\cB,$$
and $\dim [\cA,\cA]\leq \dim (1-z(\id))\cB=5$. 

Let $\cB_0$ be the linear span of $\{1,SS^*,T_0T_0^*\}$. 
In view of Lemma \ref{product}, it suffices to show that $\mathrm{H}Q_g-\rho(T_0^*)Q_g\rho(T_0)$, 
$T_0^*\rho(Q_g)T_0-Q_g\mathrm{H}$, and $\rho(T_0^*)Q_{g_1}\rho(Q_{g_2})T_0-\rho(T_0^*)Q_{g_2}\rho(Q_{g_1})T_0$ 
belong $\cB_0$ for any $g,g_1,g_2\in G\setminus \{0\}$. 
Indeed, we have $\mathrm{H}Q_g=Q_g\mathrm{H}=-\frac{1}{d-1}Q_g$ and 
$$T_0^*\rho(Q_g)T_0=\rho(T_0^*)Q_g\rho(T_0)=\sum_{k\in G}|A_0(g,k)|^2Q_k,$$
Thanks to Remark \ref{|A|2}, the right-hand side is 
$$\frac{1}{(d-1)^2}(1-SS^*)-\frac{1}{d-1}Q_g,$$
and we get 
$$\mathrm{H}Q_g-\rho(T_0^*)Q_g\rho(T_0)=Q_g\mathrm{H}-T_0^*\rho(Q_g)T_0=\frac{1}{(d-1)^2}(SS^*-1)\in \cB_0.$$

We finally show that $\rho(T_0^*)Q_{g_1}\rho(Q_{g_2})T_0$ is symmetric in $g_1$ and $g_2$. Indeed, it is equal to 
\begin{align*}
\lefteqn{\rho(T_0)T_{g_1}T_{g_1}^*\rho(T_{g_2}T_{g_2}^*)T_0=
\rho(T_0)T_{g_1}T_{g_1}^*\alpha_{-g_2}\rho(T_{-g_2}T_{-g_2}^*)T_0} \\
 &=\sum_{l\in G}\overline{A_0(g_1,l)}T_lT_{g_1+l}^*T_{g_1}^*\alpha_{-g_2}\rho(T_{-g_2}T_{-g_2}^*)T_0 \\
 &=\sum_{l\in G}\overline{A_0(g_1,l)}A_{-g_1}(g_1+g_2,l)T_lT_{-g_2+l}^*\alpha_{-g_2}\rho(T_{-g_2}^*)T_0 \\
 &=\sum_{l\in G}\overline{A_0(g_1,l)}A_{-g_1}(g_1+g_2,l)\overline{A_{-g_2}(g_2,l)}T_lT_l^*\\
 &=\sum_{l\in G}\overline{A_0(g_1,l)A_0(g_2,l)}A_{-g_2}(g_1+g_2,l)\eta_l\overline{\eta_{-g_2}\eta_{-g_2+l}}
 \epsilon_{g_2}(-g_2)\epsilon_{g_2}(-g_2+l)T_lT_l^*,
\end{align*}
where we used Eq.(\ref{hkshift}) at the end. 
Thus it suffices to show that 
$$A_{-g_2}(g_1+g_2,l)\overline{\eta_{-g_2}\eta_{-g_2+l}}\epsilon_{g_2}(-g_2)\epsilon_{g_2}(-g_2+l),$$
is symmetric in $g_1$ and $g_2$, which follows from Eq.(\ref{hkshift}),(\ref{2hshift}). 
\end{proof}

\begin{remark}
Let $z(\sigma_j)$ be the minimal projection in $\cA$ corresponding to $\sigma_j$. 
We can determine how the conjugation acts on $\{\sigma_j\}_{j\in J_1}$ by computing  the rotation 
of $z(\sigma_j)$ by 180 degrees, which can be easily done once we concretely write down $z(\sigma_{j})$ . 
We can determine the fusion coefficient $\dim(\sigma_{j_3},\sigma_{j_1}\sigma_{j_2})$ by computing
$$\dim \big(z(\sigma_{j_3})(\biota\rho\iota,(\biota\rho\iota)^2)z(\sigma_{j_1})\biota\rho\iota(z(\sigma_{j_2}))\big).$$
Since we can concretely write down a basis of $(\biota\rho\iota,(\biota\rho\iota)^2)$ by using the isomorphism 
$$(\biota\rho\iota,(\biota\rho\iota)^2)\cong (\gamma\rho,\rho\gamma\rho\gamma),$$
we can compute the fusion coefficient in principle, though it is a challenging problem to implement it 
in a concrete example. 
\end{remark}
\section{Orbifold construction}\label{orbiforld}
\subsection{De-equivariantization}
Let $M$, $\cC$, $G$, $\alpha$, $\rho$, $(\epsilon_h(g),\eta_g,A_g(h,k))$ be as in Section \ref{PEGHC}. 
We assume $z\in G\setminus\{0\}$ satisfies $2z=0$ and $\epsilon_z(z)=1$. 
We also assume that the map $G\ni g\mapsto \epsilon_z(g)\in \{1,-1\}$ is a character, 
which is always the case if $A_0(g,h)\neq 0$ for all $g,h\in G$ thanks to Lemma \ref{alpha}. 
Let $P=M\rtimes_{\alpha_z}\Z_2$ be the crossed product, which is the von Neumann algebra generated by 
$M$ and a unitary $\lambda$ satisfying $\lambda^2=1$ and $\lambda x\lambda^{-1}=\alpha_z(x)$ for all $x\in M$. 
Since $\alpha_z$ is outer, $P$ is a factor. 

We can extend $\alpha_g$ and $\rho$ to $P$ by $\talpha_g(\lambda)=\epsilon_z(g)\lambda$ and 
$\trho(\lambda)=\lambda$. 
Then $\talpha$ is a $G$-action on $P$, and $\talpha_z=\Ad \lambda$ thanks to the assumption on 
$\epsilon_z(g)$. 
Thus it makes sense to write $[\talpha_{\dot{g}}]=[\talpha_g]$ for $\dot{g}\in G/\{0,z\}$ as we have 
$[\talpha_{g+z}]=[\talpha_g]$ for all $g\in G$. 
It is easy to show the following theorem by using $P=M+M\lambda$. 

\begin{theorem}\label{orbifold1} Let the notation be as above. Then
\begin{itemize}
\item [(1)] $[\trho]$ is irreducible. 
\item [(2)] $[\talpha_g\trho]=[\talpha_h\trho]$ if and only if $g-h\in \{0,z\}$. 
\item [(3)]
$$[\trho^2]=[\id]\oplus \bigoplus_{\dot{g}\in G/\{0,z\}}2[\talpha_{\dot{g}}\trho].$$
\item [(4)] If $\id\oplus \rho$ has a $Q$-system, so does $\id\oplus \trho$.  
\end{itemize}
\end{theorem}

\begin{proof} We show only (3). 
For $x\in M$ we have $\trho^2(x)T_g=T_g\talpha_g\trho(x)$. 
For $\lambda$, we have 
$$\trho^2(\lambda)T_g=\lambda T_g=\alpha_z(T_g)\lambda=\epsilon_z(g)T_g\lambda=T_g\talpha_g(\lambda)
=T_g\talpha_g\trho(\lambda).$$
This shows 
$$[\trho^2]=[\id]\oplus \bigoplus_{g\in G}[\talpha_g\trho]
=[\id]\oplus \bigoplus_{\dot{g}\in G/\{0,z\}}2[\talpha_{\dot{g}}\trho].$$
\end{proof}

\begin{remark} One can compute the obstruction class of 
$\{[\talpha_{\dot{g}}]\}_{\dot{g}\in G/\{0,z\}}$ in the cohomology group 
$H^3(G/\{0,z\},\T)$ easily from $\epsilon_z(g)$. 
\end{remark}

\begin{remark} We can easily generalize the above argument to a subgroup $H$ of $G_2$ satisfying $\epsilon_h(k)=1$ 
for any $h,k\in H$. 
\end{remark}
\subsection{Equivariantization}
Let $(\epsilon_h(g),\eta_g,A_g(h,k))$ be a solution of 
(\ref{cocycle})-(\ref{AAA}) invariant under a group automorphism $\theta\in \Aut(G)$, that is, 
$\epsilon_{\theta(h)}(\theta(g))=\epsilon_h(g)$, $\eta_{\theta(g)}=\eta_g$, and $A_{\theta(g)}(\theta(h),\theta(k))=A_g(h,k)$. 
Let $\beta$ be an automorphism of the Cuntz algebra $\cO_{n+1}$ defined by $\beta(S)=S$ and $\beta(T_g)=T_{\theta(g)}$. 
Then it is easy to show $\beta\circ \rho=\rho\circ \beta$ and $\beta\circ \alpha_g=\alpha_{\theta(g)}\circ \beta$. 

Let $M$ be the weak closure of $\cO_{n+1}$ in the GNS representation of a KMS-state as in \cite{I93}, 
and we use the same symbols $\alpha$, $\rho$, and $\beta$ for their extensions to $M$. 
Then we may assume (replacing $(M,\rho,\alpha)$ with $(\tilde{M},\trho,\talpha)$ in Appendix if necessary) 
that $\alpha$ and $\beta$ generate an outer action of the semi-direct product group 
$G\rtimes_\theta \Z_m$ on $M$, where $m$ is the order of $\theta$. 
Let $P=M\rtimes_\beta \Z_m$ be the crossed product, which is the von Neumann algebra generated by $M$ and a unitary 
$\lambda$ satisfying $\lambda^m=1$ and $\lambda x\lambda^{-1}=\beta(x)$ for any $x\in M$. 
We extend $\rho$ to $P$ by $\trho(\lambda)=\lambda$. 
Then since $S$ and $T_0$ are invariant under $\beta$, we have $S\in (\id,\trho^2)$ and $T_0\in (\trho,\trho^2)$. 
We will compute the fusion rule of the fusion category generated by $\trho$.

Let 
$$G=\bigsqcup_{i=0}^p O_i$$ 
be the $\theta$-orbit decomposition of $G$ with $O_0=\{0\}$, and let $l_i=m/\# O_i$.  
For $0\leq i\leq p$, we set 
$$P_i=\sum_{g\in O_i}T_gT_g^*,$$
which is a projection in the fixed point algebra $M^\beta$. 
Since $M^\beta$ is a type III factor, there exists an isometry $V_i\in M^\beta$ whose range projection is $P_i$. 
We define $\sigma_i\in \End(M)$ by 
$$\sigma_i(x)=V_i^*(\sum_{g\in O_i}T_g\alpha_g(x)T_g^*)V_i.$$
By construction, we have 
$$[\sigma_i]=\bigoplus_{g\in O_i}[\alpha_g].$$
Note that $\beta\circ \sigma_i=\sigma_i\circ \beta$ holds for any $0\leq i\leq p$. 

\begin{lemma} The endomorphism $\sigma_i$ extends to $\tsigma_i\in \End(P)$ by 
$\tilde{\sigma_i}(\lambda)=V_i^*\lambda V_i=V_i^*\beta(V_i)\lambda$. 
\end{lemma}

\begin{proof}
Since the range projection of $V_i$ is invariant under $\beta$, the operator $V_i^*\lambda V_i$ is unitary. 
As 
$$V_i^*\lambda V_iV_i^*\lambda^kV_i=V_i^*\lambda P_i\lambda^kV_i=V_i^*P_i\lambda^{k+1}V_i,$$
we can show $(V_i^*\lambda V_i)^k=V_i^*\lambda^k V_i$ by induction, and so $(V_i^*\lambda V_i)^m=1$. 
For $x\in M$, we have 
\begin{align*}
V_i^*\lambda V_i\sigma_i(x)&=V_i^*\lambda P_i\sum_{g\in O_i}T_g\alpha_g(x)T_g^*V_i
=V_i^*\sum_{g\in O_i}T_{\theta(g)}\beta(\alpha_g(x))T_{\theta(g)}^*\lambda V_i\\
&=V_i^*\sum_{g\in O_i}T_{\theta(g)}\alpha_{\theta(g)}(\beta(x))T_{\theta(g)}^*V_iV_i^*\lambda V_i\\
&=\sigma_i(\beta(x))V_i^*\lambda V_i,
\end{align*}
which shows the statement. 
\end{proof}

For $i=0$, we can choose $V_0=T_0$, and so $\tsigma_0=\id$.

\begin{theorem}\label{orbifold2} Let the notation be as above. Then
\begin{itemize}
\item [(1)] Let $\hat{\beta}$ be the dual action of $\beta$. 
For $0\leq k,l\leq l_i-1$, we have $(\hbeta^k\tsigma_i\trho,\hbeta^l\tsigma_j\trho)=\delta_{k,l}\delta_{i,j}\C1.$ 
In particular, $\hbeta^k\tsigma_i\trho$ is irreducible. 
\item [(2)] 
$$[\trho^2]=[\id]\oplus\bigoplus_{i=0}^p[\tsigma_i\trho].$$
\item [(3)] If $\id\oplus\rho$ has a $Q$-system, so does $\id\oplus \trho$. 
\end{itemize}
\end{theorem}

\begin{proof} (1) It suffices to show the statement for $l=0$. 
Note that any element $X\in P$ has a unique expansion $X=\sum_{a=0}^{m-1}X_a\lambda^a$ with $X_a\in M$. 
Thus $X\in (\hbeta^k\tsigma_i\trho,\tsigma_j\trho)$ if and only if $X_a\in (\beta^a\sigma_i\rho,\sigma_j\rho)$ 
and $\zeta_m^kX_a\beta^a(V_i^*\beta(V_i))=V_i^*\beta(V_iX_a)$, where $\zeta_m=e^{2\pi i/m}$. 
Since 
\begin{align*}
\dim(\beta^a\sigma_i\rho,\sigma_j\rho) &= 
\dim(\overline{\sigma_j}\beta^a\sigma_i,\rho^2)
=\dim(\overline{\sigma_j}\sigma_i\beta^a,\rho^2)\\
&=\sum_{g\in O_i,\;h\in O_j}\dim(\alpha_{g-h}\beta^a,\id)
=\delta_{i,j}\# O_i,
\end{align*}
we have $X_a=0$ for $a\neq 0$ and $X_0$ has the following form:
$$X_0=V_i^*\sum_{g\in O_i}c_gT_gT_g^*V_i,\quad c_g\in \C.$$
Now the condition  $\zeta_m^kX_0V_i^*\beta(V_i)=V_i^*\beta(V_iX_0)$ is equivalent to 
$c_{g^\theta}=\zeta_m^{-1}c_g$. 
However, this implies $c_g=c_{\theta^{\#O_i}(g)}=\zeta_m^{-k\#O_i}c_g$. 
Since $0\leq k \leq l_i-1$, we have $\zeta_m^{-k\#O_i}\neq 1$ except for $k=0$, 
and we get the statement. 

(2) It suffices to show $V_i\in (\tsigma_i\trho,\trho)$. 
For $x\in M$, we have 
$$V_i\tsigma_i\trho(x)=P_i\sum_{g\in O_i}T_g\alpha_g(\rho(x))T_g^*V_i=T_g\alpha_g(\rho(x))T_g^*V_i=\rho^2(x)V_i=\trho^2(x)V_i.$$
For $\lambda$,  
$$V_i\tsigma_i\trho(\lambda)=P_i\lambda V_i=\lambda P_iV_i=\lambda V_i=\trho^2(\lambda)V_i,$$
which shows $V_i\in (\tsigma_i\trho,\trho^2)$ 

(3) The statement follows from $(\rho,\rho^2)=(\trho,\trho^2)=\C T_0$. 
\end{proof}

Let $M^\beta$ be the fixed point subalgebra of $M$ under $\beta$, and let $\Gamma=\hat{G}\rtimes_{\theta}\Z/m\Z$. 
Then we have $P=M^\beta\rtimes \Gamma$. 
Let $\iota_1:M\hookrightarrow P$ and $\iota_2:M^\beta\hookrightarrow M$ be the inclusion maps. 
Then we have 
$$[\hbeta^k\tsigma_i][\iota_1]=\bigoplus_{g\in O_i}[\iota_1][\alpha_g],$$
and 
$$[\hbeta^k\tsigma_i][\iota_1\iota_2]=\#O_i[\iota_1\iota_2].$$
This shows that $[\iota_1\iota_2]\overline{[\iota_1\iota_2]}$ contains $[\hbeta^k\tsigma_i]$ with multiplicity $\#O_i$. 
Since 
$$\sum_{i=0}^pl_i(\#O_i)^2=m\sum_{i=0}^p\#O_i=m\#G=\#\Gamma,$$
we get 
$$[\iota_1\iota_2]\overline{[\iota_1\iota_2]}=\bigoplus_{i=0}^p\bigoplus_{k=0}^{l_i-1}\#O_i[\hbeta^k\tsigma_i],$$
and $\{\hbeta^k\tsigma_i\}$ form the set of irreducibles in a fusion category equivalent to the representation category 
$\hat{\Gamma}$ of $\Gamma$. 
This shows that we can describe the fusion rules for the category generated by $\trho$ in terms of $\hat{\Gamma}$. 

\begin{remark} We can easily generalize the above argument to a subgroup of $\Aut(G)$ leaving the solution 
$(\epsilon_h(g),\eta_g,A_g(h,k))$ invariant. 
\end{remark}

\subsection{Accompanying solutions}
Let $M$, $\cC$, $G$, $\alpha$, $\rho$, $(\epsilon_h(g),\eta_g,A_g(h,k))$ be as in Section \ref{PEGHC}. 
For a subgroup $H\subset G$, we can extend $\rho$ to the crossed product $M\rtimes_\alpha H$ by $\tilde{\rho}(\lambda_h)=\lambda_{-h}$, 
where $\{\lambda_h\}_{h\in H}$ is the implementing unitary representation of $H$. 
It often happens that $\tilde{\rho}$ generates a generalized Haagerup category with a group whose 
order is the same as that of $G$ (see \cite{GS12}), 
which gives a new solution of the equations Eq.(\ref{cocycle})-Eq.(\ref{AAA}). 
We call it the accompanying solution of $(\epsilon_h(g),\eta_g,A_g(h,k))$ with respect to the subgroup $H$.  
In this subsection, we compute it in easy cases. 

Assume first that $G$ is an odd group and $H=G$. 
In this case the solution reduces to $(A(g,h),\eta)$. 
For $\chi\in \hat{G}$, we set 
$$\hat{T}_\chi=\frac{1}{\sqrt{|G|}}\sum_{g\in G}\inpr{g}{\chi}T_{2g}\lambda_{2g}.$$
Then $\hat{T}_\chi$ is an isometry in $(\hat{\alpha}_\chi\tilde{\rho},\tilde{\rho}^2)$, where $\hat{\alpha}$ is the dual action of $\alpha$, 
and $\hat{\alpha}_\tau(\hat{T}_\chi)=\hat{T}_{\chi+2\tau}$ holds. 
We have $S\in (\id,\tilde{\rho}^2)$, and 
\begin{align*}
\lefteqn{\sqrt{d}\hat{T}_\chi^*\rho(S)=\sqrt{\frac{d}{|G|}}\sum_{g\in G}\inpr{-g}{\chi}\lambda_{-2g}T_{2g}^*\rho(S)} \\
 &=\frac{1}{\sqrt{|G|}}\sum_{g\in G}\inpr{-g}{\chi}\lambda_{-2g}T_{2g}=\frac{1}{\sqrt{|G|}}\sum_{g\in G}\inpr{-g}{\chi}T_{-2g}\lambda_{-2g}=\hat{T}_\chi,\\
\end{align*}
\begin{align*}
\lefteqn{\sqrt{d}S^*\hat{\alpha}_\chi\circ \tilde{\rho}(\hat{T}_\chi^*)S
=\sqrt{\frac{d}{|G|}}\sum_{g\in G}\inpr{-g}{\chi}
\hat{\alpha}_\chi(\tilde{\rho}(\lambda_{-2g}T_{2g}^*)S)} \\
 &=\sqrt{\frac{d}{|G|}}\sum_{g\in G}\inpr{-g}{\chi}\hat{\alpha}_\chi(\lambda_{2g}\rho(T_{2g}^*)S)
 =\sqrt{\frac{d}{|G|}}\sum_{g\in G}\inpr{-g}{\chi}\hat{\alpha}_\chi(\alpha_{2g}\circ\rho(T_{2g}^*)S\lambda_{2g})\\
 &=\frac{\eta}{\sqrt{|G|}}\sum_{g\in G}\inpr{-g}{\chi}\hat{\alpha}_\chi(T_{2g}\lambda_{2g})=\eta\hat{T}_{\chi}. 
\end{align*}

\begin{theorem} When $G$ is an odd abelian group, the accompanying solution of $(A(g,h),\eta)$ with respect to $G$ is 
$(\hat{A}(\chi_1,\chi_2),\eta)$ with 
$$\hat{A}(\chi_1,\chi_2)=\frac{1}{|G|}\sum_{h,k\in G}A(2h,2k)\inpr{h}{\chi_2}\overline{\inpr{k}{\chi_1}}.$$
\end{theorem}

\begin{proof} The statement follows from 
\begin{align*}
\lefteqn{\hat{A}(\chi_1,\chi_2)=\hat{T}_{\chi_1+\chi_2}^*\hat{T}_{\chi_1}^*\tilde{\rho}(\hat{T}_0)\hat{T}_{\chi_2}} \\
 &=\frac{1}{|G|^2}\sum_{g,h,k,l\in G}\inpr{k}{\chi_2}\overline{\inpr{h}{\chi_1}\inpr{l}{\chi_1+\chi_2}}
 \lambda_{-2l}T_{2l}^*\lambda_{-2h}T_{2h}^*
 \tilde{\rho}(T_{2g}\lambda_{2g})T_{2k}\lambda_{2k} \\
 &=\frac{1}{|G|^2}\sum_{g,h,k,l\in G}\inpr{k}{\chi_2}\overline{\inpr{h}{\chi_1}\inpr{l}{\chi_1+\chi_2}}
 \lambda_{-2l}T_{2l}^*\lambda_{-2h}T_{2h}^*
 \rho(T_{2g})\lambda_{-2g}T_{2k}\lambda_{2k} \\
 &=\frac{1}{|G|^2}\sum_{g,h,k,l\in G}\inpr{k}{\chi_2}\overline{\inpr{h}{\chi_1}\inpr{l}{\chi_1+\chi_2}}
 \lambda_{-2h-2l}T_{4h+2l}^*T_{2h}^*
 \alpha_{-2g}\rho(T_{-2g})T_{-4g+2k}\lambda_{-2g+2k} \\
 &=\frac{1}{|G|^2}\sum_{g,h,k,l\in G}\inpr{k}{\chi_2}\overline{\inpr{h}{\chi_1}\inpr{l}{\chi_1+\chi_2}}
 A(2g+2h,-2g+2k)\delta_{h+l,-g+k}. \\
\end{align*}
\end{proof}

Now we consider the case $G=H=\Z_{2m}.$ 
We assume that the cocycle $\epsilon_h(g)$ is non-trivial, and we have $\epsilon_m(g)=(-1)^g$. 
In this case $\trho$ generates a generalized Haagerup category with $\widehat{\Z_{2m}}\cong \Z_{2m}$. 
For a natural number $n$, we denote $\zeta_n=e^{\frac{2\pi i }{n}}$. 
The dual action $\hat{\alpha}$ is given by $\hat{\alpha}_k(\lambda_l)=\zeta_{2m}^{kl}\lambda_l$. 
For $0\leq a<m$, we set 
$$\hat{T}_{2a}=\frac{1}{\sqrt{m}}\sum_{k=0}^{m-1}\zeta_m^{ak}T_{2k}\lambda_{2k},$$
$$\hat{T}_{2a+1}=\frac{\zeta_{4m}^{2a+1}}{\sqrt{m}}\sum_{k=0}^{m-1}\zeta_{2m}^{(2a+1)k}\alpha_k(T_1)\lambda_{2k+1}.$$
Then $\hat{T}_b\in (\hat{\alpha}_b\trho,\trho^2)$. 
We have $\hat{\alpha}_b(\hat{T}_{2a})=\hat{T}_{2a+2b}$ and 
$$\hat{\alpha}_1(\hat{T}_{2a+1})=\left\{
\begin{array}{ll}
\hat{T}_{2a+3} , &\quad 0\leq a<m-1 \\
-\hat{T}_1 , &\quad a=m-1
\end{array}
\right..
$$ 
We can easily compute the accompanying solution in this case too. 

For $G=\Z_2\times K$ with odd $K$, we can compute the accompanying solutions for $H=K,G$ 
in a similar way too.  
\section{Examples}
Under assumption of Eq.(\ref{Q2}), we have Eq.(\ref{AA-AA}) that is extremely useful for solving the whole equations 
(\ref{cocycle})-(\ref{AAA}) (see \cite{I01}, \cite{EG11}).  
When we do not assume Eq.(\ref{Q2}), our strategy to solve (\ref{cocycle})-(\ref{AAA}) is as follows. 
We first solve Eq.(\ref{Deg1})-(\ref{Deg3}), and we get Eq.(\ref{|A|1}). 
Next we try to solve Eq.(\ref{O2}) and (\ref{AAA}) with $p=q=x+y=0$. 
The author has checked that this strategy works at least up to $|G|\leq 6$ with help of Mathematica. 

\subsection{$G=\Z_2$}
When $G=\Z_2$, it is easy to show that there is no solution of (\ref{cocycle})-(\ref{AAA}) with trivial $\epsilon$, 
and the only remaining case is $\epsilon_{h}(g)=(-1)^{gh}$. 
In this case, there exists a unique solution $\eta_0=\eta_1=1$, 
$$A_0=\frac{1}{d-1}\left(
\begin{array}{cc}
d-2 &-1  \\
-1 &-1
\end{array}
\right),\quad 
A_1=\frac{1}{d-1}\left(
\begin{array}{cc}
d-2 &-1  \\
-1 &1
\end{array}
\right),
$$
with $d=1+\sqrt{2}$. 
This solution comes from the even part of the $A_7$ subfactor, which is discussed in \cite[Example 3.6]{I93}. 

\subsection{$G=\Z_3$}
Those solutions of (\ref{cocycle})-(\ref{AAA}) for $G=\Z_3$ satisfying (\ref{Q1}) are already given in \cite{I01}, 
and we look for those not satisfying (\ref{Q1}), which should exist as the corresponding 
fusion category was found in \cite{GS12} (see also \cite{EG15}). 
Since $G$ is an odd group, we can choose $\epsilon_h(g)$ satisfying $\epsilon_h(g)=1$. 
In this case neither $A_g(h,k)$ nor $\eta_g$ depends on $g$, which we denote $A(g,h)$ and $\eta$ respectively. 

We can set 
$$A=\left(
\begin{array}{ccc}
x_0 &\eta x_1 &\eta x_2  \\
\overline{\eta}x_1 &x_2 &y  \\
\overline{\eta}x_2 &\overline{y} &x_1 
\end{array}
\right),
$$  
with $x_0,x_1,x_2\in \R$, $y\in \C$, $\eta\in \T$ satisfying $\eta^3=1$. 
If $\eta\neq 1$, we have $x_0=y=0$, which contradicts (\ref{O2}).
Thus we have $\eta=1$. 

Eq.(\ref{Deg1'})-(\ref{Deg3'}) allow three different solutions up to group automorphism:
\begin{equation}\label{Z/3Z1}
(x_0,x_1,x_2)=\left\{
\begin{array}{ll}
(\frac{7-\sqrt{13}}{6},\frac{1-\sqrt{13}}{6},\frac{1-\sqrt{13}}{6}) , &  \\
(\frac{2-\sqrt{13}}{3},\frac{5-\sqrt{13}+\sqrt{6(1+\sqrt{13})}}{12},\frac{5-\sqrt{13}-\sqrt{6(1+\sqrt{13})}}{12}), &  \\
(0,\frac{3-\sqrt{13}+\sqrt{2(-1+\sqrt{13})}}{4},\frac{3-\sqrt{13}-\sqrt{2(-1+\sqrt{13})}}{4}). &
\end{array}
\right.
\end{equation}
Now (\ref{cocycle})-(\ref{hkshift}) are equivalent to 
$$A=\left(
\begin{array}{ccc}
x_0 & x_1 &x_2  \\
x_1 &x_2 &y  \\
x_2 &\overline{y} &x_1 
\end{array}
\right),
$$  
\begin{equation}\label{Z/3Z2}
x_1^2+x_2^2+|y|^2=1,
\end{equation}
\begin{equation}\label{Z/3Z3}
x_1x_2+x_1\overline{y}+x_2y=0.
\end{equation}  
By (\ref{Z/3Z3}), we have either  $x_1=x_2$ or $y\in \R$. 
Assume $x_1=x_2$ first. 
Then we have $x_0=(d-2)/(d-1)$, $x_1=x_2=-1/(d-1)$, which satisfies (\ref{Q1}), and 
$$y=\frac{1\pm i\sqrt{4d-1}}{2(d-1)}.$$ 
We can check that $(x_0,x_1,x_2,y)$ satisfy the whole system of equations, and we conclude that there exists 
a unique solution, up to group automorphism, satisfying (\ref{Q1}). 

Assume $y\in \R$ now. 
Then $y$ satisfies 
$$y^2=1-x_1^2-x_2^2,$$
$$y=-\frac{x_1x_2}{x_1+x_2},$$
and the only second solutions in (\ref{Z/3Z1}) is allowed with $y=(1+\sqrt{13})/2$. 
This solution is the accompanying solution of the previous one. 

\begin{theorem} There exists exactly two solutions of (\ref{cocycle})-(\ref{AAA}) for $G=\Z_3$ 
up to group automorphism. 
One of them satisfies (\ref{Q1}), and the other does not. 
\end{theorem}


\subsection{$G=\Z_4$} When $G=\Z_4$, we have $d=2+\sqrt{5}$, and 
we may assume that  $\epsilon_{1}(3)=\epsilon_3(1)=\epsilon$, $\epsilon_2(g)=\epsilon^g$ with $\epsilon\in \{1,-1\}$, and $\epsilon_h(g)=1$ otherwise. 
All the solutions of Eq.(\ref{Deg1})-(\ref{Deg3}) satisfy $x_{g,h}=x_{0,h}$, and they are as follows up to group automorphism: 
\begin{equation}\label{Z/4Z1}
(x_0,x_1,x_2,x_3)=\left\{
\begin{array}{ll}
( \frac{5 - \sqrt{5}}{4},  \frac{1 - \sqrt{5}}{4},\frac{1 - \sqrt{5}}{4},\frac{1 - \sqrt{5}}{4}) , &\quad  \\
(\frac{ 2 - \sqrt{5}}{2}, 
\frac{1 - \sqrt{5} + \sqrt{2(-1+\sqrt{5})}}{4}, 
\frac{1}{2}, 
\frac{1 - \sqrt{5} - \sqrt{2 (-1+\sqrt{5})}}{4}). &\quad
\end{array}
\right.
\end{equation}
where $x_{g}=x_{h,g}$. 
Since $x_0\neq 0$, we have $\eta_g=1$ for any $g\in G$. 

Let $y=A_0(1,2)$. 
Then 
$$A_0=\left(
\begin{array}{cccc}
x_0 &x_1 &x_2 &x_3  \\
x_1 &x_3 &y &\epsilon y  \\
x_2 &\overline{y} &x_2 &y \\
x_3 &\epsilon \overline{y} &\overline{y} &x_1 
\end{array}
\right),\quad
A_1=\left(
\begin{array}{cccc}
x_0 &x_1 &x_2 &x_3  \\
x_1 & x_3 &y & y  \\
x_2 &\overline{y} &\epsilon x_2 &\epsilon y \\
x_3 & \overline{y} &\epsilon \overline{y} &\epsilon x_1 
\end{array}
\right),
$$
$$A_2=\left(
\begin{array}{cccc}
x_0 &x_1 &x_2 &x_3  \\
x_1 &x_3 &\epsilon y & y  \\
x_2 &\epsilon \overline{y} &x_2 &\epsilon y \\
x_3 & \overline{y} &\epsilon \overline{y} &x_1 
\end{array}
\right),\quad
A_3=\left(
\begin{array}{cccc}
x_0 &x_1 &x_2 &x_3  \\
x_1 &\epsilon x_3 &\epsilon y & y  \\
x_2 &\epsilon \overline{y} &\epsilon x_2 & y \\
x_3 & \overline{y} & \overline{y} & x_1 
\end{array}
\right),
$$
and Eq.(\ref{O2}) are equivalent to 
\begin{equation}\label{Z/4Z2}
x_1^2+x_3^2+2|y|^2=2x_2^2+2|y|^2=1,
\end{equation}
\begin{equation}\label{Z/4Z3}
x_1y+x_3\overline{y}+x_1x_3+\epsilon |y|^2=0,
\end{equation}
\begin{equation}\label{Z/4Z4}
(1+\epsilon)x_2(y+\overline{y})=0,
\end{equation}
\begin{equation}\label{Z/4Z5}
(y+\epsilon \overline{y})(x_1+\epsilon x_3)=0,
\end{equation}
\begin{equation}\label{Z/4Z6}
2x_2^2+\epsilon(y^2+\overline{y}^2)=0.
\end{equation}
Eq.(\ref{Z/4Z3}) implies that either $x_1=x_3$ or $y\in \R$. 

Assume that $x_1=x_3$. 
Then $x_0=1-\frac{1}{d-1}$, and $x_1=x_2=x_3=-\frac{1}{d-1}$, and so Eq.(\ref{Q2}) is satisfied. 
Solving Eq.(\ref{Z/4Z1})-(\ref{Z/4Z6}), we get $\epsilon=-1$ and 
$$y=-\frac{1}{2}\pm i\frac{1}{\sqrt{2(d-1)}}.$$
We can show that this satisfies the whole equations. 

Assume now that $y$ is real. 
Then Eq.(\ref{Z/4Z1})-(\ref{Z/4Z6}) allow only the second solution of Eq.(\ref{Z/4Z1}), and 
they imply $\epsilon=-1$ and $y=\frac{-1}{2}$. 
This is the accompanying solution of the previous one.

\begin{theorem} There exist exactly two solutions, up to gauge equivalence and group automorphism, 
of (\ref{cocycle})-(\ref{AAA}) for $G=\Z_4$, and they satisfy $\eta_g=1$ and $\epsilon_1(2)=-1$. 
They are mutually accompanying solutions, and only one of them satisfies Eq.(\ref{Q1}), which is given by  
$$A_0=\frac{1}{d-1}\left(
\begin{array}{cccc}
d-2 &-1 &-1 &-1  \\
-1 &-1 &z &-z  \\
-1 &\overline{z} &-1 &z  \\
-1&-\overline{z} &\overline{z} &-1 
\end{array}
\right),\quad 
A_1=\frac{1}{d-1}\left(
\begin{array}{cccc}
d-2 &-1 &-1 &-1  \\
-1 &-1 &z &z  \\
-1 &\overline{z} &1 &-z  \\
-1&\overline{z} &-\overline{z} &1 
\end{array}
\right),
$$
$$A_2=\frac{1}{d-1}\left(
\begin{array}{cccc}
d-2 &-1 &-1 &-1  \\
-1 &-1 &-z &z  \\
-1 &-\overline{z} &-1 &-z  \\
-1&\overline{z} &-\overline{z} &-1 
\end{array}
\right),\quad
A_3=\frac{1}{d-1}\left(
\begin{array}{cccc}
d-2 &-1 &-1 &-1  \\
-1 &1 &-z &z  \\
-1 &-\overline{z} &1 &z  \\
-1&\overline{z} &\overline{z} &-1 
\end{array}
\right),
$$
$$z=-\frac{d-1}{2}\pm i\sqrt{\frac{d-1}{2}}=-\frac{1+\sqrt{5}}{2}\pm i\sqrt{\frac{1+\sqrt{5}}{2}}.$$
This gives a subfactor of index $3+\sqrt{5}$ whose canonical endomorphism is $\id\oplus \rho$ (see \cite{MP15-1}).  
\end{theorem}

Note that the above solution satisfies $\epsilon_2(2)=1$, to which Theorem \ref{orbifold1} applies. 
Using the notation in Section \ref{orbiforld} with $G=\Z_4$ and $z=2$, we get 
$$[\trho^2]=[\id]\oplus 2[\trho]\oplus 2[\talpha\trho],$$
$$[\talpha\trho]=[\trho\talpha],\quad [\talpha^2]=[\id],$$
where we denote $\talpha=\talpha_1$ for simplicity. 
Since $\talpha^2=\Ad \lambda$ and $\talpha(\lambda)=\epsilon_2(1)\lambda=-\lambda$, 
the homomorphism $\Z_2\ni 1\mapsto [\alpha]\in \Out(M)$ has an non-trivial obstruction in $H^3(\Z_2,\T)$. 
Since $\id\oplus \trho$ has a $Q$-system, we get the following result. 

\begin{cor} There exists a subfactor of index $3+\sqrt{5}$ with one of the principal graph as in Figure 1 (see \cite{MP15-2}). 

\begin{figure}
\begin{picture}(100,50)(0,0)
\put(10,10){\makebox(0,0){$\bullet$}}
\put(20,10){\makebox(0,0){$\bullet$}}
\put(30,10){\makebox(0,0){$\bullet$}}
\put(40,20){\makebox(0,0){$\bullet$}}
\put(40,0){\makebox(0,0){$\bullet$}}
\put(50,10){\makebox(0,0){$\bullet$}}
\put(60,10){\makebox(0,0){$\bullet$}}
\put(70,10){\makebox(0,0){$\bullet$}}
\put(10,10){\line(1,0){10}}
\put(20,10){\line(1,0){10}}
\put(30,10){\line(1,1){10}}
\put(30,10){\line(1,-1){10}}
\put(40,20){\line(1,-1){10}}
\put(40,0){\line(1,1){10}}
\put(50,10){\line(1,0){10}}
\put(60,10){\line(1,0){10}}
\end{picture}
\caption{$3^{\Z_4}$/$\Z_2$}
\end{figure}
\end{cor}

\subsection{$G=\Z_2\times \Z_2$}
We use the parametrization $G=\Z_2\times \Z_2=\{0,a,b,c\}$ of the elements of the group $G$, and   
fix the order $0,a,b,c$, to express every function $f$ on $G\times G$ by the matrix 
$$f=\left(
\begin{array}{cccc}
f(0,0) &f(0,a) &f(0,b) &f(0,c)  \\
f(a,0) &f(a,a) &f(a,b) &f(a,c)  \\
f(b,0) &f(b,a) &f(b,b) &f(b,c)  \\
f(c,0) &f(c,a) &f(c,b) &f(c,c) 
\end{array}
\right). 
$$
There exists a unique solution of Eq.(\ref{Deg1})-(\ref{Deg3}), which is 
$$x_{g,0}=\frac{\sqrt{5}-1}{2}=1-\frac{1}{d-1},$$
$$x_{g,a}=x_{g,b}=x_{g,c}=\frac{1-\sqrt{5}}{4}=-\frac{1}{d-1}.$$
Since $x_{g,0}\neq 0$, we have $\eta_g=1$, and (\ref{2hshift}) with $g=0$ shows that $\chi(h,g):=\epsilon_h(g)$ is a bicharacter. 
Therefore we can put 
$$A_0=\frac{1}{d-1}\left(
\begin{array}{cccc}
d-2 &-1 &-1 &-1  \\
-1 &-1&z &\overline{z}  \\
-1 &\overline{z} &-1&z  \\
-1 &z &\overline{z} &-1 
\end{array}
\right).
$$

\begin{theorem} The solutions of (\ref{cocycle})-(\ref{AAA}) for $G=\Z_2\times \Z_2$ are 
$$A_0=\frac{1}{d-1}\left(
\begin{array}{cccc}
d-2 &-1 &-1 &-1  \\
-1 &-1&z &\overline{z}  \\
-1 &\overline{z} &-1&z  \\
-1 &z &\overline{z} &-1 
\end{array}
\right),
$$
$$
A_a=\frac{1}{d-1}\left(
\begin{array}{cccc}
d-2 &-1 &-1 &-1  \\
-1 &1&\chi(a,b)z &\chi(a,c)\overline{z}  \\
-1 &\chi(a,b)\overline{z} &-\chi(b,a)&-z  \\
-1 &\chi(a,c)z &-\overline{z} &-\chi(c,a) 
\end{array}
\right),
$$
$$A_b=\frac{1}{d-1}\left(
\begin{array}{cccc}
d-2 &-1 &-1 &-1  \\
-1 &-\chi(a,b)&\chi(b,a)z &-\overline{z}  \\
-1 &\chi(b,a)\overline{z} &1&\chi(b,c)z  \\
-1 &-z &\chi(b,c)\overline{z} &-\chi(c,b) 
\end{array}
\right),
$$
$$A_c=\frac{1}{d-1}\left(
\begin{array}{cccc}
d-2 &-1 &-1 &-1  \\
-1 &-\chi(a,c)&-z &\chi(c,a)\overline{z}  \\
-1 &-\overline{z} &-\chi(b,c)&\chi(c,b)z  \\
-1 &\chi(c,a)z &\chi(c,b)\overline{z} &1 
\end{array}
\right),
$$
$$\chi=\left(
\begin{array}{cccc}
1 &1 &1 &1  \\
1 &-1 &s &-s  \\
1 &-s &-1 &s  \\
1 &s &-s &-1 
\end{array}
\right),
$$
where $s\in \{1,-1\}$, $z\in \{\sqrt{d},-\sqrt{d}, i\sqrt{d},-i\sqrt{d}\}$. 
By gauge transformations and group automorphisms, those solutions with $z\in \R$ (respectively $z\in i\R$) are transformed to each other. 
Up to gauge transformations, group automorphisms, and $H^2(G,\T)$-actions, there is only one equivalence class of the 
solutions. 

Let $\cC$ be the generalized Haagerup category generated by $\rho$. 
Then $\Out(\cC)$ is isomorphic to the alternating group $\fA_4=(\Z_2\times \Z_2)\rtimes \Z_3$. 
\end{theorem}

\begin{proof} It is routine work to obtain the solutions. 
The gauge transformations can switch only $z$ and $-z$ and leave $s$ invariant. 
Note that the automorphism group of $\Z_2\times \Z_2$ is the permutation group $\fS_3$. 
The even permutations leave the solutions invariant, and odd permutations switch $s$ and $-s$ and switch $z$ and $\overline{z}$. 
 
The non-trivial element of $H^2(\Z_2\times \Z_2,\T)\cong \Z_2$ is represented by a cocycle $\omega\in Z^2(\Z_2\times \Z_2,\T)$ 
given as follows 
$$\omega=\left(
\begin{array}{cccc}
1 &1 &1 &1  \\
1 &1 &i &-i  \\
1 &-i&1 &i  \\
1 &i &-i&1 
\end{array}
\right),
$$
which satisfies $\omega(g,h)\omega(h,g)=1$. 
Thus in Theorem \ref{H2action}, we may choose $\mu(g)=1$, and the action of $H^2(\Z_2\times \Z_2,\T)$ on the solutions 
is given by $\epsilon_h(g)\mapsto \epsilon_h(g)b_\omega(g,h)$ and $A_g(h,k)\mapsto A_g(h,k)\overline{\omega(g+k,h)\omega(h,g)}$, 
where 
$$b_\omega=\left(
\begin{array}{cccc}
1 &1 &1 &1  \\
1 &1 &-1 &-1  \\
1 &-1 &1 &-1  \\
1 &-1 &-1 &1 
\end{array}
\right).
$$
Thus the action of $[\omega]$ affects the parameters as $s\mapsto -s$ and $z\mapsto iz$ up to gauge transformation. 

Finally we compute the action of $G/2G=G$. 
We use the notation in Lemma \ref{translation}. 
Note that $\chi(g)$ is identified with $\chi(g,p)$.  
For the action of $p$ on $s$, we have 
$$\epsilon'_h(g)=\epsilon_h(p+g)\chi(h)=\chi(h,p+g)\chi(h,p)=\chi(h,g)=\epsilon_h(g),$$
which shows that the action of $p$ on $s$ is trivial. 
Recall that $\zeta(g)$ is a square root of $\chi(g,p)$. 
Thanks to Lemma \ref{translation}, we have 
$$A'_0(a,b)=A_p(a,b)\zeta(p+c)\zeta(p+a)\overline{\zeta(p+b)\zeta(p)\chi(a,p)},$$
and 
\begin{align*}
\lefteqn{(\zeta(p+c)\zeta(p+a)\overline{\zeta(p+b)\zeta(p)\chi(a,p)})^2} \\
 &=\chi(p+c,p)\chi(p+a,p)\chi(p+b,p)\chi(p,p)=\chi(0,p)=1.
\end{align*}
This implies $A'_0(a,b)=\pm A_0(a,b)$ for any $p\in G$, which shows that $p$ acts on $z$ trivially up to gauge transformation. 

The above computation shows that the group
$$\Gamma =(H^2(G,\T)\times G/2G)\rtimes \Aut(G)=\Z_2\times ((\Z_2\times \Z_2)\rtimes \fS_3)$$
acts on the gauge equivalence classes of the solutions of Eq.(\ref{Deg1})-(\ref{Deg3}) transitively, and the point stabilizer is 
$$(\Z_2\times \Z_2)\rtimes \Z_3,$$
which is isomorphic to the alternating group $\fA_4$. 
\end{proof}

Pinhas Grossman is the first to obtain the outer automorphism group $\Out(\cC)\cong \fA_4$. 
He also determined that the order of the Bauer-Picard group of $\cC$ is 360. 

Let $\theta\in \Aut(G)$ be given by $\theta(a)=b$, $\theta(b)=c$, $\theta(c)=a$. 
Then the above solutions are invariant under $\theta$, to which Theorem \ref{orbifold2} applies. 
We use the notation in Section \ref{orbiforld} for this $\theta$. 
Then 
$$\Gamma=\hat{G}\rtimes_\theta \Z_3\cong (\Z_2\times \Z_2)\rtimes \Z_3\cong \fA_4.$$   
For simplicity, we denote $[\tsigma]=[\tsigma_1]$. 
Then
$$[\tsigma^2]=[\id]\oplus [\hbeta]\oplus[\hbeta^2]\oplus 2[\tsigma].$$
Thus we get the following fusion rule: 
$$[\trho^2]=[\id]\oplus [\trho]\oplus [\tsigma\trho],$$
$$[\tsigma\trho][\trho]=[\tsigma]\oplus[\tsigma\trho]\oplus[\tsigma^2][\trho]=[\tsigma]\oplus 
3[\tsigma\trho]\oplus [\trho]\oplus[\hbeta\trho]\oplus[\hbeta^2\trho].$$
Since $\id\oplus \trho$ has a $Q$-system, we get the following result first shown by Morrison-Penneys \cite{MP15-1}: 

\begin{cor} There exists a subfactor for the 4442 graph. 
\end{cor}

Recently the generalized Haagerup category for $\Z_2\times \Z_2$ attracts attention of specialists 
(see \cite{X16}) and we finish this section with the following statement. 

\begin{prop} Let $\cC$ be a C$^*$-quadratic category with $(\Z_2\times \Z_2,\id,1)$. 
If $\fc^{0,3}(\cC)$ vanishes (i.e. the associator for the group of invertible objects is trivial), 
so does $\fc^{1,2}(\cC)$, and in consequence $\cC$ is a generalized Haagerup category. 
\end{prop}

\begin{proof} We may assume $\cC\subset \End_0(M)$ for a type III factor, and 
$$\cO(\cC)=\{[\alpha_g]\}_{g\in G}\sqcup \{[\alpha_g][\rho]\}_{g\in G},$$
where $G=\Z_2\times \Z_2$. 
Since $\fc^{0,3}(\cC)=0$, we may assume that $\alpha$ is a $G$-action, and there exist 
$U_g\in \cU(M)$ and $\omega\in Z^2(G,\T)$ satisfying 
$$\rho\circ \alpha_g=\Ad U_g\circ \alpha_g\circ \rho,\quad U_g\alpha_g(U_h)=\omega(g,h)U_{g+h}.$$
Let $\iota: M\hookrightarrow M\rtimes_\alpha G$ be the inclusion map, and let $\hat{\alpha}$ 
be the dual action of $\alpha$, which is an action of $\hat{G}$ on $M\rtimes_\alpha G$. 
We first compute the algebra structure of  $\cA:=(\iota\rho\biota,\iota\rho\biota)$. 

We set $\begin{xy}(0,5)*{g+h},(-5,-5)*{g},(5,-5)*{h},
{(0,3) \ar @{-} (0,0)},{(0,0) \ar @{-} (-5,-3)}, 
{(0,0) \ar @{-} (5,-3)}\end{xy}=1$, 
$\begin{xy}(0,-5)*{g+h},(-5,5)*{g},(5,5)*{h},
{(0,-3) \ar @{-} (0,0)},{(0,0) \ar @{-} (-5,3)}, 
{(0,0) \ar @{-} (5,3)}\end{xy}=1$,and choose 
\begin{xy}(0,5)*{\iota},(0,-5)*{\iota},(5,-5)*{g},
{(0,3) \ar @{-} (0,-3)},{(0,0) \ar @{-} (5,-3)}\end{xy} and 
\begin{xy}(0,5)*{\biota},(0,-5)*{\biota},(-5,-5)*{g},
{(0,3) \ar @{-} (0,-3)},{(0,0) \ar @{-} (-5,-3)}\end{xy} satisfying 
$$\begin{xy}(0,10)*{\iota},(0,-10)*{\iota},(10,-10)*{g+h},
{(0,8) \ar @{-} (0,-8)},{(0,6) \ar @{-} (10,-4)^g},
{(0,2) \ar @{-} (10,-4)_h},
{(10,-4) \ar @{-} (10,-8)}\end{xy}=
\begin{xy}(0,10)*{\iota},(0,-10)*{\iota},(10,-10)*{g+h},
{(0,8) \ar @{-} (0,-8)},{(0,0) \ar @{-} (10,-8)}\end{xy},\quad 
\begin{xy}(0,10)*{\biota},(0,-10)*{\biota},(-10,-10)*{g+h},
{(0,8) \ar @{-} (0,-8)},{(0,6) \ar @{-} (-10,-4)^g},
{(0,2) \ar @{-} (-10,-4)_h},
{(-10,-4) \ar @{-} (-10,-8)}\end{xy}=
\begin{xy}(0,10)*{\biota},(0,-10)*{\biota},(-10,-10)*{g+h},
{(0,8) \ar @{-} (0,-8)},{(0,0) \ar @{-} (-10,-8)}\end{xy}.$$

Let $X_g=\begin{xy}(0,0)*+[F]{U_g},
(-7,10)*{\iota},(0,10)*{g},(7,10)*{\biota},
(-7,-10)*{\iota},(0,-10)*{\rho},(7,-10)*{\biota},
{(-7,8) \ar @{-} (-7,-8)},{(7,8) \ar @{-} (7,-8)},
{(0,8) \ar @{-} (1,3)},{(-1,-3) \ar @{-} (0,-8)},
{(-7,6) \ar @{-} (-1,3)^g},{(1,-3) \ar @{-} (7,-6)_g}
\end{xy}\in \cA$. 
In view of the linear isomorphism 
$$\cA\cong (\biota\iota\rho,\rho\biota\iota)\cong 
((\bigoplus_{g\in G} \alpha_g)\rho,\rho(\bigoplus_{g\in G} \alpha_g)),$$
we see that $\{X_g\}_{g\in G}$ forms a basis of $\cA$. 
Since 
$$X_gX_h=
\begin{xy}(0,7)*+[F]{U_h},(0,-7)*+[F]{U_g},
{(-10,16) \ar @{-} (-10,-16)},{(10,16) \ar @{-} (10,-16)},
{(0,16) \ar @{-} (2,11)},{(-10,13) \ar @{-} (-2,11)},
{(-2,3) \ar @{-} (2,-3)},
{(-10,0) \ar @{-} (-2,-3)},{(2,3) \ar @{-} (10,0)},
{(0,-16) \ar @{-} (-2,-11)},{(10,-13) \ar @{-} (2,-11)},
\end{xy}=\begin{xy}(0,7)*+[F]{U_{h}},(0,-7)*+[F]{U_{g}},
{(-11,16) \ar @{-} (-11,-16)},{(11,16) \ar @{-} (11,-16)},
{(0,16) \ar @{-} (2,11)},{(-11,13) \ar @{-} (-2,11)},
{(-2,3) \ar @{-} (2,-3)},
{(-9,12) \ar @{-} (-2,-3)},{(2,3) \ar @{-} (9,-12)},
{(0,-16) \ar @{-} (-2,-11)},{(11,-13) \ar @{-} (2,-11)},
\end{xy} 
=\begin{xy}(0,0)*+[F]{U_g\alpha_g(U_h)},
{(-11,16) \ar @{-} (-11,-16)},{(11,16) \ar @{-} (11,-16)},
{(0,16) \ar @{-} (1,3)},{(-1,-3) \ar @{-} (0,-16)},
{(-11,12) \ar @{-} (-1,3)},{(1,-3) \ar @{-} (11,-12)}\end{xy}=\omega(g,h)X_{g+h},$$
and $\cA$ is isomorphic to the twisted group algebra $\C_\omega G$. 

We assume that the cohomology class $[\omega]$ is not trivial in $H^2(G,\T)$, and deduce contradiction. 
The algebra $\cA$ is noncommutative with $\dim \cA=4$, and it is isomorphic to the 2 by 2 matrix algebra. 
Thus there exists an irreducible $\sigma\in \End_0(M\rtimes_\alpha G)$ satisfying 
$[\iota\rho\biota]=2[\sigma]$. 
On the other hand, we have 
\begin{align*}
[\iota\rho\biota][\iota\rho\biota]&=[\iota\rho(\bigoplus_{g\in G}\alpha_g)\rho\biota]
=\sum_{g\in G}[\iota\alpha_g\rho^2\biota]=4[\iota\rho^2\biota]=4([\iota\biota]+\sum_{g\in G}[\iota\alpha_g\rho\biota]) \\
 &=4(\sum_{\chi\in \hat{G}})[\hat{\alpha}_\chi]+4[\iota\rho\biota])=4(\sum_{\chi\in \hat{G}})[\hat{\alpha}_\chi]+8[\sigma]),
\end{align*}
and 
$$[\sigma][\sigma]=\sum_{\chi\in \hat{G}}[\hat{\alpha}_\chi]+8[\sigma].$$
This means that $\sigma$ generates a near-group category with group $\hat{G}\cong \Z_2\times \Z_2$ and 
multiplicity 8. 
However, such a category does not exist (see \cite[Theorem 10.24]{I15}). 
\end{proof}
\section{Appendix} Let $(\epsilon_{h}(g),\eta_g,A_g(h,k))$ be a solution of the polynomial equations 
Eq.(\ref{cocycle})-Eq.(\ref{AAA}), and let $\cO_{n+1}$, $\alpha$, and $\rho\in \End(\cO_{n+1})$ be 
as in Section \ref{Reconstruction}. 
As in \cite{I93} and \cite[Appendix]{I15}, we introduce a weighted gauge action $\gamma$ on $\cO_{n+1}$ 
by $\gamma_t(S)=e^{2i t}S$, and $\gamma_t(T_g)=e^{i t}T_g$, and we denote by $\varphi$ the unique KMS state 
of $\gamma$. 
Since $\varphi\circ \rho$ is a KMS state for $\gamma$ too, we have $\varphi(\rho(x))=\varphi(x)$. 
We denote by $M$ the weak closure of $\cO_{n+1}$ in the GNS representation for $\varphi$. 
We still use the same symbols $\alpha,\rho, \varphi$ for their extension to $M$. 
Let $E_\rho(x)=\rho(S^*\rho(x)S)$ for $x\in M$, then $E_\rho$ is a $\varphi$-preserving 
normal conditional expectation from $M$ onto $\rho(M)$ (see \cite{I93}). 

If $\{h\in G_2|\; \chi_g(h)=1,\quad \forall g\in G\}\neq \{0\}$,  
the fusion category in $\End_0(M)$ generated by $\alpha$ and $\rho$ does not give 
the original data $(\epsilon_{h}(g),\eta_g,A_g(h,k))$ back because $\alpha$ is not a faithful action of $G$. 
In this section, we give a remedy for this problem by using a free product method. 
For the basics of free products of von Neumann algebras with general faithful normal states, 
the reader is referred to \cite{U11} and references therein.  

We set $M_0=M$, $\psi_0=\varphi$, $\rho_0=\rho$. 
Let $\tau$ be the trace on $\ell^\infty(G)$ defined by 
$$\tau(f)=\frac{1}{|G|}\sum_{g\in G}f(g).$$
We set $(M_i,\psi_i)=(\ell^\infty(G),\tau)$ for $i=1,2$. 
Let 
$$(\tilde{M},\psi)=(M_0,\psi_0)*(M_1,\psi_1)*(M_2,\psi_2),$$ 
be the free product von Neumann algebra, which is a factor of type III$_{\frac{1}{d}}$ 
(see \cite[Theorem 3.4]{U11}). 
Although $\tilde{M}$ is not hyperfinite, it is enough for our purpose of constructing a generalized Haagerup 
category having $(\epsilon_{h}(g),\eta_g,A_g(h,k))$. 
Let $\iota_i:M_i\to \tilde{M}$ be the embedding map for $i=0,1,2$. 
Whenever there is no possibility of confusion, we suppress $\iota_i$. 

Let $\theta :G\to \Aut(\ell^\infty(G))$ be the left translation. 
We define a $G$-action $\talpha:G\to \Aut(\tilde{M})$ by 
$\talpha_g=\alpha_g*\theta_g*\theta_{-g}$, 
which is an outer action of $G$ on $\tilde{M}$. 
We define unital homomorphisms $\rho_i:N_i\to \tilde{M}$ for $i=1,2$ by 
$\rho_1(\iota_1(f))=\iota_2(f)$ and 
$$\rho_2(\iota_2(f))=S\iota_1(f)S^*+\sum_{g\in G}T_g\talpha_g(\iota_2(f))T_g^*.$$
Then $\psi(\rho_1(f))=\psi_1(f)$ holds, and the KMS condition of $\psi$ implies 
$\psi(\rho_2(f))=\psi_2(f)$. 
Recall that we have $\psi(\rho_0(x))=\psi_0(x)$. 

\begin{lemma} With the above notation, the three von Neumann algebras $\rho_i(M_i)$, $i=0,1,2$,  
are free independent with respect to $\psi$. 
\end{lemma}

\begin{proof}
It suffices to show the free independence of the three sets:
$$\rho_0(M_0),\; M_2,\;SM_1S^*+\sum_{g\in G}T_gM_2T_g^*.$$
Note that we have $S,T_g\in \ker \psi_0$. 
Let $x\in \ker \psi_0$. 
The KMS condition of $\psi_0$ implies 
$$\psi_0(S^*\rho(x)S)=d^2\psi_0(\rho(x)SS^*)=d^2\psi_0(\rho(x)E_\rho(SS^*))=\psi_0(\rho(x))=0,$$
$$\psi_0(T_g^*\rho(x)T_g)=d\psi_0(\rho(x)T_gT_g^*)=d\psi_0(\rho(x)E_\rho(T_gT_g^*))=\psi_0(\rho(x))=0.$$
Since $E_\rho(T_g)=0$, we also have $\psi_0(\rho(x)T_g)=\psi_0(T_g^*\rho(x))=0$. 
Using these conditions, we claim that any alternative word of $\rho_0(\ker \psi_0)$ and $\rho_2(\ker \psi_2)$ 
is a linear combination of words of the form $y_1y_2\cdots y_n$ such that 
$y_i\in \ker\psi_{j(i)}$ with $j(i)\neq j(i+1)$ and $j(1),j(n)\neq 2$. 
It is clear that the claim implies the lemma. 

We denote by $\cW$ the set of words $z=z_1z_2\cdots z_n$ with $z_i\in M_{j(i)}$ and $j(i+1)\neq j(i)$ 
satisfying the following properties. 
\begin{itemize}
\item[(1)] when $j(i)=1,2$, we have $z_i\in \ker \psi_{j(i)}$,
\item[(2)] when $j(i)=0$ and $i\neq 0,n$, the element $z_i\in M_0$ belongs to either of the following set: 
$\ker\psi_0$, $T_g^*\rho_0(\ker\psi_0)S$, $S^*\rho_0(\ker\psi_0)T_g$. 
\item[(3)] $z_1$ belongs to either of the following set: $\ker\psi_1$, $\rho_0(\ker \psi_0)$, $\rho_0(\ker\psi_0)S$. 
\item[(4)] $z_n$ belongs to either of the following set: $\ker\psi_1$, $\rho_0(\ker \psi_0)$, $S^*\rho_0(\ker\psi_0)$. 
\end{itemize}
We define $l(z)$ by the number of $1\leq i\leq n$ with 
\begin{equation}\label{word}
z_i\in \rho_0(\ker\psi_0)S\cup S^*\rho_0(\ker\psi_0)\cup\bigcup_{g\in G}T_g^*\rho_0(\ker\psi_0)S \cup 
\bigcup_{g\in G}S^*\rho_0(\ker\psi_0)T_g.\end{equation}
Note that any alternative word of $\rho_0(\ker \psi_0)$ and $\rho_2(\ker \psi_2)$ is a linear combination 
of words in $\cW$. 

We prove that any word $z=z_1z_2\cdots z_n\in \cW$ is a linear combination of words of the desired form 
by induction of $l(z)$, which will finish the proof. 
When $l=0$, the word $z$ is of the desired form. 
Assume that the statement holds for $l=L$, and assume $l(z)=L+1$. 
Then there exists $i$ with Eq.(\ref{word}). 
Assume $i=1$ and $z_1\in \rho_0(\ker\psi_0)S$ first. 
Then $z_2\in \ker\psi_1$, and 
$$z_1z_2\cdots z_n=\overset{\circ}{z_1}z_2\cdots z_n+\psi_0(z_1)z_2\cdots z_n,$$
where $\overset{\circ}{z_1}=z_1-\psi_0(z_1)$. 
Applying the induction hypothesis to $\overset{\circ}{z_1}z_2\cdots z_n$ and $z_2\cdots z_n$, 
we see that $z$ is a linear combination of words of the desired form. 
The case with $i=n$ can be handled in a similar way. 
Assume that $1<i<n$ and $z_i\in T_g^*\rho_0(\ker \psi_0)S$. 
Then $z_{i-1}\in \ker\psi_2$ and $z_{i+1}\in\ker\psi_1$, and 
$$z_1z_2\cdots z_n=z_1\cdots z_{i-1}\overset{\circ}{z_i}z_{i+1}\cdots z_n
+\psi_0(z_i)z_1\cdots z_{i-1}z_{i+1}\cdots z_n.$$
Applying the induction hypothesis to $z_1\cdots z_{i-1}\overset{\circ}{z_i}z_{i+1}\cdots z_n$ and 
$z_1\cdots z_{i-1}z_{i+1}\cdots z_n$, we see that $z$ is a linear combination of words of the desired form.  
The case with $z_i\in S^*\rho_0(\ker \psi_0)T_g$ can be handled in the same way. 
\end{proof}

Thanks to the above lemma, we can defined an endomorphism $\trho\in \End(\tilde{M})$ extending 
$\rho_i$ for $i=0,1,2$. 
Thanks to \cite[Corollary 3.2]{U11}, we have $\trho(\tilde{M})'\cap \tilde{M}=\C$. 
 
\begin{theorem} With the above notation, we have $\talpha_g\circ \trho=\trho\circ \talpha_{-g}$ and 
$$\trho^2(x)=SxS^*+\sum_{g\in G}T_g\talpha_g\circ \trho(x)T_g^*,\quad \forall x\in \tilde{M}.$$
\end{theorem}

\begin{proof} It is easy to show the first relation on $\tilde{M}$, and  the second relation on $M_0$ and $M_1$. 
For $x\in \ell^\infty(G)$, 
\begin{align*}
\lefteqn{\trho^2(\iota_2(x))=\trho(S\iota_1(x)S^*+\sum_{g\in G}T_g\iota_2(\theta_{-g}(x))T_g^*)} \\
 &=\rho(S)(\iota_2(x))\rho(S)^*+\sum_{g\in G}
 \rho(T_g)(S\iota_1(\theta_{-g}(x))S^*+\sum_{h\in G}T_h\iota_2(\theta_{-g-h}(x))T_h^*)\rho(T_g)^* \\
 &=\rho(S)(\iota_2(x))\rho(S)^*+
 \sum_{g\in G}\rho(T_g)S\iota_1(\theta_{-g}(x))S^*\rho(T_g^*)\\
 &+\sum_{g,h\in G}\rho(T_g)T_h\iota_2(\theta_{-g-h}(x))T_h^*\rho(T_g^*). 
\end{align*}
The first term is 
\begin{align*}
\lefteqn{\frac{1}{d^2}S\iota_2(x)S^*+\frac{1}{d\sqrt{d}}\sum_{g\in G}S\iota_2(x)T_g^*T_g^*} \\
 &+\frac{1}{d\sqrt{d}}\sum_{g\in G}T_gT_g\iota_2(x)S^*+\frac{1}{d^2}\sum_{g,h}T_gT_g\iota_2(x)T_h^*T_h^*.
\end{align*}
The second term is 
$$\sum_{g\in G}\alpha_g\rho(T_g)S\iota_1(\theta_g(x))S^*\alpha_g\rho(T_g^*)=
\sum_{g\in G}T_gS\iota_1(\theta_g(x))S^*T_g^*.$$
The third term is 
\begin{align*}
\lefteqn{\sum_{g,k}\alpha_g\rho(T_g)T_{g+k}\iota_2(\theta_{-k})T_{g+k}^*\alpha_g\rho(T_g^*)} \\
 &=\sum_{g,k}(\frac{\eta_g^{-1}\delta_{k,0}}{\sqrt{d}}S+\sum_{h\in H}A_g(h,k)T_{g+h}T_{g+h+k})\iota_2(\theta_{-k}(x)) \\
 &\times(\frac{\eta_g\delta_{k,0}}{\sqrt{d}}S^*+\sum_{l\in G}\overline{A_g(l,k)}T_{g+l+k}^*T_{g+l}^*) \\
 &=\frac{n}{d}S\iota_2(x)S^*+\frac{1}{\sqrt{d}}\sum_{g,h\in G}\eta_gA_g(h,0)T_{g+h}T_{g+h}\iota_2(x)S^*\\
 &+\frac{1}{\sqrt{d}}\sum_{g,l\in G}\overline{\eta_gA_g(l,0)}S\iota_2(x)T_{g+l}^*T_{g+l}^*\\
 &+\sum_{g,h,k,l}A_g(h,k)\overline{A_g(l,k)}T_{g+h}T_{g+h+k}\iota_2(\theta_{-k}(x))T_{g+l+k}^*T_{g+l}^*\\
 &=\frac{n}{d}S\iota_2(x)S^*+\frac{1}{\sqrt{d}}\sum_{g,h\in G}\eta_gA_g(h-g,0)T_hT_h\iota_2(x)S^*\\
 &+\frac{1}{\sqrt{d}}\sum_{g,l\in G}\overline{\eta_gA_g(l-g,0)}S\iota_2(x)T_l^*T_l^*\\
 &+\sum_{g,h,k,l}A_g(h-g,k)\overline{A_g(l-g,k)}T_hT_{h+k}\iota_2(\theta_{-k}(x))T_{l+k}^*T_l^*\\
\end{align*}
Thanks to Eq.(\ref{hkshift}) and Eq.(\ref{2hshift}), we have 
$$\sum_{g\in G}\eta_gA_g(g-h,0)=\eta_h\sum_{g\in G}A_h(h-g,0)=-\frac{1}{d},$$
\begin{align*}
\lefteqn{\sum_{g\in G}A_g(h-g,k)\overline{A_g(l-g,k)}} \\
 &=\sum_{g\in G}A_h(h-g,k)\overline{A_l(l-g,k)}
\eta_l\eta_{l+k}\overline{\eta_h\eta_{h+k}}\epsilon_h(-g)\epsilon_h(k-g)\epsilon_l(-g)\epsilon_l(k-g) \\
 &=\sum_{g\in G}A_{-h}(h-g,k)\overline{A_{-l}(l-g,k)}
\eta_l\eta_{l+k}\overline{\eta_h\eta_{h+k}}\epsilon_h(-h)\epsilon_h(k-h)\epsilon_l(-l)\epsilon_l(k-l) \\
 &=(\delta_{h,l}-\frac{\eta_{-l}\eta_{-h}^{-1}}{d})
 \eta_l\eta_{l+k}\overline{\eta_h\eta_{h+k}}\epsilon_h(-h)\epsilon_h(k-h)\epsilon_l(-l)\epsilon_l(k-l) \\
 &=\delta_{h,l}-\frac{1}{d}\delta_{k,0}.
\end{align*}
Therefore the third term is equal to 
\begin{align*}
\lefteqn{\frac{n}{d}S\iota_2(x)S^*-\frac{1}{d\sqrt{d}}\sum_{h\in G}T_hT_h\iota_2(x)S^*
-\frac{1}{d\sqrt{d}}\sum_{l\in G}S\iota_2(x)T_l^*T_l^*} \\
 &+ \sum_{h,k\in G}T_hT_{h+k}\iota_2(\theta_{-k}(x))T_{h+k}^*T_h^*
 -\frac{1}{d}\sum_{h,l\in G}T_hT_h\iota_2(x)T_l^*T_l^*,
 \end{align*}
and 
\begin{align*}
\lefteqn{\trho^2(\iota_2(x))} \\
 &=S\iota_2(x)S^*+\sum_{h\in G}T_hS\iota_1(\theta_h(x))S^*T_h^*+\sum_{h,k\in G}T_hT_{h+k}\iota_2(\theta_{-k}(x))T_{h+k}^*T_h^*\\
 &=S\iota_2(x)S^*+\sum_{h\in G}T_h\talpha_h\trho(\iota_2(x))T_h^*. 
\end{align*}
\end{proof}

\end{document}